\documentclass{amsart}

\usepackage[utf8]{inputenc}
\usepackage{amsfonts}
\usepackage{amsmath}
\usepackage{amsthm}
\usepackage{amssymb}
\usepackage{latexsym}
\usepackage{enumerate}

\newtheorem{theorem}{Theorem}[section]
\newtheorem{lemma}[theorem]{Lemma}
\newtheorem{proposition}[theorem]{Proposition}
\newtheorem{corollary}[theorem]{Corollary}

\newtheorem{theorem?}[theorem]{Theorem?}
\newtheorem{problem}[theorem]{Problem}
\newtheorem{example}[theorem]{Example}
\newtheorem{definition}[theorem]{Definition}

\numberwithin{equation}{section}

\begin{document}

\newcommand{\cc}{\mathfrak{c}}
\newcommand{\N}{\mathbb{N}}
\newcommand{\PP}{\mathbb{P}}
\newcommand{\QQ}{\mathbb{Q}}
\newcommand{\forces}{\Vdash}
\newcommand{\R}{\mathbb{R}}
\newcommand{\LL}{\mathbb{L}}
\newcommand{\MM}{\mathbb{M}}
\newcommand{\dual}[1]{#1^\star}
\newcommand{\ddual}[1]{#1^{\star\star}}
\newcommand{\str}[1]{\hat{#1}}
\newcommand{\ud}{\,\mathrm{d}}
\newcommand{\wstar}[1]{#1^*}

\title{On automorphisms of the Banach space $\ell_\infty/c_0$}
\author{Piotr Koszmider}
\address{Institute of Mathematics, Polish Academy of Sciences,
ul. \'Sniadeckich 8,  00-656 Warszawa, Poland}
\email{\texttt{piotr.koszmider@impan.pl}}
\thanks{The first author was partially supported by the National Science Center research grant 
2011/01/B/ST1/00657. } 

\author{Crist\'obal Rodr\'\i guez-Porras}
\address{Departamento de Matem\'aticas, Facultad de Ciencias, Universidad de Los
 Andes. 5101 M\'erida, Venezuela}
\address{Equipe de Logique, UFR de Math\'ematiques, Universit\'e Denis Diderot Paris
7. 75013 Paris, France}
\email{\texttt{cristobalr@ula.ve}}
\thanks{The second author was partially supported by  the
ESF-INFTY grant number 4473} 

\begin{abstract}
We investigate Banach space automorphisms 
$T:\ell_\infty/c_0\rightarrow\ell_\infty/c_0 $ focusing on the possibility of representing  their fragments
of the form $$T_{B,A}:\ell_\infty(A)/c_0(A)\rightarrow \ell_\infty(B)/c_0(B)$$ for $A, B\subseteq \N$ infinite
by means of linear operators from $\ell_\infty(A)$ into $\ell_\infty(B)$,  infinite  $A\times B$-matrices,
continuous maps from $B^*=\beta B\setminus B$ into $A^*$, or 
bijections from $B$ to $A$.
This leads to the analysis of general linear operators on $\ell_\infty/c_0$.
We present many examples, introduce and investigate several classes of operators,
for some of them we obtain satisfactory representations and for other 
give examples showing that it is impossible. 
In particular, we show that there are automorphisms of  $\ell_\infty/c_0$
which cannot be lifted to operators on $\ell_\infty$ and assuming OCA+MA we show that every automorphism of 
$\ell_\infty/c_0$ with no fountains or with no funnels is locally, i.e.,
for some infinite $A, B\subseteq \N$ as above,  induced by a bijection from $B$ to $A$.
This additional set-theoretic assumption is necessary as we show that the continuum hypothesis 
implies the existence of counterexamples of diverse flavours. However, many 
basic  problems, some of which are listed in the last section, remain open.
\end{abstract}

\maketitle
{}

\markright{}

\section{Introduction}
The set-theoretic analysis of the Boolean algebra $\wp(\N)/\text{\emph{Fin}}$ of sets of the integers
modulo finite sets, being far from concluded,
has been quite successful. Some of the impressive results  refer to the structure of automorphisms
of the algebra which under CH can be quite pathologically complicated, 
as first observed by W. Rudin (Theorem 4.7 of \cite{rudin})  but consistently may all be
trivial, that is induced by an almost permutation of $\N$ (originally proved  by S. Shelah in \cite{shelahproper}).  
During the development of the theory, the Proper Forcing Axiom (\cite{shelahproper})
and the Open Coloring Axiom (introduced in \cite{stevopartitions}
by S. Todorcevic)
 have been established as  tools not only implying
the triviality of all automorphisms but also excluding other pathological mappings, for example embeddings of 
a big class of quotients into $\wp(\N)/\text{\emph{Fin}}$  (\cite{boban, stevoquotients, ilijasauto}).

These results have also a profound impact on more complex mathematical structures. For example,
they directly imply that the question whether the only automorphisms of the Banach algebra $\ell_\infty/c_0$
are those induced by almost permutations of $\N$ is undecidable. Indirectly they recently  inspired
the research in $C^*$-algebras resulting in the undecidability of
the structure of the automorphisms of the Calkin algebra of operators on the Hilbert space modulo
the compact operators (\cite{ilijas, weaver}).

The main focus of this paper is another  natural question, namely, what is the impact of the combinatorics of $\wp(\N)/\text{\emph{Fin}}$ on
the automorphisms of $\ell_\infty/c_0$ considered as a Banach 
space\footnote{Recall that the Banach space  $\ell_\infty/c_0$ is canonically isometric to
the Banach space $C(\N^*)$ of all continuous functions on the Stone space  $\N^*=\beta\N\setminus \N$ of
the Boolean algebra $\wp(\N)/\text{\emph{Fin}}$. Hence, the link between $\ell_\infty/c_0$
and $\wp(\N)/\text{\emph{Fin}}$ is canonical, however there are many more linear operators
on $\ell_\infty/c_0$ than those induced by homomorphisms of the Boolean algebra
$\wp(\N)/\text{\emph{Fin}}$.}, in particular if the Open Coloring Axiom (OCA) or the Proper Forcing Axiom (PFA) can be successfully
used in this context. At the moment the situation seems similar to that of 
the early stage of the research on $\wp(\N)/\text{\emph{Fin}}$ and $\N^*$: the usual axioms 
seem too weak to resolve many basic questions about the Banach space $\ell_\infty/c_0$
(\cite{christina, univ, stevoembeddings, krupski}),
the continuum hypothesis provides some answers leaving a chaotic picture full of pathological objects
obtained using transfinite induction (\cite{drewnowski, castilloplichko}) 
and there is a hope (based, for example, on e.g., \cite{dowpfa}) that alternative axioms like OCA, OCA+MA, PFA etc.,
would provide an elegant structural  theory of automorphisms of $\ell_\infty/c_0$. This hope
is not only based on the case of $\wp(\N)/\text{\emph{Fin}}$  but some other cases as well 
(\cite{stevobasis, justin}). 

In order to explain our results we need to introduce some background and terminology.
In the case of the Boolean algebra $\wp(\N)/\text{\emph{Fin}}$ and its automorphism $h$
the following conditions are equivalent for every two cofinite sets $A, B\subseteq \N$:
\begin{itemize}
\item There is an isomorphism $H: \wp(A)\rightarrow\wp(B)$ such that $[H(C)]_{\text{\emph{Fin}}}=h([C]_{\text{\emph{Fin}}})$
for all $C\subseteq A$ ($h$ lifts to $\wp(\N)$);
\item There is an isomorphism $G: \text{\emph{FinCofin}}(A)\rightarrow \text{\emph{FinCofin}}(B)$
such that $[\bigcup\{G(n): n\in C\}]_{\text{\emph{Fin}}}=h([C]_{\text{\emph{Fin}}})$
for all $C\subseteq A$  ($h$ is induced by an almost automorphism of $\text{\emph{FinCofin}}(\N)$);
\item There is a bijection $\sigma: B\rightarrow A$ 
 such that
 $[\{n\in B: \sigma(n)\in C\}]_{\text{\emph{Fin}}}=h([C]_{\text{\emph{Fin}}})$
for all $C\subseteq  A$ ($h$ is trivial).
\end{itemize}
Another special feature of liftings of automorphisms on $\wp(\N)/\text{\emph{Fin}}$, i.e., homomorphisms of $\wp(\N)$ satisfying the properties above is that
\begin{itemize}
\item Every isomorphism from $\wp(A)$ and $\wp(B)$ for $A, B\subseteq \N$ infinite
is continuous with respect to the product topologies on $\{0,1\}^A$ and $\{0,1\}^B$.
\end{itemize}
Moreover,  if we identify points of $\N^*$ with ultrafilters
 in $\wp(\N)/\text{\emph{Fin}}$, the Stone duality gives that:
\begin{itemize}
\item for every homomorphism $h$ of $\wp(\N)/\text{\emph{Fin}}$ there is a continuous map $\psi:\N^*\rightarrow \N^*$ 
such that   
$$\chi_{h([A]_{\text{\emph{Fin}}})^*}=\chi_{A^*}\circ\psi$$
for every $A\subseteq \N$.
\end{itemize}
The corresponding notions for operators on $\ell_\infty/c_0$ are summarized in the
following: 
\begin{definition}\label{trivializations} If $T: \ell_\infty/c_0\rightarrow \ell_\infty/c_0$ is a
linear bounded operator,   $A, B\subseteq\N$ cofinite, then we say that
\begin{enumerate}
\item $T$ is liftable (can be lifted) if, and only if, there is a linear bounded
$S:\ell_\infty(A)\rightarrow \ell_\infty(B)$ such that for all $f\in \ell_\infty$ we have
$$T([f]_{c_0})= [S(f)]_{c_0}$$
 
\item $T$ is a matrix operator if, and only if, there is an operator
$S:c_0(A)\rightarrow c_0(B)$ given by a real  matrix $(b_{ij})_{i\in B, j\in A}$ such that
for all $f\in \ell_\infty(A)$ we have
$$T([f]_{c_0})=[(\sum_{j\in A} b_{ij}f(j))_{i\in B}]_{c_0}.$$

\item $T$ is a trivial operator if, and only if, there is a  {\bf nonzero} real $r\in \R$,
and a bijection $\sigma: B\rightarrow A$
 such that for all
$f\in \ell_\infty(A)$ we have
$$T([f]_{c_0})=[rf\circ\sigma]_{c_0}.$$

\item $T$ is  canonizable\footnote{
It would be reasonable to consider here also the possibility of having  for all 
$f^*\in C(\N^*)$  the condition $T(f^*)=gf^*\circ \psi$, for some continuous nonzero $g\in C(\N^*)$.
However, in the context of $\N^*$ all continuous functions are ``locally constant" (\ref{constantfunction})
so in the context of this paper there is no sense of introducing such a property.} 
 along $\psi:\N^*\rightarrow\N^*$
 if, and only if, $\psi$ is a {\bf surjective} continuous mapping and there is a  
{\bf nonzero} real $r$ such that for all $f^*\in C(\N^*)$ we have 
$$\str{T}(f^*)=rf^*\circ \psi.$$
\end{enumerate}
In the case of liftable and matrix operators we will be using more complex phrases like
automorphic liftable operator, embedding matrix operator etc., meaning that the
operator is liftable or matrix respectively and it has the additional property.
\end{definition}
In contrast with the case of $\wp(\N)/\text{\emph{Fin}}$ our
 results show that the relationships among these notions are far from equivalences: 
{\tiny
\begin{table}[h]
\begin{tabular}{|c|l|c|l|c|l|c|l|l|l}
\cline{1-1} \cline{3-3} \cline{5-5} \cline{7-7} \cline{9-9}
\begin{tabular}[c]{@{}c@{}} trivial \\ automorphism\end{tabular} &
 $\stackrel{\Rightarrow}{\nLeftarrow}$ & \begin{tabular}[c]{@{}c@{}}
automorphic\\  
matrix \\ operator\end{tabular} & $\Leftrightarrow$ & \begin{tabular}[c]{@{}c@{}}
Automorphic\\  liftable\\ operator with \\ a lifting continuous \\ on $B_{\ell_\infty}$ \end{tabular} 
& $\stackrel{\Rightarrow}{\nLeftarrow}$ & \begin{tabular}[c]{@{}c@{}}Automorphic \\ liftable\\ operator\end{tabular}
 & $\stackrel{\Rightarrow}{\nLeftarrow}$ & Automorphism & 
 \\ \cline{1-1} \cline{3-3} \cline{5-5} \cline{7-7} \cline{9-9}
\end{tabular}
\end{table}
}
None of the implications or counterexamples to the reverse implications require additional
set-theoretic axioms. The nontrivial parts of the above chart  are the following  
facts:
\begin{itemize}
\item There are automorphisms of $\ell_\infty/c_0$ which are not liftable to a linear operator 
on $\ell_\infty$ (\ref{nonliftable});
\item There are automorphisms of $\ell_\infty/c_0$ which are liftable  but they are not
matrix operators and none of their liftings are  continuous on $B_{\ell_\infty}$ in the product topology
(\ref{discontinuouslifting}).
\item Automorphisms of $\ell_\infty/c_0$ which have liftings to $\ell_\infty$ continuous
in the product topology are exactly the automorphic matrix operators  (\ref{summary_matrices}).
\end{itemize}
Note that the question of canonizing globally  all automorphisms other than trivial is outright excluded by
the clear fact that there are many matrices of isomorphisms on $c_0$  which are not matrices of
almost  permutations modulo $c_0$.
%(\ref{matrixnotpermutation}).

In the light of the above absolute results and the exclusion of the possibility of
a global canonization or matricization we will concentrate on ``local" versions of the above properties
 of the operators in the sense that they hold in some sense for  copies of $\ell_\infty/c_0$
of the form $\ell_\infty(A)/c_0(A)$ for an infinite $A\subseteq \N$. 
Since the above properties depend on the link between $\ell_\infty/c_0$ and $\N^*$ or $\N$
we choose the 
approach of Drewnowski and Roberts from \cite{drewnowski} which has functional analytic motivations and
applications:
\begin{definition}
Suppose $A\subseteq \N$ is infinite. We define
$P_A: \ell_\infty/c_0\rightarrow \ell_\infty(A)/c_0(A)$ and
$I_A: \ell_\infty(A)/c_0(A)\rightarrow \ell_\infty/c_0$
 by
$$P_A([f]_{c_0})=[f|A]_{c_0(A)}, \ \ \  I_A([g]_{c_0(A)})=[g\cup 0_{\N\setminus A}]_{c_0}$$
for all $f\in \ell_\infty$ and all $g\in \ell_\infty(A)$.
Suppose that $T: \ell_\infty/c_0\rightarrow \ell_\infty/c_0$ is a
linear bounded operator and $A, B\subseteq \N$ two infinite sets. The localization of
$T$ to $(A, B)$ is the operator $T_{B, A}: \ell_\infty(A)/c_0(A)\rightarrow \ell_\infty(B)/c_0(B)$ given
by $$ T_{B, A}=P_B\circ T\circ I_A.$$
\end{definition}
It was proved by Drewnowski and Roberts in \cite{drewnowski} that for every operator
$T:\ell_\infty/c_0\rightarrow \ell_\infty/c_0$ and every infinite $A\subseteq \N$ there is
an infinite $A_1\subseteq A$ such that for all $[f]_{c_0}\in \ell_\infty(A_1)/c_0(A_1)$ we have 
$T_{A_1,A_1}([f]_{c_0})=[rf]_{c_0}$ for some real $r\in \R$. However this does not exclude the possibility of 
$T_{A_1, A_1}=0$, which actually
is quite common. Thus the focus of this paper is to obtain localizations which
are isomorphic embeddings or isomorphisms
and the ultimate goal (not completely achieved) is to localize somewhere   any automorphism
to a canonical operator along a homeomorphism (which turned out
to be impossible in ZFC by \ref{nowherecanonizable}) and to a trivial automorphism under
OCA+MA. However if one wants to iterate the use of several localization results (like
in the case of \cite{drewnowski}) it is
useful to have right-local or left-local results and not just somewhere local results:
\begin{definition} Suppose that $T: \ell_\infty/c_0\rightarrow \ell_\infty/c_0$ is a
linear bounded operator. Let $\PP$ be one of the properties `` liftable'', `` matrix operator'',
`` trivial'', ``canonizable''. 
\begin{enumerate}
\item We say that $T$ is somewhere $\PP$ if, and only if, there are infinite $A\subseteq \N$
 and $B\subseteq \N$ such that $T_{B, A}$ has $\PP$.
\item We say that $T$ is right-locally $\PP$ if, and only if, for every infinite $A\subseteq \N$
there are infinite $A_1\subseteq A$ and $B\subseteq \N$ such that $T_{B, A_1}$ has $\PP$.
\item We say that $T$ is left-locally $\PP$ if, and only if, for every infinite $B\subseteq \N$
there are infinite $B_1\subseteq B$ and $A\subseteq \N$ such that $T_{B_1, A}$ has $\PP$.
\end{enumerate}
\end{definition}
To hope for isomorphic left-local properties one needs to assume that the image of $T$ is
big, for example that $T$ is surjective. Similarly, for nontrivial right-local properties we need 
to assume that the kernel is small, for example that $T$ is injective. 
In contrast to the global versions, the local versions of the notions from Definition \ref{trivializations}
behave like the Boolean counterparts:
 \begin{proposition} Suppose that $T: \ell_\infty/c_0\rightarrow \ell_\infty/c_0$ is an
automorphism. Then the following are equivalent
\begin{enumerate}
\item $T$ is somewhere a liftable isomorphism,
\item  $T$ is somewhere   an isomorphic matrix operator,
\item $T$ is somewhere a liftable isomorphism with a lifting which is continuous in the product topology,
\item $T$ is somewhere trivial.
\end{enumerate}
\end{proposition}
\begin{proof} The implication from (1) to (2) follows from \ref{liftableembedding->somewherematrixiso};
the equivalence of (2) and (3) is \ref{summary_matrices}; the implication
from (2) to (4) follows from \ref{canonembedding}; the fact that (4)
implies (1) is clear.
\end{proof}
In fact the above equivalences hold (with the same proof) in the case of   $T$
being an isomorphic embedding\footnote{By isomorphic embedding we mean an operator
which is an isomorphism onto its closed range. Sometimes these operators are called bounded below.}
and for right-localizations which are isomorphic embeddings.
However, a surjective operator can be globally liftable but nowhere a matrix operator
(\ref{liftablenonmatrix}) or can be globally a matrix operator but nowhere trivial (\ref{matrixnowheretrivial}).
Another reason why the above local notions make sense is the following:
\begin{proposition}
 Suppose that $T: \ell_\infty/c_0\rightarrow \ell_\infty/c_0$ is a
linear bounded operator and $A, B\subseteq \N$ are two infinite sets. 
Suppose that $T_{B,A}$ is canonical  along a homeomorphism.
Then, $T$ fixes a complemented copy of $\ell_\infty/c_0$ whose
image under $T$ is complemented in $\ell_\infty/c_0$.
\end{proposition}
\begin{proof} 
See the proof of Corollary 2.4 of \cite{drewnowski}.
\end{proof}
In fact, the above proposition would also be true with the same proof if we weakened the
hypothesis on $B$ from clopen to closed subset of $\N^*$ homeomorphic to $\N^*$.
But to make sure that $A$ induces a subspace not just a quotient which is to be fixed 
we must insist on $A^*$ to be clopen.
This approach in the context of other spaces $C(K)$ is quite fruitful for obtaining
complemented copies of the entire $C(K)$ inside any isomorphic copy of the $C(K)$
(for example, for $C(K)$ with $K$ metrizable see \cite{pelczynskiC(S)}, for
$\ell_\infty$ see \cite{haydon}, and for $C([0,\omega_1])$ see \cite{niels}), see
problems in Section \ref{sec:problems}. 

One should note, however, that the notion of e. g., somewhere trivial automorphism
on $\ell_\infty/c_0$ has quite a different character than being somewhere trivial
automorphism of $\wp(\N)/Fin$, this is because the images of subspaces of the form
$\{[f]\in \ell_\infty/c_0: f|A=0\}$ 
for $A\subseteq \N$ are usually not of the form $\{[f]\in \ell_\infty/c_0: f|B=0\}$
for $B\subseteq N$, even if $T_{B,A}$ is trivial. Also, trivialization
or canonization of $T_{B,A}$ does not yield any information about $T^{-1}_{A,B}$
as in the case of automorphisms of $\wp(\N)/Fin$.

Having proven the equivalence of the local versions of the above notions one is left
with deciding if automorphisms of $\ell_\infty/c_0$ are somewhere
canonizable along homeomorphisms. If this happens their local structure is similar to
that of homeomorphisms of $\N^*$ i.e., assuming OCA+MA they would be trivial
and, for example, under CH not. 

Canonization of automorphisms  $T: \ell_\infty/c_0\rightarrow \ell_\infty/c_0$
(or corresponding $\str{T}:C(\N^*)\rightarrow C(\N^*)$) encounters, however, problems
at least as difficult as understanding continuous maps defined on closed subsets of $\N^*$ with ranges in
$\N^*$ (not only automorphisms of $\N^*$). To better understand why this is so, 
let us recall that linear bounded operators
on $C(\N^*)$ can be represented as weakly$^*$ continuous mappings
$\tau: \N^*\rightarrow M(\N^*)$ (see Theorem 1 in VI.7 of \cite{dunford}),
 where $M(\N^*)$ denotes the Banach space of all Radon measures on $\N^*$ 
with the total variation norm identified by
the Riesz representation theorem with the dual to $C(\N^*)$ with the weak$^*$ topology
(see \cite{semadeni}). Often the points
of $\N^*$ (identified with the Dirac measures) are sent by this map to measures that do not have atoms, and if they have atoms they may have many
of them giving rise to partial multivalued functions into $\N^*$. One obtains $\tau(x)$ as
$\dual{T}(\delta_x)$ for each $x\in \N^*$ and the representation is given by
$$\str{T}(f^*)(x)=\int f^*\,\mathrm{d}\tau(x)$$
for every $f^*\in C(\N^*)$. The multifunctions, possibly of empty values, are given by
$$\varphi_\varepsilon^T(y)=\{x\in \N^*: |\dual{T}(\delta_y)(\{x\})|\geq\varepsilon\}$$
for any $\varepsilon>0$ or by $\varphi^T(y)=\bigcup_{\varepsilon>0}\varphi^T_\varepsilon.$
An equivalent condition for $T$ being somewhere canonizable along a homeomorphism is the
existence of infinite $A, B\subseteq \N$ and a homeomorphism $\psi:B^*\rightarrow A^*$
such that 
$$\dual{T}(\delta_y)|A^*=r\delta_{\psi(y)}$$
for some nonzero $r\in \R$, which in particular means that $\varphi^T(y)\cap A^*=\{\psi(y)\}$,
or in other words  that $\psi$ is a homeomorphic selection
from $\varphi^T$. Right up front there could be two basic obstacles for the existence of 
such a selection,
namely $\bigcup_{y\in B^*}\varphi^T(y)$ could have empty interior 
or 
$\{y\in B^*: \varphi^T(y)\not=\emptyset\}$ could have empty interior  for an infinite $B\subseteq \N^*$.
We call these obstacles (in stronger versions including nonatomic measures) fountains and funnels
respectively
and introduce two classes of operators (fountainless operators, Definition \ref{locallydetermined},
 and  funnelless operators, Definition \ref{funnelless})  for which by 
definition the above obstacles cannot arise, respectively,
and we obtain some reasonable sufficient conditions for the canonization:
\begin{itemize}
\item Every automorphism  on $\ell_\infty/c_0$ which is fountainless is left-locally canonizable along a 
quasi-open mapping (\ref{surcanonization});
\item Every automorphism on $\ell_\infty/c_0$ which is funnelless is right-locally canonizable along
a quasi-open mapping (\ref{injcanonization});
\end{itemize}
where quasi-open means that the image of every open set has nonempty interior (\ref{quasi-open}).
The second result is in fact a  consequence of a study by G. Plebanek \cite{plebanekisrael},
however the proof of the first takes a considerable part of this paper.  
The possibility of obtaining these results is based on special properties of isomorphic embeddings
and surjections. One ingredient is
an improvement of a theorem of Cengiz (``P" in \cite{cengiz}) obtained by Plebanek
 (Theorem 3.3. in \cite{plebanekisrael}) which guarantees that the range of $\varphi^T_{\|T\|\|T^{-1}\|}$
covers $\N^*$ if $T$ is an isomorphic embedding. However, in this result the set of $y$'s where $\varphi^T(y)$
is nonempty could be nowhere dense, so we exclude this possibility by assuming that $T$ has no funnels. 
On the other hand we prove that if $T$ is surjective, then either for each $y$ the set $\varphi^T(y)$
is nonempty or else there is a infinite $A\subseteq \N^*$ such that $\bigcup\{\varphi^T(y): y\in A^*\}$
is nowhere dense, the second possibility being excluded if $T$ has no fountains.

Then one is still left with the problem of reducing a quasi-open  map to
a homeomorphism between two clopen sets. The results of I. Farah \cite{ilijasauto}
 allow us to conclude that OCA+MA implies that a quasi-open mapping defined on a clopen subset of $\N^*$
and being onto a clopen subset of $\N^*$ is 
somewhere a homeomorphism and so by results of Velickovic \cite{boban}  it is somewhere induced by a 
bijection between two infinite subsets  of $\N$. Hence we obtain:
\begin{itemize}
\item (OCA+MA) Every  fountainless automorphism
of $\ell_\infty/c_0$ is left-locally trivial (\ref{ocatrivialization});
\item (OCA+MA) Every funnelless automorphism
of $\ell_\infty/c_0$ is right-locally trivial (\ref{ocatrivialization}).
\end{itemize}
The continuum hypothesis shows that the above results are optimal in
many directions. First, an obstacle to
improving our above-mentioned  ZFC selection results (\ref{surcanonization}, \ref{injcanonization})
 by replacing quasi-open to a homeomorphism between
clopen sets is the following example:
\begin{itemize}
 \item (CH) There is a fountainless and funnelless everywhere present isomorphic embedding
globally canonizable along 
quasi-open map which
is nowhere canonizable along a homeomorphism (\ref{quasi-opench}).
\end{itemize}
Here everywhere present is a weak version of a surjective operator ($P_A\circ T\not=0$
for any infinite $A\subseteq \N$ see \ref{everywhere-non0}). Automorphisms  $T$ have the property that
$P_A\circ T$ is everywhere present and $T\circ I_A$ is an isomorphic embedding
for any infinite $A\subseteq \N$.
Moreover we have the following:
\begin{itemize}
\item (CH) There is an automorphism of $\ell_\infty/c_0$ which is nowhere canonizable
along a quasi-open map, in particular along a  homeomorphism (\ref{nowherecanonizable}).
\end{itemize}
The above example is not a direct construction, but we have more concrete and slightly weaker
examples (\ref{notlocalstrongcanon})
based on the existence in $\N^*$ of nowhere dense $P$-sets which are retracts
of $\N^*$, due to van Douwen and van Mill (\cite{vandouwen}).
It is not excluded by our results (see Section \ref{sec:problems}) that consistently all isomorphic embeddings on $\ell_\infty/c_0$ are funnelless,
however there are ZFC surjective operators which are not fountainless (\ref{locallynullsurjective}). 
And, of course, assuming CH  there are
well familiar nowhere trivial homeomorphisms of $\N^*$ which provide examples of globally
canonizable operator which is nowhere liftable (\ref{nowheretrivialch}).

The are many basic  problems concerning the automorphisms of $\ell_\infty/c_0$ left
open, some of them are listed in Section \ref{sec:problems}. A breakthrough in developing
the methods of direct applications of PFA in the space $\ell_\infty/c_0$ which was recently
obtained by A. Dow in \cite{dowpfa} may be especially useful in attacking these problems.
The structure of the paper is as follows:
\tableofcontents
Notation and conventions:
\begin{itemize}
\item\emph{Fin} - Ideal of finite subsets of $\mathbb{N}$

\item\emph{FinCofin} - The Boolean algebra of finite and cofinite subsets of $\N$

\item$[A]=[A]_{\text{\emph{Fin}}}$ - The equivalence class of $A$ with respect to \emph{Fin}

\item$A=_*B$ - $A\triangle B\in$\emph{Fin}

\item$A\subseteq_* B$ - $A\setminus B\in$\emph{Fin}

\item$\beta\N$ - The \v Cech-Stone compactification of the integers and the Stone space of $\wp(\N)$

\item$\N^*$ - The \v Cech-Stone remainder $\N^*=\beta\N\setminus\N$ and
the Stone space of $\wp(\N)/\text{\emph{Fin}}$

\item$\beta A$ - The clopen set in $\beta\N$ defined by $\{x\in \beta\N:A\in x\}$

\item$A^*$ - The clopen set in $\N^*$ defined by $\beta A\setminus A$

\item$\beta f$ - The element of $C(\beta\N)$ which extends $f\in\ell_\infty$

\item$f^*$ - The element of $C(\N^*)$ obtained by restricting $\beta f$
to $\N^*$, for some $f\in\ell_\infty$

\item$[f]=[f]_{c_0}$ - The equivalence class of $f\in\ell_\infty$ with respect to $c_0$
\end{itemize}
Since any element of $C(\N^*)$ or $C(\beta\N)$ is of the form $f^*$ or $\beta f$,
for some $f\in \ell_\infty$ respectively,
we may use this convention when talking about general elements of these spaces.
However, not all continuous functions on $\N^*$ or linear operators on $C(\N^*)$
are induced by corresponding objects in $\N$ or $\ell_\infty$. So for 
the passage from an endomorphism $h$ of $\wp(\N)/\text{\emph{Fin}}$ to a continuous
self-mapping on $\N^*$ or from a linear operator $T$
on $\ell_\infty/c_0$ to a linear operator on $C(\N^*)$ we will use $\str{h}$ and $\str{T}$,
respectively.

\begin{itemize}

\item $[T]$ - The operator on $\ell_\infty/c_0$ induced by
an operator $T:\ell_\infty\rightarrow \ell_\infty$ which preserves $c_0$ (i.e., $T[c_0]\subseteq c_0$)
that is $[T]([f]_{c_0})=[T(f)]_{c_0}$ for any $f\in\ell_\infty$

\item $\beta T$ - The operator on $C(\beta\N)$ induced by
an operator $T:\ell_\infty\rightarrow \ell_\infty$ which preserves $c_0$ (i.e., $T[c_0]\subseteq c_0$)
that is $\beta T(\beta f)=\beta(T(f))$ for any $f\in\ell_\infty$

\item $\wstar{T}$  - The operator on $C(\N^*)$ induced by
an operator $T:\ell_\infty\rightarrow \ell_\infty$ which preserves $c_0$ (i.e., $T[c_0]\subseteq c_0$)
that is $\wstar{T}(f^*)=(T(f))^*$ for any $f\in\ell_\infty$

\item $\str{T}$ - The operator from $C(\N^*)$ into itself which corresponds 
to $T:\ell_\infty/c_0\rightarrow\ell_\infty/c_0$, i.e., $\str{T}(f^*)=g^*$ where $[g]=T([f])$

\item $\str{h}$ - The continuous selfmap of $\N^*$ which corresponds via the Stone duality
to an endomorphism $h$ of $\wp(\N)/\text{\emph{Fin}}$, i.e.,  $\str{h}(x)=h^{-1}[x]$ when
we identify points of $\N^*$ with the ultrafilters of $\wp(\N)/\text{\emph{Fin}}$

\item$T_{\psi}$ - The operator $T_{\psi}:C(\N^*)\rightarrow C(\N^*)$ which maps $f$ to $f\circ\psi$,
for some continuous $\psi:\N^*\rightarrow\N^*$

\end{itemize}

The remaining often used symbols are:
\begin{itemize}
\item $M(\N^*)$ - the Banach space of Radon measures on $\N^*$ with the total
variation norm, identified with the dual space to $C(\N^*)$ 
or the dual space to $\ell_\infty/c_0$ via the Riesz representation theorem

\item$\dual{T}$ - The dual or adjoint operator of $T$, i.e., $\dual{T}(\mu)(f)=\mu(T(f))$.
$\dual{T}$ acts on the spaces of Radon measures if $T$ acts on a space of continuous functions

\item$\delta_x$ - The Dirac measure concentrated on $x$

\item $\mu|F$ - the restriction of  a measure $\mu\in M(\N^*)$ to
a Borel subset $F\subseteq \N^*$, i.e.,  $\mu|F$ is an element of $M(\N^*)$ such 
that $(\mu|F)(G)=\mu(G\cap F)$ for any Borel $G\subseteq \N^*$

\item$\chi_a$ - The characteristic function of $a$

\item$B_X$ - The unit ball of the Banach space $X$

\item$\varphi_{\varepsilon}^T(y)$ - The set $\{x\in\N^*:|\dual{T}(\delta_y)(\{x\})|>\varepsilon\}$,
 where $T$ is an operator on $C(\N^*)$

\item$\varphi^T(y)$ - The set $\bigcup_{\varepsilon>0}\varphi_{\varepsilon}^T(y)$, where $T$ is an operator
on $C(\N^*)$

\end{itemize}
\section{Operators on $\ell_\infty$ preserving $c_0$}

\subsection{Operators given by $c_0$-matrices}
A linear operator $R$ on $\ell_\infty$ which preserves $c_0$ (i.e., $R[c_0]\subseteq c_0$)
defines, of course, an operator on $c_0$.
In the case of the Boolean  algebra $\wp(\N)$, any Boolean automorphism preserves $\text{\emph{FinCofin}}(\N)$ and
its restriction to $\text{\emph{FinCofin}}(\N)$  completely determines the automorphism. The analogous 
fact does not hold for linear automorphisms on $\ell_\infty$, for example there are many distinct
automorphisms of $\ell_\infty$ which do not move $c_0$ (see \ref{discontinuousauto}). 
However, the restrictions  to $c_0$ of operators on $\ell_\infty$ which preserve $c_0$ will
play an important role, and in some cases will determine a given operator.
So let us establish a transparent representation of operators on $c_0$:

\begin{proposition}\label{charac_matrix}
 $R:c_0\rightarrow c_0$ is a linear bounded operator if, and only if,
there exists an $\mathbb{N}\times\mathbb{N}$ matrix $(b_{ij})_{i,j\in\mathbb{N}}$
such that
\begin{enumerate}
\item every row is in $\ell_1$,
\item if we write $b_i=(b_{ij})_j$, then $\{\|b_i\|_{\ell_1}:i\in\mathbb{N}\}$ is a bounded set, 
\item every column is in $c_0$,
\end{enumerate}
and such that for every $f\in c_0$ we have
$$R(f)=\left(\begin{array}{lll}
                   b_{00} & b_{01} & \dots \\
			  b_{10} & b_{11}&\dots \\
			  \vdots & \vdots & \ddots  
                  \end{array}\right)
\left(\begin{array}{c}
       f(0)\\
	f(1)\\
	\vdots
      \end{array}
\right).$$

\end{proposition}
\begin{proof} Use the fact that $\dual{c_0}=\ell_1$ and put $b_i=\dual{R}(\delta_i)$, where
$\delta_i$ is the functional corresponding to the $i$-th coordinate for each $i\in \N$.
\end{proof}

This representation corresponds to representing endomorphisms of
\emph{FinCofin}$(\N)$ by finite-to-one functions from $\N$ into itself. 
Such endomorphisms induce operators on $c_0$ whose matrix satisfies
the above characterization and where every row has one entry equal to $1$ and 
the remaining entries equal to $0$.
Matrices define some operators on $\ell_\infty$ as well, of course:

\begin{proposition}\label{l_infty-matrix}
 Let $(b_{ij})_{i,j\in\mathbb{N}}$ be a matrix and let
$b_i=(b_{ij})_j$ be the $i$-th row, for every $i\in\mathbb{N}$. Then, 
$$R(f)=\left(\begin{array}{lll}
                   b_{00} & b_{01} & \dots \\
			  b_{10} & b_{11}&\dots \\
			  \vdots & \vdots & \ddots  
                  \end{array}\right)
\left(\begin{array}{c}
       f(0)\\
	f(1)\\
	\vdots
      \end{array}
\right)$$
defines an linear bounded operator $R:\ell_{\infty}\rightarrow\ell_{\infty}$
 if, and only if, $b_i\in\ell_1$, for all $i\in\mathbb{N}$, and
 $\{\|b_i\|_{\ell_1}:i\in\mathbb{N}\}$ is a bounded set.
\end{proposition}
\begin{proof} Use the fact that $\ell_1\subseteq \dual{\ell_\infty}$ and put $b_i=\dual{R}(\delta_i)$, where
$\delta_i$ is the functional corresponding to the $i$-th coordinate for each $i\in \N$.
\end{proof}

\begin{definition}
\begin{enumerate}
 \item[(i)] We say that a matrix is  a $c_0$-matrix if it satisfies conditions (1)--(3) 
of Proposition \ref{charac_matrix}.
\item[(ii)] We say that a linear bounded operator $R:\ell_{\infty}\rightarrow\ell_{\infty}$
is given by a $c_0$-matrix if there exists a $c_0$-matrix $(b_{ij})_{i,j\in\mathbb{N}}$
 such that
$$R(f)=\left(\begin{array}{lll}
                   b_{00} & b_{01} & \dots \\
			  b_{10} & b_{11}&\dots \\
			  \vdots & \vdots & \ddots  
                  \end{array}\right)
\left(\begin{array}{c}
       f(0)\\
	f(1)\\
	\vdots
      \end{array}
\right),$$
for every $f\in\ell_{\infty}$.

\end{enumerate}

\end{definition}

\begin{corollary}\label{matrixoperators} 
Suppose that $R:\ell_\infty\rightarrow\ell_\infty$
is a linear bounded operator which preserves $c_0$ and is given by

$$R(f)=\left(\begin{array}{lll}
                   b_{00} & b_{01} & \dots \\
			  b_{10} & b_{11}&\dots \\
			  \vdots & \vdots & \ddots  
                  \end{array}\right)
\left(\begin{array}{c}
       f(0)\\
	f(1)\\
	\vdots
      \end{array}
\right),$$
where $(b_{ij})_{i,j\in\mathbb{N}}$ is a real matrix. Then $(b_{ij})_{i,j\in\mathbb{N}}$
is a $c_0$-matrix.
\end{corollary}
\begin{proof} If such an operator on $\ell_\infty$ was not given by a $c_0$-matrix, 
then some of the columns of the corresponding matrix would not be in $c_0$ by \ref{l_infty-matrix} and by \ref{charac_matrix}.
Then the operator would not preserve $c_0$.
\end{proof}

%As we will see soon, not all operators on $\ell_\infty$ which preserve $c_0$
%are given by multiplication by a matrix as  above. Nevertheless, an important role
%will be played by those operators on $\ell_\infty$ which are given by a matrix and preserve $c_0$,
%that is, those operators given by a $c_0$-matrix (\ref{matrixoperators}).

\begin{proposition}\label{doubleadjoint}
If a linear  bounded operator $R:\ell_\infty\rightarrow\ell_\infty$ is
 given by a  $c_0$-matrix, then $R=\ddual{(R|c_0)}$.
\end{proposition}
\begin{proof} Appendix \ref{doubleadjoint_appendix}.
\end{proof}

%\begin{proposition}\label{matrixnotpermutation}
% There is an automorphism $R:\ell_\infty\rightarrow\ell_\infty$ which is given
%by a $c_0$-matrix but is not of the form $rf\circ \sigma$, for any real $r$
%and any almost permutation $\sigma$ of $\N$.
%\end{proposition}
%\begin{proof}
% Let $a_0,a_1\in\mathbb{R}$ be such that $a_0a_1\neq 0$ and $a_0\neq a_1$. Let
%$M=(b_{ij})_{i,j\in\N}$ be a real matrix defined by
%$$b_{(2i)(2i)}=a_0\qquad b_{(2i+1)(2i+1)}=a_1$$
%and $b_{ij}=0$ for all other $i,j\in\N$.

%It is clear that $M$ is a $c_0$-matrix and that the operator $R$ it defines is 
%an automorphism. However, if for every $f\in\ell_\infty$ we have that $R(f)$ 
%is eventually equal to $rf\circ\sigma$,
%for some $r\in\mathbb{R}$ and some $\sigma$ almost permutation of $\N$, then
%in particular $R(\chi_\N)$ is eventually constant and equal to $r$. 
%This means that $a_0=r=a_1$, a contradiction.

%\end{proof}

\subsection{Falling and weakly compact operators}
Let us recall the following characterization of weakly compact operators on
$c_0$:

\begin{theorem}\label{matrixweaklycompact}
Let $R:c_0\rightarrow c_0$ be a linear bounded operator and let $(b_{ij})_{i,j\in\N}$
be the corresponding matrix. The following are equivalent:

\begin{enumerate}%\label{c_0-matrix--weakly_compact}
 \item $R$ is weakly compact.
\item $\ddual{R}[\ell_{\infty}]\subseteq c_0$.
\item $\|b_i\|_{\ell_1}\rightarrow 0$. 

\end{enumerate}
 \end{theorem}

\begin{proof} Appendix \ref{c_0-matrix--weakly_compact_appendix}
\end{proof}

\begin{proposition}\label{rangeinc_0} 
Let $R:\ell_\infty\rightarrow\ell_\infty$ be
an operator given by a  $c_0$-matrix. Then, $R$ is weakly compact if, and only if,
 $R[\ell_\infty]\subseteq c_0$.
\end{proposition}
\begin{proof} If $R$ is weakly compact, then $R|c_0$ must be as well, and so
by \ref{matrixweaklycompact} we have $\ddual{(R|c_0)}[\ell_\infty]\subseteq c_0$
but $\ddual{(R|c_0)}=R$ by \ref{doubleadjoint}.  In the other direction,
use the fact that every operator defined on a Grothendieck Banach space into a separable
Banach space is weakly compact (Theorem 1 (v) of \cite{diestelgrothen}).
\end{proof}

\begin{definition}\label{falling} 
A $c_0$-matrix operator $R:\ell_\infty\rightarrow \ell_\infty$ is called falling
if, and only if, for every $\varepsilon>0$ there is a partition $A_0, ..., A_{k-1}$ of $\N$ such that
$$\sum_{j\in A_m}|b_{ij}|<\varepsilon$$
for all $m< k$ and $i\in \N$ sufficiently large.
\end{definition}

\begin{proposition}Every operator on $\ell_\infty$ which is given by a $c_0$-matrix and is weakly compact
is falling.
\end{proposition}
\begin{proof}
Use Theorem \ref{matrixweaklycompact}.
\end{proof}

\begin{proposition}\label{densityoperator} There is a falling, non-weakly compact operator on $\ell_\infty$
given by a $c_0$-matrix. 
\end{proposition}
\begin{proof}
Let $R:\ell_\infty\rightarrow\ell_\infty$ be given by the matrix 
$$b_{ij}=\left\{\begin{array}{ll}
                 1/(i+1), \text{ if }j\leq i\\
		  0, \text{ otherwise}
                \end{array}\right.$$ 
for all $i\in \N$. By \ref{matrixweaklycompact}
this is not a weakly compact operator. Given $k\in\N$
if we consider $A_m=\{lk+m: l\in \N\},$ for $m<k$ then,
$$\sum_{j\in A_m}|b_{ij}|\leq \left(\frac{i+1}{k}\right)\left(\frac{1}{i+1}\right)=1/k,$$ so the operator is falling.
\end{proof}

\subsection{Antimatrix operators}

The  behaviour opposite to operators given by a $c_0$-matrix
 is the subject of the following:

\begin{definition}\label{antimatrix}
 A linear bounded operator $R:\ell_{\infty}\rightarrow\ell_{\infty}$ will be called
an \emph{antimatrix} operator if, and only if, $R[c_0]=\{0\}$.
\end{definition}

Using the isometry between $\ell_\infty$ and $C(\beta\N)$, an operator $R$ on $\ell_\infty$
can be associated with 
an operator $\beta R$ on $C(\beta\N)$  and these operators can be
associated with weak$^*$ continuous functions from $\beta\N$ into the Radon measures
$M(\beta\N)$ on $\beta\N$ (see Theorem 1 in VI 7. of \cite{dunford}). Since $\N$ is dense in $\beta\N$, 
such functions are determined by their values on $\N$.
The following characterizations will be useful later on:

\begin{lemma}\label{antimatrix_charac}
 Suppose $R:\ell_{\infty}\rightarrow\ell_{\infty}$ is a bounded linear operator such that 
$R[c_0]\subseteq c_0$. Then,
\begin{enumerate}
 \item[(a)] $R$ is given by a $c_0$-matrix if, and only if, 
$\dual{R}(\delta_n)$ is concentrated on $\N$ for all $n\in\mathbb{N}$, 
that is, $\dual{R}(\delta_n)\in\ell_1,\,\forall n\in\mathbb{N}$.
\item[(b)] $R$ is an antimatrix operator if, and only if, 
$\dual{R}(\delta_n)$ is concentrated on $\N^*$ for all $n\in\mathbb{N}$.
\end{enumerate}
\end{lemma}

\begin{proof}
\textbf{(a)}\hspace{3pt} Assume $R$ is given by a $c_0$-matrix $(b_{ij})_{i,j\in\mathbb{N}}$. Then, 
for every $f\in\ell_{\infty}$ we have $\dual{R}(\delta_n)(f)=R(f)(n)=b_n(f)$,
where $b_n$ is the $n$th row of $(b_{ij})_{i,j\in\mathbb{N}}$. So $\dual{R}(\delta_n)=b_n$
and by definition of $c_0$-matrix, we have that $b_n\in\ell_1$.

Conversely, assume $\dual{R}(\delta_n)\in\ell_1$. Let $M$
be the matrix formed by putting $\dual{R}(\delta_n)$ as the $n$th row. Then,
$R$ is induced by $M$. Moreover, since $R[c_0]\subseteq c_0$,
we know that $M$ is a $c_0$-matrix.\\

\noindent\textbf{(b)}\hspace{3pt} Suppose $\dual{R}(\delta_n)$ is not concentrated on $\mathbb{N}^*$
for some $n\in\mathbb{N}$. Then, there exists an $m\in\mathbb{N}$ such that 
$\dual{R}(\delta_n)(\{m\})\neq0$. Then, 
$R(\chi_{\{m\}})(n)=\dual{R}(\delta_n)(\chi_{\{m\}})%=\int \chi_{\{m\}}\,\mathrm{d}\dual{R}(\delta_n)=\dual{R}(\delta_n)(\{m\})
\neq 0$. Therefore, $\chi_{\{m\}}\in c_0$ is a witness to the 
fact that $R[c_0]\neq\{0\}$, so $R$ is not an antimatrix operator.

Conversely, assume $\dual{R}(\delta_n)$ is concentrated on $\mathbb{N}^*$,
for every $n\in\mathbb{N}$.
Fix $f\in c_0$. Then, for every $n\in\mathbb{N}$ we have 
$R(f)(n)=\dual{R}(\delta_n)(\beta f)=\int \beta f\,\mathrm{d}\dual{R}(\delta_n)=
\int_{\mathbb{N}^*}\beta f\,\mathrm{d}\dual{R}(\delta_n)=0$, because $\beta f| \mathbb{N}^*=0$.
\end{proof}

Thus a typical example of an antimatrix operator is one
given by $R(f)=((\beta f(x_i))_{i\in \N})$ where $(x_i)_{i\in \N}$ is any sequence
of nonprincipal ultrafilters.

\begin{proposition}\label{decompositionmatrix} If $R:\ell_{\infty}\rightarrow\ell_{\infty}$ is
such that $R[c_0]\subseteq c_0$, then $R=S_0+S_1$, where $S_0$ is given by a 
$c_0$-matrix and $S_1$ is an antimatrix operator. 
\end{proposition}
\begin{proof} As $R[c_0]\subseteq c_0$ there is a matrix $(b_{ij})_{i,j\in \N}$ which
 satisfies \ref{charac_matrix}.  Define $S_0$ as multiplication by this matrix, i.e.
 $S_0=\ddual{(R| c_0)}$ by \ref{doubleadjoint}.
Now $S_1=R-S_0$ is antimatrix, so we obtain the desired decomposition.
\end{proof}

\subsection{Product topology continuity of operators}

The importance of operators on $\ell_\infty$ given by $c_0$-matrices is 
expressed in the following theorem which exploits the fact that $\ell_\infty$
is the bidual space of $c_0$. In the theorem below, the weak$^*$ topology on $\ell_\infty$
is given by the duality $\dual{\ell_1}=\ell_\infty$ and $\tau_p$ 
denotes the product topology in $\R^\N$.

\begin{theorem}\label{summary_matrices}
 Let $R:\ell_{\infty}\rightarrow\ell_{\infty}$ be a linear bounded operator.
 The following are equivalent:
\begin{enumerate}
 \item $R=\ddual{(R|c_0)}$.
\item $R$ is given by a $c_0$-matrix.
\item $R$ is $w^*$-$w^*$-continuous and $R[c_0]\subseteq c_0$.
\item $R| B_{\ell_{\infty}}:(B_{\ell_{\infty}},\tau_p)\rightarrow (\ell_{\infty},\tau_p)$
is continuous and $R[c_0]\subseteq c_0$.
\end{enumerate}
\end{theorem}
\begin{proof} Appendix \ref{summary_matrices_appendix}.
\end{proof}

Thus, the nonzero antimatrix operators are  discontinuous in the product topology. Such
discontinuities are not, however,  incompatible with being an automorphism or having
a nice behaviour on $c_0$.

\begin{theorem}\label{discontinuousauto} There are discontinuous automorphisms of
$\ell_\infty$ preserving $c_0$. There are different automorphisms on $\ell_\infty$
which agree on $c_0$. They can be the identity on $c_0$.
\end{theorem}
\begin{proof}
Let $(A_i)_{i\in \N}$ be a partition of $\N$ into infinite sets. Let $x_i$ be any
nonprincipal ultrafilter such that $A_i\in x_i$ for all $i\in \N$.
For a permutation $\sigma: \N\rightarrow \N$ define
$$R_\sigma(f)(n)= f(n)-\beta f(x_i)+\beta f(x_{\sigma(i)}),$$
where $i\in \N$ is such that  $n\in A_i$.
First note that $R_{\sigma^{-1}}\circ R_\sigma=R_\sigma\circ R_{\sigma^{-1}}=Id$ and so
 $R_\sigma$ is an automorphism.
One verifies that $R_\sigma|c_0$ is the identity for any permutation $\sigma$, in particular
$R_\sigma-Id\not=0$ is antimatrix for any permutation $\sigma$ different than
the identity  and hence $R_\sigma$ is discontinuous by
\ref{summary_matrices}.

\end{proof}

In this proof we really decompose $\ell_\infty$ as a direct sum $X\oplus Y$,
both factors necessarily isomorphic to $\ell_\infty$: the first of the functions constant on
each set $A_i$ and the second of the functions equal to zero in each point $x_i$.
Since the second factor contains $c_0$, the automorphisms of the first factor induce automorphisms of
$\ell_\infty$ which do not move $c_0$. This lack of continuity is also
present in homomorphisms of $\wp(\N)$ (3.2.3. of \cite{ilijas}) but not its automorphisms.

\section{Operators on $\ell_\infty/c_0$}

\subsection{Ideals of operators on $\ell_\infty/c_0$}

As usual by an (left, right) ideal we will mean a collection $\mathcal I$ of operators
such that $T+S\in \mathcal I$ whenever $T,S\in \mathcal I$
and $S\circ R, R\circ S\in \mathcal I$ ($R\circ S \in \mathcal I$,  $S\circ R\in \mathcal I$)
whenever $S\in\mathcal I$ and $R$ is any operator on $\ell_\infty/c_0$.
We say that an operator $T$ on $\ell_\infty/c_0$ factors through $\ell_\infty$ 
if, and only if, there are operators $R_1:\ell_\infty/c_0\rightarrow \ell_\infty$
and $R_2: \ell_\infty\rightarrow \ell_\infty/c_0$  such that $T=R_2\circ R_1$.
It is clear that operators which factor through $\ell_\infty$ form a 
two-sided ideal. Also it is well known that weakly compact
operators form a two-sided proper ideal (VI 4.5. of \cite{dunford}). We introduce another class
of operators:
\begin{definition}
An operator $T:\ell_\infty/c_0\rightarrow \ell_\infty/c_0$ is locally null if, and only if,
for every infinite $A\subseteq \N$ there is an infinite $A_1\subseteq A$ such that
$$T\circ I_{A_1}=0.$$
\end{definition}

Locally null should really be right-locally null, but left-locally null is just null, so 
there is no need of using the word ``right".

\begin{proposition}\label{idealsinclusions} Locally null operators form a proper left ideal which contains
all weakly compact operators and all operators which factor through $\ell_\infty$.
\end{proposition}
\begin{proof}
It is clear that locally null operators form a proper left ideal.

 Let us prove that every weakly compact operator  on $\ell_\infty/c_0$ is locally null.
We will use the fact that an operator $T$ 
on a $C(K)$-space is weakly compact if, and only if,  $\|T(f_n)\|\rightarrow 0$
whenever $(f_n)_{n\in\N}\subseteq C(K)$ is a bounded pairwise disjoint sequence 
(i.e., $f_n\cdot f_m=0$ for $n\neq m$) (see Corollary VI--17 of \cite{diestel-ulh}). %page 160

Let $A\subseteq \N$ be infinite. Consider $\{A_{\xi}:\xi<\omega_1\}$, a family 
of almost disjoint infinite subsets  of $A$. Notice that by the weak compactness
of $T$ we have that the set of $\alpha\in\omega_1$ such that
$T\circ I_{A_\alpha}\not=0$ must be at  most countable, so take $\alpha$
outside this set.

Now let $T=R_2\circ R_1$ where $R_1:\ell_\infty/c_0\rightarrow \ell_\infty$
and $R_2: \ell_\infty\rightarrow \ell_\infty/c_0$. Let $\mu_n=\dual{R_1}(\delta_n)$.
Let $A\subseteq \N$ be infinite. As the supports of $\mu_n$'s are c.c.c.
and there are continuum many pairwise disjoint clopen subsets of $A^*$, there is an infinite
$A_1\subseteq A$ such that $|\mu_n|(A_1^*)=0$ for every $n\in\N$.
It follows that $R_1\circ I_{A_1}=0$, which completes the proof.

\end{proof}

\begin{proposition}\label{locallynullsurjective} 
There is a locally null operator
on $\ell_\infty/c_0$ which factors through $\ell_\infty$ and is surjective. 
There is no surjective weakly  compact operator.
\end{proposition}
\begin{proof} Let $(x_n)_{n\in\N}$ be a discrete subset of $\N^*$. Define
$R:\ell_\infty/c_0\rightarrow\ell_\infty$ by $R([f]_{c_0})=(f^*(x_n))_{n\in\N}$.
It is well-known that the closure of $\{x_n:n\in \N\}$ in $\N^*$ is homeomorphic to $\beta\N$.
So by the Tietze extension theorem $R$ is onto $\ell_\infty$. Furthermore, $Q\circ R$ is surjective, where 
$Q:\ell_\infty\rightarrow\ell_\infty/c_0$ is the quotient map.
As no clopen subset $A^*$  of $\N^*$ is separable, below every infinite $A$ 
there is an infinite $A_1\subseteq A$ such that no $x_n$ belongs to $A_1^*$. Then,
$R\circ I_{A_1}=0$ which proves that $R$ is locally null, and so $Q\circ R$ as well.

Weakly compact operators on an infinite dimensional  $C(K)$ cannot be surjective because
weakly compact subsets of an infinite dimensional Banach space have empty interior if the space is not reflexive.
So countable unions of them are of the first Baire category, and in particular,
the images of the balls cannot cover an infinite dimensional Banach space $C(K)$.
\end{proof}

See \ref{notrightidealCH} for more information on the ideal of locally null operators under CH.

\subsection{Local behaviour of functions associated with  the adjoint operator}

In general, for a linear bounded operator $T$ acting
on the Banach space $C(K)$ for a compact $K$, 
the function which sends $x\in K$ to $\|\dual{T}(\delta_x)\|$ is lower semicontinuous
(e.g., Lemma 2.1. of \cite {plebanekisrael}) and may be quite discontinuous.  

\begin{proposition}\label{discontinuousnorm}
Suppose $F\subseteq \N^*$ is a nowhere dense retract
of $\N^*$. There is a linear bounded operator $T$ on $C(\N^*)$ such that
the function $\alpha:\N^*\rightarrow \mathbb{R}$ defined by
$$\alpha(y)=\|\dual{T}(\delta_y)\|$$
for every $y\in\N^*$, is discontinuous in every point of $F$.
\end{proposition}
\begin{proof} Define $T$ by putting 
$$T(f)=f-f\circ r,$$
where $r:\N^*\rightarrow F$ is the retraction onto $F$.  Then $T(f)(y)=f(y)-f(r(y))$
and so $\dual{T}(\delta_y)=\delta_y-\delta_{r(y)}$. Hence $\alpha=2\chi_{\N^*\setminus F}$.
Since $F$ is nowhere dense, the set of discontinuities of $\alpha$ is $F$. 
\end{proof}

By Lemma 4.1. of \cite{plebanekisrael}, for every lower semicontinuous 
function, every  $\varepsilon>0$ and
every open $U\subseteq \N^*$ there is an open $V\subseteq U$ such that
the function's oscillation on $V$ is smaller than $\varepsilon$. Hence, by \ref{gdelta}
there is a dense open subset of $\N^*$ where  the function which sends $y\in \N^*$
to $\|\dual{T}(\delta_y)\|$ is locally constant.
In the case of $\N^*$ we have not only
the local stabilization of the values of $\|\dual{T}(\delta_y)\|$ but 
the local stabilization of  the Hahn decompositions of the measures
 $\dual{T}(\delta_y)$:

\begin{lemma}\label{norm_stability} 
Suppose $T:C(\N^*)\rightarrow C(\N^*)$ is bounded linear and $B\subseteq \N$ is
infinite. Then, there are an infinite $B_1\subseteq_* B$, a real number $s$,  
and partitions $\N=C_n\cup D_n$ into infinite sets, such that for every $y\in B_1^*$ 
 we have:
\begin{enumerate}
\item[(i)] $s=\|\dual{T}(\delta_y)\|$
\item[(ii)] if $\dual{T}(\delta_y)=\mu^+-\mu^-$ is the Jordan decomposition of the measure,
 then $\mu^-(C^*_n)<1/4(n+1)$ and $\mu^+(D^*_n)<1/4(n+1)$.
\end{enumerate}
\end{lemma}

\begin{proof} 
We construct by induction a $\subseteq_*$-decreasing sequence of infinite
sets $(A_n)_{n\in\N}$, $y_n\in A_n^*$, and 
partitions $\N=C_n\cup D_n$ into infinite sets 
such that for every $n\in\N$ we have:
\begin{enumerate}
\item $\sup\{\|\dual{T}(\delta_y)\|: y\in A_n^*\}-\|\dual{T}(\delta_{y_n})\|<1/6(n+1)$
\item $\|\dual{T}(\delta_{y_n})\|-\dual{T}(\delta_{y_n})(C_n^*)+\dual{T}(\delta_{y_n})(D_n^*)<2/6(n+1)$
\item For all $y\in A_{n+1}^*$ we have 
$|\dual{T}(\delta_y)(C_n^*)-\dual{T}(\delta_{y_n})(C_n^*)|<1/6(n+1)$ and 
$|\dual{T}(\delta_y)(D_n^*)-\dual{T}(\delta_{y_n})(D_n^*)|<1/6(n+1)$.
\end{enumerate}
This is arranged as follows. Put $A_0=B$ and assume we have constructed $A_n$. Take
$y_n\in A_n^*$ such that $\|\dual{T}(\delta_{y_n})\|>\sup\{\|\dual{T}(\delta_y)\|: y\in A_n^*\}-1/6(n+1)$.
Take a Hahn decomposition $\N^*=H^+_n\cup H_n^-$ for the measure $\dual{T}(\delta_{y_n})$.
By the regularity, we may choose an infinite $C_n\subseteq\N$ such that 
$|\dual{T}(\delta_{y_n})(H^+_n)-\dual{T}(\delta_{y_n})(C_n^*)|<1/6(n+1)$. If we put 
$D_n=\N\setminus C_n$, we obtain $|\dual{T}(\delta_{y_n})(H_n^-)-\dual{T}(\delta_{y_n})(D_n^*)|<1/6(n+1)$.
Therefore,
$$\|\dual{T}(\delta_{y_n})\|=\dual{T}(\delta_{y_n})(H^+_n)-\dual{T}(\delta_{y_n})(H^-_n)
<\dual{T}(\delta_{y_n})(C_n^*)-\dual{T}(\delta_{y_n})(D_n^*)+2/6(n+1),$$
and so (2) holds.

By the weak$^*$ continuity of $\dual{T}$, the set of points which satisfy the condition in (3) is an 
open neighbourhood of $y_n$, so we may take $A_{n+1}\subseteq_* A_n$ satisfying (3).
This ends the induction.

Notice that $|\dual{T}(\delta_{y_n})(C_n^*)|\leq\|T\|$ for every $n\in\N$, and so
there exists a convergent subsequence of $(\dual{T}(\delta_{y_n})(C_n^*))_{n\in\N}$. 
The same is true for the $D_n$'s and so we may assume that both of these 
sequences converge. Let us define
$$s^+=\lim_{n\rightarrow\infty}\dual{T}(\delta_{y_n})(C_n^*)
\qquad s^-=\lim_{n\rightarrow\infty}\dual{T}(\delta_{y_n})(D_n^*).$$
Now let $B_1\subseteq \N$ be infinite such that $B_1\subseteq_* A_n$, for all $n\in\N$.
We will show that for every $y\in B_1^*$ we have $\|\dual{T}(\delta_y)\|=s$, where $s=s^+-s^-$.

So let us fix $y\in B_1^*$. Notice that from (3) we obtain that
$$s=\lim_{n\rightarrow\infty}(\dual{T}(\delta_y)(C_n^*)-\dual{T}(\delta_y)(D_n^*))
=\lim_{n\rightarrow\infty}\dual{T}(\delta_y)(\chi_{C_n^*}-\chi_{D_n^*}).$$
Therefore, $s\leq \|\dual{T}(\delta_y)\|$.

Now, by (1) and (2) the following holds for every $n\in\N$:
$$\begin{array}{rcl}
   \dual{T}(\delta_{y_n})(C_n^*)-\dual{T}(\delta_{y_n})(D_n^*)&>&\|\dual{T}(\delta_{y_n})\|-2/6(n+1)\\
    &>&\sup\{\|\dual{T}(\delta_z)\|:z\in A_n^*\}-1/2(n+1)\\
    &\geq& \|\dual{T}(\delta_y)\|-1/2(n+1).
  \end{array}$$

Therefore, 
$s=\lim_{n\rightarrow\infty}\dual{T}(\delta_{y_n})(C_n^*)-\dual{T}(\delta_y)(D_n^*)
\geq \|\dual{T}(\delta_y)\|$. This proves the first statement of the lemma.

To check that (ii) holds, let us fix $y\in B_1^*$ and the Jordan decomposition 
for the measure $\dual{T}(\delta_y)=\mu^+-\mu^-$. 
By going to a subsequence if necessary, we may assume that 
$|s-(\dual{T}(\delta_{y_n})(C_n^*)-\dual{T}(\delta_{y_n})(D_n^*))|<1/6(n+1)$, for every $n\in\N$.

Observe that since $C_n^*\cup D_n^*=\N^*$, we
have that 
$$\begin{array}{rcl}
\dual{T}(\delta_y)(C_n^*)-\dual{T}(\delta_y)(D_n^*)&=&
|\dual{T}(\delta_y)|(C_n^*)+|\dual{T}(\delta_y)|(D_n^*)-2\mu^-(C_n^*)-2\mu^+(D_n^*)\\
&=&\|\dual{T}(\delta_y)\|-2\mu^-(C_n^*)-2\mu^+(D_n^*), 
\end{array}$$

for every $n\in\N$. Then, by (3) we obtain
$$\begin{array}{rcl}
2(\mu^-(C_n^*)+\mu^+(D_n^*))&=&s-(\dual{T}(\delta_y)(C_n^*)-\dual{T}(\delta_y)(D_n^*))\\
&\leq &s-(\dual{T}(\delta_{y_n})(C_n^*)-\dual{T}(\delta_{y_n})(D_n^*)-2/6(n+1))\\
&<&1/6(n+1)+2/6(n+1)
=1/2(n+1).\end{array}$$

 Since both $\mu^-(C_n^*)$ and $\mu^+(D_n^*)$
are non negative, they are both strictly less than $1/4(n+1)$ and (ii) is proved.
\end{proof}

\begin{corollary} 
Suppose $T:C(\N^*)\rightarrow C(\N^*)$ is bounded linear and $B\subseteq \N$ is
infinite. Then, there are an infinite $B_1\subseteq_* B$ and a Borel partition
$\N^*=X\cup Y$ such that $X$ and $Y$ form a Hahn decomposition of $\dual{T}(\delta_y)$,
for every $y\in B_1^*$.
\end{corollary}
\begin{proof}
 Let $B_1\subseteq_* B$ and $C_n,D_n\subseteq\N$ be as in \ref{norm_stability}.
Let $(n_k)_{k\in\N}$ be a strictly increasing sequence of positive integers 
such that $\frac{1}{4(n_k+1)}<1/2^k$, $\forall k\in\N$.

Let $F_i=\bigcap_{k\geq i} C_{n_k}^*$, $X=\bigcup_{i\in\N}F_i$ and $Y=\N^*\setminus X$.
Fix $y\in B_1^*$ and let $\dual{T}(\delta_y)=\mu^+-\mu^-$ be a Jordan decomposition of 
the measure. Since $F_i\subseteq F_{i+1}$ and $F_i\subseteq C_{n_i}^*$, for every $i\in\N$,
by \ref{norm_stability} we have that 
$$\mu^-(F_{i_0})\leq\mu^-(F_i)\leq\mu^-(C_{n_i})<\frac{1}{4(n_i+1)},$$
for every $i_0\in\N$ and every $i\geq i_0$. Therefore, $\mu^-(F_{i_0})=0$, for every 
$i_0\in\N$, and so $\mu^-(X)=0$.

On the other hand, we have $Y\subseteq \N^*\setminus F_i=\bigcup_{k\geq i}D_{n_k}^*$,
for every $i\in\N$. Therefore, 
$\mu^+(Y)\leq \sum_{k\geq i}\mu^+(D_{n_k}^*)<\sum_{k\geq i}\frac{1}{4(n_k+1)}
<\sum_{k\geq i}1/2^k$, for every $i\in\N$. It follows that $\mu^+(Y)=0$.
\end{proof}

\begin{corollary}
Suppose that $T:C(\N^*)\rightarrow C(\N^*)$ %\ell_\infty/c_0\rightarrow \ell_\infty/c_0$
is a linear bounded operator and $B\subseteq \N$ is infinite. Then, there is an
infinite $B_1\subseteq B$ such that the functions which map $y\in B_1^*$ to
the positive part, to the negative part, and to the total
variation measure of the measure $\dual{T}(\delta_y)$, respectively,
are all weak$^*$ continuous. In particular, $T$ is left-locally a regular operator.
\end{corollary}
\begin{proof} 
Let $B_1\subseteq B$ and $s$ be as in \ref{norm_stability}. If $s=0$, then 
the first part of the corollary is trivially true and $P_{B_1}\circ T=0$ is positive.
 Otherwise, consider the operator
$\frac{1}{s}T$. We have that $\|\dual{(\frac{1}{s}T)}(\delta_y)\|=1$ for all
$y\in B_1^*$. Now apply the second part of Lemma 2.2 of \cite{plebanekstudia} which says that
on the dual sphere sending the measure $\mu$ to its total variation $|\mu|$ is weakly$^*$ continuous.
Since the positive and the negative parts of $\mu$ can be obtained from $\mu$ and $|\mu|$,
and $|s\mu|=s|\mu|$ for nonnegative $s$, using the weak$^*$ continuity of
$\dual{T}$ we conclude the first part of the corollary.

For the second part we will define two positive operators $T^+$ and $T^-$ such
that $T^+-T^-=P_{B_1}\circ T$. For every $y\in B_1^*$ we have that 
$\dual{(P_{B_1}\circ T)}(\delta_y)=\dual{T}(\dual{P_{B_1}}(\delta_y))=\dual{T}(\delta_y)$.
So for every $y\in B_1^*$ define 
$$T^+(f)(y)=\int f\ud(\dual{T}(\delta_y))^+, \qquad T^-(f)(x)=\int f\ud(\dual{T}(\delta_y))^-.$$
It is clear that $P_{B_1}\circ T=T^+-T^-$.
 The linearity of $T^+$ and $T^-$ follows from general properties of the integral.
% The boundedness of $T^+$ and $T^-$ follows from the weak$^*$ continuity of the 
% maps which send $y\in B^*$ to the positive and the negative parts of $\dual{T}(\delta_y)$. 
To see that they are bounded, notice that for every 
$f\in C(\N^*)$ with $\|f\|\leq 1$ we have that $\|T^+(f)\|\leq\sup_{y\in B_1^*}(\dual{T}(\delta_y))^+(\N^*)$.
If $\{(\dual{T}(\delta_y))^+(\N^*):y\in B_1^*\}$ were unbounded, 
the set $F_n=\{y\in B_1^*:(\dual{T}(\delta_y))^+(\N^*)\geq n\}$ would be nonempty for every $n\in\N$. 
But by the $w^*$-continuity of the map $y\mapsto (\dual{T}(\delta_y))^+$,
each $F_n$ is closed.
Since the $F_n$'s form a decreasing chain of nonempty closed sets,  there exists 
$y\in\bigcap_{n\in\N}F_n$. But this is impossible. 

The same argument shows that $T^-$ is bounded.

\end{proof}

\begin{definition}
Let $X$, $Y$ be topological spaces. A function $\varphi:X\rightarrow\wp(Y)$
is called \emph{upper semicontinuous} if for every open set $V\subseteq Y$
the following set is open in $X$
$$\{x\in X:\varphi(x)\subseteq V\}.$$
\end{definition}

Our main interest in multifunctions is related  to the following

\begin{definition} Suppose that $T:\ell_\infty/c_0\rightarrow \ell_\infty/c_0$
is a linear bounded operator and $\varepsilon>0$. We define
$$\varphi_\varepsilon^T(y)=\{x\in \N^*: |\dual{T}(\delta_y)(\{x\})|\geq\varepsilon\},$$
$$\varphi^T(y)=\bigcup_{\varepsilon>0}\varphi^T_\varepsilon(y).$$
\end{definition}

\begin{proposition}
 There is a linear bounded $T:\ell_\infty/c_0\rightarrow \ell_\infty/c_0$
such that $\varphi^T_{1/2}$ is not upper semicontinuous.
\end{proposition}
\begin{proof} 
Consider the operator from the proof of \ref{discontinuousnorm}.
We have $\varphi_{1/2}(y)=\emptyset$ if $y\in F$ and $\varphi_{1/2}(y)=\{r(y), y\}$
whenever $y\in \N^*\setminus F$. So for example taking $V=\N^*\setminus F$
we obtain that $\{y: \varphi_{1/2}(y)\subseteq V\}=F$ which is not open, but
closed nowhere dense.
\end{proof}

However we obtain the left-local upper semicontinuity:

\begin{lemma}\label{upper_semicontinuity}
Let $T: C(\N^*)\rightarrow C(\N^*)$ be a bounded linear operator and let 
$B\subseteq \N$ be infinite. 
% Let $B\subseteq \N$ be infinite and let $\N=C_n\cup D_n$ be partitions 
% into infinite sets, such that for every $y\in B^*$ and its
% Jordan decomposition  $\dual{T}(\delta_y)=\mu^+-\mu^-$ we have that
% $\mu^-(C^*_n)<1/4(n+1)$ and $\mu^+(D^*_n)<1/4(n+1)$. 
Then, there exists $B_1\subseteq_*B$ such that $\varphi^T_{\varepsilon}| B_1^*$ is 
upper semicontinuous for every $\varepsilon>0$.
\end{lemma}
\begin{proof}
Let $B_1\subseteq_* B$ and $C_n$, $D_n$ be as in \ref{norm_stability}.
Fix $y\in B_1^*$ and an open $V\subseteq \N^*$ such that 
$\varphi^T_{\varepsilon}(y)\subseteq V$. Then, for every $x\in\N^*\setminus V$
we find a clopen neighbourhood $U_x$ of $x$ as follows. First notice that $|\dual{T}(\delta_y)(\{x\})|<\varepsilon$. 
Let $N_x\in\N$ be such that 
$|\dual{T}(\delta_y)(\{x\})|<\varepsilon-1/N_x$. Now, using
 the regularity of the measure $\dual{T}(\delta_y)$, 
find $U_x$ such that $|\dual{T}(\delta_y)|(U_x)<\varepsilon-1/N_x$. We may assume that 
$U_x$ is included in either $C_{N_x}^*$ or $D_{N_x}^*$. 

Since $\N^*\setminus V$ is compact, we may find $x_0,\dots,x_{k-1}\in\N^*\setminus V$ 
such that $\N^*\setminus V\subseteq\bigcup_{i<k} U_{x_i}$.
Using the weak$^*$-continuity of $\dual{T}$, we now find a clopen neighbourhood of 
$y$, say $E^*$, which we may assume to be included in $B_1^*$, such that for every $z\in E^*$ we have 
$$|\dual{T}(\delta_z)(U_{x_i})|<\varepsilon-1/N_{x_i}, \text{ for each }i<k.$$
This is possible because $|\dual{T}(\delta_y)(U_x)|\leq|\dual{T}(\delta_y)|(U_x)<\varepsilon-1/N_x$.

We claim that for every $z\in E^*$ and every  $x\in\bigcup_{i<k} U_{x_i}$ we have 
$|\dual{T}(\delta_z)(\{x\})|<\varepsilon$, that is $\varphi^T_{\varepsilon}(z)\subseteq V$.
 So fix $z\in E^*$ and $x\in U_{x_i}$. Let $\dual{T}(\delta_z)=\mu_z^+-\mu_z^-$
be a Jordan decomposition of the measure.
Notice that $|\dual{T}(\delta_z)|(U_{x_j}) \leq |\dual{T}(\delta_z)(U_{x_j})|+2\mu_z^-(U_{x_j})$
and $|\dual{T}(\delta_z)|(U_{x_j}) \leq |\dual{T}(\delta_z)(U_{x_j})|+2\mu_z^+(U_{x_j})$.
So if $U_{x_i}\subseteq C_{N_{x_i}}^*$, since 
$\mu_z^-(U_{x_i})\leq \mu_z^-(C^*_{N_{x_i}})<1/4(N_{x_i}+1)$, we have that  

$$\begin{array}{rcl}
|\dual{T}(\delta_z)(\{x\})|&\leq&|\dual{T}(\delta_z)(U_{x_j})|+2\mu_z^-(U_{x_j})\\
&<&\varepsilon-1/N_{x_j}+1/2(N_{x_j}+1)\\
&<&\varepsilon \end{array}$$

If $U_{x_i}\subseteq D_{N_{x_i}}^*$, we use the fact that 
$\mu_z^+(U_{x_i})\leq \mu_z^+(D^*_{N_{x_i}})<1/4(N_{x_j}+1)$ to obtain the same result.
 
\end{proof}

\subsection{Fountains and funnels}

The property of being locally null can be expressed using a  topological
property of $\dual{T}$.

\begin{proposition}\label{locallydeterminednwd} A bounded linear operator $T:C(\N^*)\rightarrow C(\N^*)$  is
locally null if, and only if, there is a nowhere dense set $F\subseteq\N^*$
such that $\dual{T}(\delta_y)$ is concentrated on $F$ for every $y\in \N^*$.
\end{proposition}
\begin{proof}
Suppose $T$ is locally null. If we set $D=\bigcup\{A^*:T\circ I_A=0\}$, then
$D$ is an open dense set. Suppose $|\dual{T}(\delta_y)|(D)>\varepsilon$ for some $y\in\N^*$
and some $\varepsilon>0$. By the regularity of the measure we may find a compact
$G\subseteq D$ such that $|\dual{T}(\delta_y)|(G)>\varepsilon$. We may further find
finitely many $A_0,\dots,A_{n-1}\subseteq \N$ such that $T\circ I_{A_i}=0$ for all $i<n$
and $\sum_{i<n}|\dual{T}(\delta_y)|(A_i^*)>\varepsilon$. Choose $i<n$ such that
$|\dual{T}(\delta_y)|(A_i^*)>\varepsilon/n$ and a function $f$ with support included in $A_i^*$
such that $\dual{T}(\delta_y)(f)>\varepsilon/n$. Then, $T(f)(y)\neq 0$, which
contradicts the hypothesis. Therefore, $\dual{T}(\delta_y)$ is concentrated on $F=\N^*\setminus D$,
for every $y\in\N^*$.

Conversely, suppose $F$ is a nowhere dense set such that for every $y\in\N^*$ the measure $\dual{T}(\delta_y)$ is 
concentrated on $F$. 
Given an infinite $A\subseteq\N$, take $A_1\subseteq_*A$ infinite
such that $A_1^*\cap F=\emptyset$. Then, $|\dual{T}(\delta_y)|(A_1^*)=0$ and it follows 
that $T\circ I_{A_1}=0$.
\end{proof}

As in the previous proposition, many results in the following parts of the paper
will show the important role played by nowhere dense sets of $\N^*$
in the context of operators on $C(\N^*)$. It is this fact that 
leads to the definitions of fountains, funnels and fountainless and funnelless operators:

\begin{definition}\label{locallydetermined}
An operator $T:C(\N^*)\rightarrow C(\N^*)$  is called 
%(strongly)
 fountainless or without fountains if, and only if,
 for every nowhere dense set $F\subseteq \N^*$  the set 
$$G=\{y\in\N^*:\dual{T}(\delta_y)\ \hbox{is nonzero and concentrated on}\ F\}$$ is nowhere dense.
A fountain for $T$ is 
 a pair $(F, U)$ with $F\subseteq \N^*$ nowhere dense and $U\subseteq \N^*$
 open such that 
 all the measures 
$\dual{T}(\delta_y)$ for $y\in U$ are concentrated on $F$. 
%(for every nowhere dense set $F\subseteq \N^*$  the set   $$\{y\in\N^*: |T(\delta_y)|(F)\not=0\}$$ is nowhere dense).
\end{definition}

\begin{lemma}\label{locallydet+somewherelocallynull->null}
 Let $T$ be fountainless and let $B\subseteq\N$ be infinite.
If $P_B\circ T$ is locally null, then $P_B\circ T=0$.
\end{lemma}
\begin{proof}
 By \ref{locallydeterminednwd} there is a nowhere dense $F\subseteq \N^*$
such that for every $y\in B^*$ we have that $\dual{(P_B\circ T)}(\delta_y)$,
which is equal to $\dual{T}(\delta_y)$, is concentrated on $F$.
By \ref{locallydetermined} the set $G=\{y\in B^*: \dual{T}(\delta_y)\not=0\}$ is nowhere
dense. But this means that for every $f\in C(\N^*)$ we have $T(f)(x)=0$ 
if $x\in B^*\setminus G$. Since $B^*\setminus G$ is dense in $B^*$ 
we conclude that $P_B\circ T=0$.

\end{proof}

\begin{corollary}\label{locallynull-det-null}
Suppose that $T$ is locally null and has no fountains, then $T=0$.
\end{corollary}
\begin{proof} 
Put $B=\N$ in \ref{locallydet+somewherelocallynull->null}.
\end{proof}

%\begin{lemma}\label{stronglydetermined}
% Let $T: C(\N^*)\rightarrow C(\N^*)$ be a strongly fountainless operator
%and let $D\subseteq \N^*$ be a dense open set. Then there is a dense open $E\subseteq \N^*$
%and measures $\mu_x$ for all $x\in E$ such that
%$$T(f)(x)=\int_D fd\mu_x.$$
%\end{lemma}
%\begin{proof}
%\end{proof}

% \begin{definition}
% An operator $T:C(\N^*)\rightarrow C(\N^*)$  is called (strongly) fountainless, if, and only if, 
%  for every nowhere dense set $G\subseteq \N^*$ there is a nowhere dense set $F$ such that 
% all the measure $\dual{T}(\delta_x)$ are concentrated  on $F$
% \end{definition}

%\begin{proposition} Suppose that $T=\sum_{i\in \N} T_i$ where each $T_i$
%is strongly fountainless. Then $T$ is strongly fountainless.
%\end{proposition}
%\begin{proof} Fix a nowhere dense set $F\subseteq \N^*$.
%Let $G_i$ be nowhere dense sets of all those $x\in \N^*$
%for which $|T_i^*(\delta_x)|(F)\not=0$.
%$G=\bigcup_{i\in \N} G_i$ is nowhere dense by \ref{?}. If $x\not\in G$, then
%$\dual{T}(\delta_x)=\sum_{i\in \N} T_i^*(\delta_x)$ has zero variation on $F$
%as the sum of measures which have zero variation on $F$.
%\end{proof}

\begin{definition}\label{everywhere-non0}
We say that an operator $T: C(\N^*)\rightarrow C(\N^*)$ is everywhere present
if, and only if, for every infinite $B\subseteq \N$ we have that $P_B\circ T\not=0$ 
\end{definition}

In the following lemma we obtain a kind of left dual to an improvement of
 a theorem of Cengiz (``P" in \cite{cengiz}) obtained by Plebanek
 (Theorem 3.3. in \cite{plebanekisrael}) which implies  that if $T$
is an isomorphic embedding then
every $x\in \N^*$ is in $\varphi^T(y)$ for some $y\in \N^*$.

\begin{lemma}\label{nonempty} 
Suppose that  $T: C(\N^*)\rightarrow C(\N^*)$
is an everywhere present fountainless operator. Then, for every infinite $B\subseteq\N$
there exists an infinite $B_1\subseteq_* B$ such that $\varphi^T(y)\neq\emptyset$,
for every $y\in B_1^*$.
\end{lemma}

\begin{proof} 
Given an infinite $B\subseteq \N$, let $B_1\subseteq_* B$ and $C_n$, $D_n\subseteq \N$ 
be as in  \ref{norm_stability}. 
Suppose that $y_0\in B_1^*$ is such  that $\varphi^T(y_0)=\emptyset$. 
For every $n\in\N$ we find an open covering of $\N^*$ as follows. Given $x\in\N^*$,
find by the regularity of the measure $\dual{T}(\delta_{y_0})$ a clopen neighbourhood 
of $x$, say $U_x$, such that $|\dual{T}(\delta_{y_0})|(U_x)<1/2(n+1)$ and
$U_x$ is included in either $C_n^*$ or $D_n^*$.

By the compactness of $\N^*$ obtain
for each $n\in\N$ an open covering $\{U_{n,i}:i<j_n\}$ of  $\N^*$  such that
for each $i<j_n$ we have that 
\begin{enumerate}
 \item $|\dual{T}(\delta_{y_0})(U_{n,i})|\leq|\dual{T}(\delta_{y_0})|(U_{n,i})<1/2(n+1)$, and
\item either $U_{n,i}\subseteq C_n^*$  or $U_{n,i}\subseteq D_n^*$.
\end{enumerate}
 By the weak$^*$ continuity of
$\dual{T}$ there are open neighbourhoods $V_n$ of $y_0$ such that
$|\dual{T}(\delta_{y})(U_{n,i})|<1/2(n+1)$ holds for all $y\in V_n$ and all $i<j_n$.
 Let $V^*$ be a clopen subset of $\bigcap_{n\in \N}V_n\cap B_1^*$ and 
consider the family $\mathcal A\subseteq \wp(\N)$ of those sets $A$ such that 
 for each $n\in \N$ we have $A^*\subseteq U_{n,{i_n}}$ for some $i_n<j_n$.
We claim that $|\dual{T}(\delta_{y})|(A^*)=0$ for every $y\in V^*$ and every $A\in\mathcal A$.

So fix $y\in V^*$, $A\in\mathcal A$ and $n\in\N$. We will show that $|\dual{T}(\delta_{y})|(A^*)<1/(n+1)$. 
Let $\dual{T}(\delta_y)=\mu^+-\mu^-$ be a Jordan decomposition of the measure.
By  \ref{norm_stability} we have that $\mu^-(C_n^*)<1/4(n+1)$ and 
$\mu^+(D^*_n)<1/4(n+1)$.
Assume without loss of generality that $U_{n,i_n}\subseteq C_n^*$. Then,

$$\begin{array}{rcl}
    |\dual{T}(\delta_y)|(A^*)&\leq& |\dual{T}(\delta_y)|(U_{n,i_n})\\
    &=&\dual{T}(\delta_y)(U_{n,i_n})+2\mu^-(U_{n,i_n})\\
    &\leq& |\dual{T}(\delta_y)(U_{n,i_n})|+2\mu^-(C^*_n)\\
    &<& 1/2(n+1)+2/4(n+1)\\
    &=& 1/(n+1)
  \end{array}$$

So the claim is proved.

Notice that this implies that $(P_V\circ T)(f)=0$ for every $f\in C(\N^*)$ 
whose support is included in $A^*$, for some $A\in\mathcal A$. Therefore,
if $\mathcal A$ is a dense family,
by  \ref{locallydet+somewherelocallynull->null} we would have that $(P_V\circ T)(g)=0$
for all $g\in C(\N^*)$, but this would contradict the hypothesis that $T$
is everywhere present.

We prove that $\mathcal A$ is a dense family. For a fixed infinite 
$E_0\subseteq \N$, we may define by induction a $\subseteq_*$- decreasing
sequence $(E_n)$ of infinite sets by choosing $\emptyset\neq E^*_{n+1}\subseteq E_n^*\cap U_{n,i_n}$,
for some $i_n<j_n$ (this is possible because $\{U_{n,i}:i<j_n\}$ is an open covering
of $\N^*$ for each $n\in\N$). Take $A$ such that $A\subseteq_* E_n$, for all $n\in\N$.
It is clear that $A\subseteq_* E_0$ and $A\in\mathcal A$.

\end{proof}

%\begin{proposition}\label{locallynullsurjective}
%There is a locally null, not fountainless operator which
%is surjective.
%\end{proposition}
%\begin{proof} Let $F\subseteq \N^*$ be a nowhere dense (one obtained
%as a closure of a discrete countable set or a nontrivial one like in \cite{downontrivial}).
%Let $\phi: \N^*\rightarrow F$ be the witnessing homeomorphism.
%Now $T_\phi$ is the required operator. It is surjective by the Tietze theorem.
%\end{proof}

Let us introduce a dual notion to a  fountain:

\begin{definition}\label{funnelless}
An operator $T:C(\N^*)\rightarrow C(\N^*)$  is called funnelless or without funnels if, and only if, 
 for every nowhere dense set $F\subseteq \N^*$  there is a nowhere dense $G\subseteq\N^*$
such that for all $y\in F$ the measure $\dual{T}(\delta_y)$ is concentrated on $G$.
A funnel for $T$ is a pair $(U,F)$ with $F\subseteq \N^*$ nowhere dense and $U\subseteq \N^*$ open
 such that there is no proper closed subset of $U$ where all the measures 
$\dual{T}(\delta_y)|U$ for $y\in F$ are concentrated. 
\end{definition}

\subsection{Operators induced by continuous maps and nonatomic operators}

\begin{definition} Suppose that  $\psi: \N^*\rightarrow \N^*$ is a continuous map.
Then $T_\psi: C(\N^*)\rightarrow C(\N^*)$ is given
for every $f\in C(\N^*)$  by $$T_\psi(f)=f\circ \psi.$$
\end{definition}

\begin{definition}\label{quasi-open} A continuous map $\psi: \N^*\rightarrow \N^*$ is called quasi-open if, and only if, 
the image of every nonempty open set under $\psi$ has nonempty interior.
\end{definition}

\begin{proposition} \label{locallydet=quasi-open}
Suppose that $\psi: \N^*\rightarrow \N^*$ is a  continuous map.
Then $T_\psi$ is fountainless if, and only if, $\psi$ is quasi-open.
\end{proposition}
\begin{proof}

 Notice that for every $y\in\N^*$  we have $\dual{T_\psi}(\delta_y)=\delta_{\psi(y)}$. 
Notice also that for every subset $X\subseteq\N^*$ the following holds
$$|\delta_{\psi(y)}|(X)\neq 0 \qquad\text{iff} \qquad\psi(y)\in X\qquad\text{iff}\quad y\in\psi^{-1}[X].$$
Therefore, if $\psi$ is quasi-open and $F\subseteq \N^*$ is nowhere dense, we have that 
 $\{y\in\N^*:|\dual{T_\psi}(\delta_y)|(\N^*\setminus F)= 0\}=\psi^{-1}[F]$ is nowhere dense, 
and so $T_\psi$ is fountainless.
On the other hand if $T_\psi$ is fountainless, consider $\psi[U]$ where $U$ is open.
If $\psi[U]$ were nowhere dense, then 
$\{y\in\N^*:|\dual{T_\psi}(\delta_y)|(\N^*\setminus \psi[U])= 0\}=\psi^{-1}[\psi[U]]$
would be nowhere dense, which contradicts the fact that $U\subseteq\psi^{-1}[\psi[U]]$.

\end{proof}

\begin{proposition}\label{funnelless=nwdp}
Let $\psi: \N^*\rightarrow \N^*$ be a  continuous map.
Then $T_\psi$ is funnelless if, and only if, $\psi$ sends nowhere dense sets into
nowhere dense sets.
\end{proposition}
\begin{proof} 
Suppose $T_\psi$ is funnelless and let $F\subseteq\N^*$ be nowhere dense. 
Let $G\subseteq\N^*$ be nowhere dense such that $\dual{T_\psi}(\delta_y)$ is 
concentrated on $G$ for every $y\in F$.
Then, as in the proof of \ref{locallydet=quasi-open}, we have that 
 $F\subseteq\{y\in\N^*:|\dual{T_\psi}(\delta_y)|(\N^*\setminus G)= 0\}=\psi^{-1}[G]$.

Now suppose $\psi$ sends nowhere dense sets into nowhere dense sets and let 
 $F\subseteq\N^*$ be nowhere dense. Then, 
$F\subseteq\psi^{-1}[\psi[F]]=\{y\in\N^*:|\dual{T_\psi}(\delta_y)|(\N^*\setminus \psi[F])= 0\}$,
which means that $\dual{T_\psi}(\delta_y)$ is concentrated on $\psi[F]$, for
 every $y\in F$.
\end{proof}

\begin{definition}\label{nonatomic}
An operator $T:\ell_\infty/c_0\rightarrow \ell_\infty/c_0$ is nonatomic if, and only if,
for every $y\in \N^*$ the measure $\dual{T}(\delta_y)$ is nonatomic.
\end{definition}

\begin{proposition}\label{nonatomiclocallynull} Every positive nonatomic operator 
on $\ell_\infty/c_0$ is locally null.
\end{proposition}
\begin{proof}
Since for every 
$y\in \N^*$ the measure $\dual{T}(\delta_y)$ has no atoms,
 by the regularity of $\dual{T}(\delta_y)$ and the compactness of $\N^*$
we may find for each $n\in\N$ a finite open covering $(U_i(y,n))_{i<j(y,n)}$
of $\N^*$ by clopen sets such that $|\dual{T}(\delta_y)|(U_i(y,n))<1/2(n+1)$ holds
for all $i<j(y,n)$. 

Now  by the weak$^*$ continuity of $\dual{T}$, we may choose for each $n\in\N$ an
open neighbourhood $V_n(y)$ of $y$ such that for all $z\in V_n(y)$ we have
$$|\dual{T}(\delta_z)|(U_i(y,n))=|\dual{T}(\delta_z)(U_i(y,n))|<1/2(n +1)$$
for all $i<j(y,n)$. The first equality follows from the hypothesis that $T$ is positive.

We have thus constructed for each $n\in\N$ an open covering $\{V_n(y):y\in \N^*\}$ of $\N^*$.
By the compactness of $\N^*$, for each $n \in\N$ take $y_0(n), ..., y_{m(n)-1}(n)\in \N^*$
such that 
$$\N^*\subseteq \bigcup_{l< m(n)}V_n(y_l).$$

Now consider the family $\mathcal A$ of those sets $A\subseteq \N$
such that given $n\in \N$,  for each $l< m(n)$ there is
$i<j(y_l,n)$ such that $A^*$ is included in $U_i(y_l,n)$.
As in the proof of \ref{nonempty}, it is easy to see that 
 $\mathcal A$ is dense and that for every $z\in\N^*$ and every
$A\in\mathcal A$ we have that $|\dual{T}(\delta_z)|(A^*)=0$. 
Therefore if $f\in C(\N^*)$ is $A^*$- supported we have $T(f)=0$, as required.

\end{proof}

\section{Operators on $\ell_\infty/c_0$ and operators on $\ell_\infty$}

\subsection{Operators induced by operators on $\ell_\infty$}

\begin{definition} Suppose that $R: \ell_\infty\rightarrow \ell_\infty$ is a
linear operator which preserves $c_0$, then $[R]: \ell_\infty/c_0\rightarrow \ell_\infty/c_0$
is a linear operator defined by
$$[R]([f]_{c_0})=[R(f)]_{c_0},$$
for every $f\in\ell_\infty$.
If $T: \ell_\infty/c_0\rightarrow \ell_\infty/c_0$ is a linear operator, then a 
lifting $R:\ell_\infty\rightarrow\ell_\infty$ is any linear operator such that 
$[R]=T$.  
\end{definition}

Note that our terminology is slightly different than the one used in the literature concerning the
trivialization of endomorphisms of $\wp(\N)/\text{\emph{Fin}}$. This is due to the fact that we do not use
nonlinear liftings of linear operators.

\begin{lemma} 
Let $R_0,R_1:\ell_\infty\rightarrow\ell_\infty$ be linear
operators which preserve $c_0$. Then,
\begin{enumerate}\label{combinations'}
\item $[R_0+\alpha R_1]=[R_0]+\alpha [R_1]$, for every real $\alpha$.
\item $[R_1\circ R_0]=[R_1]\circ [R_0]$.
\end{enumerate}
\end{lemma}
\begin{proof}
Fix $f\in\ell_\infty$. Then, 
$$[R_0+\alpha R_1)]([f]_{c_0})=[(R_0+\alpha R_1)(f)]_{c_0}=[R_0(f)]_{c_0}+\alpha[R_1(f)]_{c_0}=([R_0]+\alpha [R_1])([f]_{c_0})$$
$$\text{and }\ [R_1]\circ [R_0]([f]_{c_0})=[R_1]([R_0(f)]_{c_0})=[R_1(R_0(f))]_{c_0}=[R_1\circ R_0]([f]_{c_0}).$$
\end{proof}

\begin{proposition}\label{weaklycompact'}
 Let $R_0,R_1:\ell_{\infty}\rightarrow \ell_{\infty}$ be linear operators which preserve $c_0$.
Then,
\begin{enumerate}
 \item If  $[R_0]=0$ then,  $R$ is weakly compact.
\item If  $[R_0]=[R_1]$ then, $R_0-R_1$ is weakly compact.
\end{enumerate} 
\end{proposition}

\begin{proof}
$[R_0]=0$ means that the image of $R_0$ is included in $c_0$. However,
$\ell_\infty$ is a Grothendieck space and all operators
from such spaces into separable spaces are weakly compact (Theorem 1 of \cite{diestelgrothen}).
For part (2)  apply \ref{combinations'} and part (1) to $R_0-R_1$.
\end{proof}

So there  could be many liftings of the same operator but they all differ by
a weakly compact perturbation. When we look at $\ell_\infty/c_0$ as $C(\N^*)$, then liftings correspond
to extensions.

\begin{lemma}\label{*lifting} 
Let $T:C(\N^*)\rightarrow C(\N^*)$ be liftable to
 $R:C(\beta\N)\rightarrow C(\beta\N)$. Then, for every $y\in \N^*$ we have
$$\dual{R}(\delta_y)|\N^*=\dual{T}(\delta_y).$$
\end{lemma}
\begin{proof} 
If $Q:C(\beta\N)\rightarrow C(\N^*)$ is the restriction map, then
the dual of the lifting relation $T\circ Q=Q\circ R$ is
$\dual{Q}\circ \dual{T}=\dual{R}\circ \dual{Q}$. Notice that 
$\dual{Q}$ acts on measures on $\N^*$
by extending them to $\beta\N$ with $\N$ having measure zero.  
 So for every $y\in\N^*$ we have 
$\dual{T}(\delta_y)=(\dual{Q}\circ \dual{T})(\delta_y)|\N^*=
(\dual{R}\circ \dual{Q})(\delta_y)|\N^*=\dual{R}(\delta_y)|\N^*$.
\end{proof}

\subsection{Local properties of liftable operators on $\ell_\infty/c_0$}

\begin{proposition}
\label{fallinglocallynull} 
If $R: C(\beta\N)\rightarrow C(\beta\N)$ 
is a positive falling operator, then the  operator $[R]: C(\N^*)\rightarrow C(\N^*)$
is nonatomic and locally null. 
%If in addition $R$ is positive, then $R'$ is locally null.
\end{proposition}
\begin{proof}
By the definition (\ref{falling}), given $\varepsilon>0$ we have a cofinite set $B\subseteq \N$
and a partition $\{A_1, ..., A_k\}$ of $\N$ such that 
$$\dual{R}(\delta_i)(\beta A_m)=|\dual{R}(\delta_i)|(\beta A_m)<\varepsilon$$
for every $m\leq k$ and every $i\in B$. As any $\delta_y$, for $y\in \N^*$, is in the weak$^*$
closure of $\{\delta_n:n\in B\}$, it follows 
 by the weak$^*$ continuity of $\dual{R}$  that $\dual{R}(\delta_y)(\beta A_m)< \varepsilon$,
for every $y\in\N^*$ and every $m\leq k$.
 But by  \ref{*lifting} and the positivity of $R$ we have $\dual{[R]}(\delta_y)(A_m^*)=\dual{R}(\delta_y)(A_m^*)
\leq \dual{R}(\delta_y)(\beta A_m)< \varepsilon$, 
for every every $y\in\N^*$ and every $m\leq k$. %\textbf{[[ONLY IF $R$ is positive]]}.
 As $\{A_1^*, ..., A_k^*\}$ is a partition of $\N^*$, 
we conclude that $\dual{[R]}(\delta_y)$ is nonatomic for every $y\in \N^*$.
By \ref{nonatomiclocallynull}, $[R]$ is locally null.

\end{proof}

\begin{corollary}\label{matrixnowheretrivial}
There is a matrix operator $T$ which has  fountains and such that
whenever $T\circ I_A\not=0$, we have that $T\circ I_A$ is not canonizable along any continuous map.
In particular $T$ is nowhere trivial.
\end{corollary}
\begin{proof} 
The operator $R$ from \ref{densityoperator} is a non-weakly compact, positive, falling operator
on $\ell_\infty$. Its range is not included in $c_0$
by \ref{rangeinc_0} (actually, the characteristic function of a subset
of $\N$ of positive density is sent to an element not in $c_0$). So $T=[R]\not=0$.
On the other hand, by \ref{fallinglocallynull} we know that $T$ is locally null, 
so by \ref{locallynull-det-null} it follows that $T$ has fountains.

Now note that by \ref{fallinglocallynull} we have that 
 $\dual{(T\circ I_A)}(\delta_y)= \dual{T}(\delta_y)|A^*$ is nonatomic 
or zero for every $y\in \N^*$, so the second part of the corollary follows.
\end{proof}

%\begin{proposition}\label{matrixnowheretrivial}
%There is a matrix operator $T:\ell_\infty/c_0 \rightarrow \ell_\infty/c_0$
%which is nowhere trivial.
%\end{proposition}
%\begin{proof}
% Let $R:\ell_\infty\rightarrow\ell_\infty$ be given by the matrix 
%defined by 
%$$b_{ij}=\left\{\begin{array}{ll}
 %               1/(i+1), & \text{ if }\sum_{k<i}k< j\leq\sum_{k<i+1}k\\
%		0, & \text{ otherwise}\\
%                \end{array}
%\right.$$
%for every $i\in\N$. Let $T=[R]$. Let $A,B\subseteq\N$ be infinite and suppose there exist
%a bijection $\sigma:B\rightarrow A$ and a nonzero $r\in\mathbb R$ such that 
%if $supp(f)\subseteq A$ then $R(f)|B-rf\circ \sigma\in c_0(B)$. 

%By \ref{matrixweaklycompact} and since all the terms of the matrix are positive, 
%if $(\sum_{j\in A}b_{ij})_{i\in B}\in c_0(B)$
%we have that $R(f)|B\in c_0$, for all $A$-supported $f$. But this is impossible
%because $r\chi_A\circ \sigma\equiv r\neq 0$. 
%So $B'=\{i\in B:\sum_{j\in A}b_{ij}\neq 0\}$ is infinite. 
%Choose for every $i\in B'$ an element $j_i\in A$ such that $\sum_{k<i}k< j_i\leq\sum_{k<i+1}k$.
%Let $f$ be the characteristic function of $A'=\{j_i:i\in B'\}$. Then, it is clear that
%$R(f)\in c_0$. However, $rf(\sigma(j))=r\neq 0$ for every $j\in\sigma^{-1}(A')$,
%a contradiction because $A'$ is infinite. 

%\end{proof}

\begin{proposition}\label{antimatrixlocallynull} 
If $R:\ell_\infty\rightarrow\ell_\infty$ 
is an antimatrix operator, then the  operator $[R]: \ell_\infty/c_0\rightarrow \ell_\infty/c_0$
factors through $\ell_\infty$ and so is locally null.
\end{proposition}
\begin{proof}

Let $\mu_n=\dual{R}(\delta_n)$. Since $R$ is antimatrix (Definition \ref{antimatrix})  we may consider $\mu_n$
as a measure on $\N^*$. 
Consider $S:\ell_\infty/c_0\rightarrow \ell_\infty$ given by $S([f]_{c_0})=(\mu_n(\beta f))_{n\in \N^*}$
for every $f\in\ell_\infty$, 
and $Q:\ell_\infty\rightarrow \ell_\infty/c_0$ the quotient map. $S$ is well-defined
since the measures $\mu_n$ are null on $\N$.
We have for every $f\in\ell_\infty$ that 
$$(Q\circ S)([f]_{c_0})=[(\mu_n(\beta f))_{n\in\N}]_{c_0}=[R(f)]_{c_0}=[R]([f]_{c_0}),$$
so $(Q\circ S)$ is $[R]$. To conclude that $[R]$ is locally null use \ref{idealsinclusions}.
\end{proof}

\begin{theorem}\label{canonembedding}

If $T:\ell_{\infty}/c_0\rightarrow\ell_{\infty}/c_0$ is a matrix operator which is an isomorphic
embedding, then it is right-locally trivial.
 
 \end{theorem}

\begin{proof}

Let $R:\ell_{\infty}\rightarrow\ell_{\infty}$ be 
given by a $c_0$-matrix such that $[R]=T$.
% Assuming that  $T$ is bounded-below, we will
% show that  for all infinite $\tilde{A}\subseteq \N$ there exist 
% infinite sets $A\subseteq \tilde{A}$ and $B\subseteq \N$,  a bijection
% $\sigma:B\rightarrow A$, and $r\neq 0$ such that 
% $$T(f^*)| B^*=rf^*\circ \sigma_*,\quad\forall f^*\in C(A^*).$$
Let $(b_{ij})_{i,j\in\N}$ be the matrix corresponding to $R$. Let $M>0$ be such 
that $\|T([f]_{c_0})\|\geq M\|[f]_{c_0}\|$, for every $f\in \ell_\infty\setminus c_0$.
Notice that this condition  is equivalent to the statement that
$\limsup_{n\rightarrow\infty}|R(f)(n)|\geq M$, for every $f\in\ell_{\infty}$ such that
$\limsup_{n\rightarrow\infty}|f(n)|=1$.

Fix an infinite $\tilde{A}\subseteq\N$.

 \textsc{Claim:} $\lim_{i\rightarrow\infty}\max\{|b_{ij}|:j\in\tilde{A}\}\neq 0$.

Assume otherwise. We will construct an $f\in\ell_{\infty}$ such that
$\limsup_{n\rightarrow\infty}|f(n)|=1$ and $\limsup_{n\rightarrow\infty}|R(f)(n)|< M$.

Let $m_i=\min\{k\in\N:\sum_{j\geq k}|b_{ij}|<1/(i+1)\}$, for every $i\in\N$.
We shall construct  by induction two strictly increasing sequences of integers 
$(i_n)_{n\in\N}$ and $(j_n)_{n\in\N}$, with $j_n\in\tilde{A}$ for every $n\in\N$.
Let $i_0=0$ and $j_0=\min \tilde{A}$. If we have constructed $i_l$, $j_l$, for $l\leq n$,
 take $i_{n+1}>i_n$ such that $\max\{|b_{ij}|:j\in\tilde{A}\}<\frac{1}{(n+2)^2}$,
for every $i\geq i_{n+1}$; take $j_{n+1}\in\tilde{A}$ such that 
$j_{n+1}>\max\{m_l:l<i_{n+1}\}$ and $j_{n+1}>j_n$.

Now let $f$ be the characteristic function of $\{j_n:n\in\N\}$ and 
let $N\in\N$ be such that $\frac{N}{(N+1)^2}<M/4$ and $1/N<M/4$.
Fix $k\geq i_N$. Then, $k\geq N$ and also, $i_n\leq k<i_{n+1}$, for some $n\geq N$. Consider the 
following:

$$\begin{array}{rcl}
  |R(f)(k)|=|\sum_{j\in\N}b_{kj}f(j)|&\leq&\sum_{j<m_k}|b_{kj}f(j)|+\sum_{j\geq m_k}|b_{kj}|\\
&\leq&\sum_{j<m_k}|b_{kj}f(j)|+1/k\\
&\leq&\sum_{l<n}|b_{kj_l}|+1/N \qquad(\text{because } j_{n+1}>m_k)\\
&\leq& n\cdot\max\{|b_{kj}|:j\in\tilde{A}\}+1/N\\
&\leq&\frac{n}{(n+1)^2}+1/N\qquad(\text{because }k\geq i_n)\\
&<&M/2\end{array}$$

This contradicts the definition of $M$, and so the claim is proved.\\

Let $\delta>0$ and let $B_0\subseteq \N$ be infinite such that 
$\max\{|b_{ij}|:j\in\tilde{A}\}>\delta$, for every $i\in B_0$.

We shall construct by induction three strictly increasing sequences of integers,
$(i_n)_{n\in\N}$, $(j_n)_{n\in\N}$, $(k_n)_{n\in\N}$, satisfying the following
for every $n\in\N$:
\begin{enumerate}
 \item $|b_{i_nj_n}|>\delta$
\item $j_n\in \tilde{A}$
\item $k_n\leq j_n<k_{n+1}$
\item $\sum_{j<k_n}|b_{i_nj}|<\frac{1}{2(n+1)}$
\item $\sum_{j\geq k_{n+1}}|b_{i_nj}|<\frac{1}{2(n+1)}$
\end{enumerate}

Let $k_0=0$ and $i_0=\min( B_0)$. Let $j_0\in\tilde{A}$ be such that
$|b_{i_0j_0}|>\delta$. Let $k_1>j_0$ be such that $\sum_{j\geq k_1}|b_{i_0j}|<1$.

Assume we have constructed $i_l$, $j_l$ and $k_{l+1}$, satisfying 1--5 for every
$l\leq n$. Let $N$ be such that $\sum_{j<k_{n+1}}|b_{ij}|<\min\{\delta,\frac{1}{2(n+2)}\}$,
for every $i\geq N$ (it exists because it is a $c_0$-matrix). %by Lemma \ref{matrix_partialsum_converge}).
Let $i_{n+1}\in B_0\setminus N$. Let $j_{n+1}\in\tilde{A}$ be such that
$|b_{i_{n+1}j_{n+1}}|>\delta$ (it exists because $i_{n+1}\in B_0$).
Notice that $j_{n+1}\geq k_{n+1}$ because $|b_{i_{n+1}j}|<\delta$, for every
$j< k_{n+1}$. Let $k_{n+2}>j_{n+1}$ be such that $\sum_{j\geq k_{n+2}}|b_{i_{n+1}j}|<\frac{1}{2(n+2)}$.
This ends the inductive construction.

Now, $\delta<|b_{i_nj_n}|\leq\sup\{|b_{ij}|:i,j\in\N\}$, for every $n\in\N$.
Therefore, by going to a subsequence we may assume that $b_{i_nj_n}$
converges to some $r$ with $|r|\geq \delta$. 
Let $A=\{j_n: n\in\N\}$ and $B=\{i_n: n\in\N\}$. Let 
$\sigma:B\rightarrow A$ be given by $\sigma(i_n)=j_n$, for each $n\in\N$.

\textsc{Claim:} 
$(P_B\circ T\circ I_A)([f]_{c_0(A)})=[rf\circ \sigma]_{c_0(B)}$, for every $f\in\ell_\infty(A)$.

Note that what we need to show is that 
$\lim_{n\rightarrow\infty}|R(f)(i_n)-rf(\sigma(i_n))|=0$, for every $f\in\ell_\infty(A)$.
So fix $f\in\ell_{\infty}(A)$
and fix an arbitrary $\varepsilon>0$. Let $M'$ be such that $\|\dual{T}(\delta_n)\|\leq M'$, 
for every $n\in\N$ (it exists by definition of $c_0$-matrix).
Let $N_0$ be such that 
$|b_{i_nj_n}-r|<\frac{\varepsilon}{3\|f\|}$, for all $n\geq N_0$, and let $N_1$ be such that 
$1/(N_1+1)<\frac{\varepsilon}{3\|f\|}$. %Since $supp(f^*)\subseteq A^*$, we have
% that $\lim_{\stackrel{n\rightarrow\infty}{n\notin A}}f(n)\rightarrow 0$,
% so there exists $N_2$ such that $|f(n)|<\frac{\varepsilon}{3M'}$, 
% for all $n\notin A$, $n\geq N_2$. 
Then, for every $n \geq N_0+N_1$ we have

$$\begin{array}{rcl}
   |R(f)(i_n)-rf(\sigma(i_n))|&=&|\sum_{j\in\N}b_{i_nj}f(j)-rf(j_n)|\\
&\leq&\sum_{j<k_n}|b_{i_nj}f(j)|+\sum_{\stackrel{k_n\leq j<k_{n+1}}{j\neq j_n}}|b_{i_nj}f(j)|\\
&&\qquad\qquad+\sum_{j\geq k_{n+1}}|b_{i_nj}f(j)|+|b_{i_nj_n}f(j_n)-rf(j_n)|\\
&<&\frac{\|f\|}{2(n+1)}+0%\sum_{\stackrel{k_n\leq j<k_{n+1}}{j\neq j_n}}|b_{i_nj}|\frac{\varepsilon}{3M'}
+\frac{\|f\|}{2(n+1)}+\|f\||b_{i_nj_n}-r|\\
&<&\frac{\|f\|}{N_1+1}+%M'\frac{\varepsilon}{3M'}+
\|f\|\frac{\varepsilon}{3\|f\|}\\
&<&\varepsilon/3+\varepsilon/3%+\varepsilon/3=
<\varepsilon.\end{array}$$

This concludes the proof.

\end{proof}

\begin{corollary}\label{corcanonembedding}
Every %right-locally (somewhere) 
liftable isomorphic embedding $T:\ell_\infty/c_0\rightarrow\ell_\infty/c_0$
is right-locally %(somewhere) 
trivial.
\end{corollary}
\begin{proof} 
Since $T$ is liftable, there exist $R_0,R_1:\ell_\infty\rightarrow\ell_\infty$
an antimatrix operator and one given by a $c_0$-matrix, respectively,
such that $T=[R_0+R_1]=[R_0]+[R_1]$. Fix an infinite $A\subseteq\N$.
By \ref{antimatrixlocallynull}, take an infinite $A_0\subseteq A$ such that 
$T\circ I_{A_0}=[R_1]\circ I_{A_0}$. Then, $[R_1]\circ I_{A_0}$ is a matrix operator 
which is an isomorphic embedding, so by \ref{canonembedding} there exist 
infinite $A_1\subseteq A_0$ and $B\subseteq\N$
such that $P_B\circ T\circ I_{A_1}=P_B\circ [R_1]\circ I_{A_1}$ is trivial.
\end{proof}

\begin{corollary}\label{liftableembedding->somewherematrixiso}
 Every liftable isomorphic embedding $T:\ell_\infty/c_0\rightarrow\ell_\infty/c_0$
is right-locally an isomorphic matrix operator.
\end{corollary}
\begin{proof}
By \ref{corcanonembedding} it suffices to recall that a trivial operator
 is an isomorphic matrix operator.
\end{proof}

\begin{corollary}\label{locallynullsum} 
Let $\PP$ be one of the following properties: 
isomorphically liftable, 
isomorphically matrix, 
trivial, 
canonizable along $\psi$.
Suppose that  $S: \ell_\infty/c_0\rightarrow \ell_\infty/c_0$
is locally null. If $T: \ell_\infty/c_0\rightarrow \ell_\infty/c_0$ is 
right-locally $\PP$ (left-locally $\PP$, somewhere $\PP$), then  $S+T$ is right-locally
$\PP$ (left-locally $\PP$, somewhere $\PP$).
\end{corollary}
\begin{proof} 
First we will note that if the localization $T_{B, A}$ of $T$ to $(A, B)$ 
has $\PP$, then for every infinite $A'\subseteq A$ there is an infinite
$B'\subseteq B$ such that  the localization $T_{B', A'}$ of $T$ to $(A', B')$ 
has $\PP$.

In the case where $T_{B, A}$ is isomorphically liftable, by \ref{corcanonembedding}
it is enough to notice that a trivial operator is isomorphically liftable.
Similarly, if $T_{B, A}$ is isomorphically matrix, by \ref{canonembedding}
it is enough to notice that a trivial operator is isomorphically matrix.

If $T_{B, A}$ is trivial, it is enough to take $B'=\sigma^{-1}[A']$, where
$\sigma:B\rightarrow A$ is the bijection witnessing the triviality of $T_{B, A}$.
Similarly, if $T_{B, A}$ is canonizable along $\psi$, we take $B'\subseteq B$
such that $(B')^*=\psi^{-1}[(A')^*]$.

Now, given a localization $T_{B, A}$ with property $\PP$, take an infinite 
$A'\subseteq A$ such that $S\circ I_{A'}=0$. By the above, there exists $B'\subseteq B$
such that $T_{B', A'}=(S+T)_{B', A'}$ has $\PP$. 

\end{proof}

If we do not assume that the operator is bounded below, then there is no hope
of obtaining local trivialization anywhere:

\begin{proposition}\label{liftablenonmatrix} There is a surjective operator $T: \ell_\infty/c_0\rightarrow\ell_\infty/c_0$
which is globally liftable but is nowhere a nonzero matrix operator.
\end{proposition}
\begin{proof}
Let $(x_n)_{n\in\N}$ be a discrete sequence of nonprincipal ultrafilters
and consider the typical antimatrix operator $R:\ell_\infty\rightarrow\ell_\infty$
given by $R(f)=((\beta f)(x_n))_{n\in \N}$. Let $T=[R]$. By \ref{locallynullsurjective}
we know that $T$ is surjective. Suppose for some infinite $A, B\subseteq\N$ 
there is $S:\ell_\infty(A)\rightarrow\ell_\infty(B)$
given by a $c_0$-matrix and such that $[S]=T_{B,A}$. 
Let us denote by $R_{B,A}$ the operator which maps $f\in\ell_\infty(A)$ into 
$R(f\cup 0_{\N\setminus A})|B$. By \ref{weaklycompact'}
we have that $S-R_{B,A}$ is weakly compact, and since $R$ is an antimatrix operator
we have that $R|c_0=0$, so $S|c_0(A)$ is weakly compact.
Therefore, by \ref{matrixweaklycompact} and \ref{doubleadjoint} we have 
that the image of $S$ is included in $c_0(B)$ and so $T_{B,A}=[S]=0$. 

\end{proof}

\subsection{Lifting operators on $\ell_\infty/c_0$}

In the case  of the Boolean algebra $\wp(\N)/\text{\emph{Fin}}$,  any endomorphism which
can be lifted to  a homomorphism of $\wp(\N)$ is induced by
a homomorphism of $\text{\emph{FinCofin}}(\N)$. However, in the case of
$\ell_\infty/c_0$, like for $\ell_\infty$ (\ref{discontinuousauto}),
 there exist automorphisms which are not determined by its values on $c_0$:

\begin{proposition}\label{discontinuouslifting} 
There are liftable operators %$T:  \ell_\infty/c_0\rightarrow \ell_\infty/c_0$
such that all their liftings are discontinuous 
and are not induced by its action on $c_0$, i.e., are not matrix operators. Moreover,
 such operators can be automorphisms of $\ell_\infty/c_0$.
\end{proposition}
\begin{proof}
Let $(A_i)_{i\in \N}$ be a partition of $\N$ into infinite sets. For each $i\in\N$,
let $x_i$ be any nonprincipal ultrafilter such that $A_i\in x_i$.
For a permutation $\sigma:\N\rightarrow\N$ consider the automorphism
$R_\sigma:\ell_\infty\rightarrow\ell_\infty$ from the proof of \ref{discontinuousauto}  which is given by
$$R_\sigma(f)(n)= f(n)-\beta f(x_i)+\beta f(x_{\sigma(i)}),$$
where $i\in \N$ is such that  $n\in A_i$.
Recall that $R_\sigma\circ R_{\sigma^{-1}}=Id_{\ell_\infty}$, so by \ref{combinations'} 
we have that
$[R_\sigma]\circ [R_{\sigma^-1}]=[Id_{\ell_\infty}]=Id_{\ell_\infty/c_0}$. 
It follows that the operators $[R_\sigma]$ are automorphisms of $\ell_\infty/c_0$. 

 Now suppose that $S:\ell_\infty\rightarrow \ell_\infty$ is a 
continuous lifting of $[R_\sigma]$. By \ref{summary_matrices}
the operator  $S$ is given by a $c_0$-matrix, and by \ref{weaklycompact'} 
we have that $S-R_\sigma$ is a weakly compact operator into $c_0$.
Note that $R_\sigma|c_0=Id_{c_0}$, therefore $S|c_0=Id_{c_0}+W$,
where $W:c_0\rightarrow c_0$ is the restriction of $S-R_\sigma$ to $c_0$
 and so is weakly compact. 
By \ref{doubleadjoint}  we have
$$S=\ddual{(S|c_0)}=\ddual{Id_{c_0}}+\ddual{W}=Id_{\ell_\infty}+\ddual{W},$$
and so $\dual{S}=Id_{M(\beta\N^*)}+U$, where $U$ is weakly compact by the
Gantmacher theorem. 
Therefore,  $Id_{M(\beta\N^*)}+U-\dual{R_\sigma}$ is a weakly compact operator,
and so is $Id_{M(\beta\N^*)}-\dual{R_\sigma}$.
We will show that this is impossible by showing that the bounded sequence 
of measures $(\delta_{x_i})_{i\in \N}$ is not mapped onto a relatively weakly compact set.

A simple calculation gives that 
$\dual{R_\sigma}(\delta_{x})=\delta_x-\delta_{x_i}+\delta_{x_{\sigma(i)}},$
if $x\in A_i^*$. It follows that
$\dual{R_\sigma}(\delta_{x_i})=\delta_{x_{\sigma(i)}}$, for every $i\in\N$.
So $(Id_{M(\beta\N^*)}-\dual{R_\sigma})(\delta_{x_i})=\delta_{x_i}-\delta_{x_{\sigma(i)}}$,
which by the Dieudonne-Grothendieck theorem implies that $Id_{M(\beta\N^*)}-\dual{R_\sigma}$ 
is not weakly compact unless
$\sigma$ moves only finitely many $i\in \N$, as 
the sequence $(x_i)_{i\in\N}$ is discrete.

\end{proof}

 Unlike in the case of the algebra $\wp(\N)/\text{\emph{Fin}}$, nonliftable
automorphisms of $\ell_\infty/c_0$ exist in ZFC. Before proving this
we need one:

\begin{lemma}\label{R'notweaklycompact}
Suppose  $R:\ell_\infty\rightarrow \ell_\infty$ preserves $c_0$. 
% and $S=R':\ell_\infty/c_0\rightarrow\ell_\infty/c_0$.
If  $R$ is not weakly compact, then $[R]$ is not weakly compact either.
\end{lemma}
\begin{proof}
If $R$ is not weakly compact, then there is an infinite $A\subseteq \N$  such that
$R$ restricted to $\ell_\infty^0(A)=\{f\in \ell_\infty: f|(\N\setminus A)=0\}$
 is an isomorphism onto its range (see Prop. 1.2. from \cite{rosenthal1970} and Corollary VI--17 of \cite{diestel-ulh}).
 Consider $X=R^{-1}[c_0]$, a closed subspace 
of $\ell_\infty$ containing $c_0$. Note that $X\cap \ell_\infty^0(A)$ is separable as
$R[X\cap \ell_\infty^0(A)]\subseteq c_0$ and $R$ is an isomorphism on $\ell_\infty^0(A)$.
By the standard argument using the Stone-Weierstrass theorem with respect to simple functions
one can find a countable Boolean algebra $\mathfrak B$ of subsets of $A$ such that
$X\cap \ell_\infty^0(A)$ is included in the closure of the span of $\{\chi_B: B\in\mathfrak B\}$. 

% This closure is a Banach algebra which is isometric to the space of 
% real-valued continuous functions over the Stone space of $\mathcal B$. 
 Let $(D_\xi)_{\xi<\omega_1}$ be a family  of 
pairwise almost disjoint infinite subsets of $A$. For each $\xi<\omega_1$ 
take $x\in D_\xi^*$ and let $E_\xi$ be infinite such that 
$E_\xi^*\subseteq \bigcap\{B^*:B\in\mathfrak B\cap x\}\cap
\bigcap\{\N^*\setminus B^*:B\in\mathfrak B\setminus x\}$ (it exists by \ref{gdelta}
and because $\mathfrak B$ is countable). 
Now take  $u_\xi, v_\xi\in E^*_\xi$ distinct. It follows that no element of 
$\mathfrak B$ separates any of the pairs $(u_\xi, v_\xi)$. Therefore, $\beta f(u_n)=\beta f(v_n)$  for every
$f\in X\cap \ell_\infty^0(A)$. 

For every $\xi<\omega_1$ choose $g_\xi\in\ell_\infty^0(A)$ with support in $D_\xi$ such that
$\|g_\xi\|=1$, $g_\xi(u_\xi)=1$ and $g_\xi(v_\xi)=-1$. 
Notice that $R(g_\xi)\notin c_0$, for every $\xi<\omega$. This implies that 
$\|[R(g_\xi)]_{c_0}\|>0$ for every $\xi<\omega_1$, so there exists $n\in\N$
such that for infinitely many $\xi<\omega_1$ we have $\|[R(g_\xi)]_{c_0}\|>1/n$.
Since the 
$g_\xi^*$ are pairwise disjoint, $[R]$ is not weakly compact 
by Corollary VI--17 of \cite{diestel-ulh}.
\end{proof}

\begin{proposition} Every weakly compact operator on $\ell_\infty/c_0$
with nonseparable range is nonliftable. Such operators exist.
\end{proposition}
\begin{proof}
Suppose that $S:\ell_\infty/c_0\rightarrow\ell_\infty/c_0$ is weakly compact with nonseparable range
and $R:\ell_\infty\rightarrow \ell_\infty$  is such that $[R]=S$ and $R$ preserves $c_0$.
$R$ must be weakly compact itself by \ref{R'notweaklycompact}.
In particular, the image of the unit ball under $R$ is weakly compact. Since weakly compact
subsets of $\ell_\infty$ are norm separable (Corollary. 4.6 of \cite{rosenthalacta}), we have that
 the image of $R$ is separable. But this implies that the image of
$[R]=S$ is separable as well, contradicting the hypothesis.

Now we construct a weakly compact operator $S:\ell_\infty/c_0\rightarrow \ell_\infty/c_0$
with nonseparable range which is weakly compact. 
The construction is based on the fact that $\ell_\infty/c_0$ contains an isometric copy of $\ell_2(2^\omega)$.
This follows from the result of Avilés in \cite{antoniol2} which states that the unit ball in
$\ell_2(2^\omega)$ with the weak topology (equivalently weak$^*$ topology)  is a continuous image of $A(2^\omega)^\N$,
where $A(2^\omega)$ is the one point compactification of the discrete space of size $2^\omega$.
On the other hand, by Theorem 2.5 and Example 5.3. in \cite{bell} we
have that $A(2^\omega)^\N$ is a continuous image of $\N^*$. Hence $C(B_{\ell_2(2^\omega)})$ embeds
isometrically into $C(\N^*)$ and so does $\ell_2(2^\omega)$.
So let $S_1:\ell_2(2^\omega)\rightarrow \ell_\infty/c_0$ be an isomorphism onto its range.

To complete the construction, it is enough to take a surjective operator
 $S_2: \ell_\infty/c_0\rightarrow \ell_2(2^\omega)$ and consider
 $S=S_1\circ S_2$. This is because any 
operator into a reflexive Banach space is  weakly compact (Corollary VI.4.3 of \cite{dunford})
and  weakly compact operators form a two sided ideal (Theorem VI.4.5 of \cite{dunford}).

 The existence such of a surjective operator follows from
the complementation of $\ell_\infty$ in $\ell_\infty/c_0$ and the existence of a surjective operator
 $T: \ell_\infty\rightarrow \ell_2(2^\omega)$ which was proved in \cite{rosenthalquasi}
Proposition 3.4.  and remark 2 below it.  It is based on  a construction 
of an isomorphic copy of $\ell_2(2^\omega)$ inside $\dual{\ell_\infty}$ (proposition 3.4
of \cite{rosenthalquasi}). Once we have an isomorphic embedding $T:\ell_2(2^\omega)\rightarrow
\dual{\ell_\infty}$ we consider 
$$\dual{T}\circ J: \ell_\infty\rightarrow\ddual{\ell_\infty}\rightarrow\dual{\ell_2(2^\omega)},$$
where $J:\ell_\infty\rightarrow \ddual{\ell_\infty}$ is the canonical embedding. 
We have that $\dual{(\dual{T}\circ J)}=\dual{J}\circ \ddual{T}=T$ using the reflexivity of $\ell_2(2^\omega)$
to identify it with $\ddual{\ell_2(2^\omega)}$. But $T$ is one-to-one with closed
range, so $\dual{T}\circ J$ must be onto as required.

\end{proof}

\begin{theorem}\label{nonliftable} 
There is an automorphism $T: \ell_\infty/c_0\rightarrow \ell_\infty/c_0$ 
which cannot be lifted to a linear operator on $\ell_\infty$.
\end{theorem}
\begin{proof}

Consider $T_1=Id+S$ where $S$ is any weakly compact operator on $\ell_\infty/c_0$
from the previous proposition. 
 Since $S$ is strictly singular, $T_1$ is a Fredholm operator
 of Fredholm index $0$ (see Proposition 2.c.10 of \cite{tzafriri}),
i.e., its kernel is finite dimensional of dimension $n$
 and its range is of the same finite codimension $n$. %proposition 2.c.10
Since finite dimensional subspaces of Banach spaces  are complemented we can write %Theorem 5.6 of Fabian et al. Functional analysis and infinite dimensional geometry
$$T_1:Ker(T_1)\oplus X\rightarrow Range(T_1)\oplus Y$$
where $Y$ is of finite dimension $n$ and $X$ has the same finite codimension $n$. 
Let $U: Ker(T_1)\rightarrow Y$ be an isomorphism 
and define $T: Ker(T_1)\oplus X\rightarrow Range(T_1)\oplus Y$  by
$T(z, x)=(T_1(x),U(z))$. It follows that $T$ satisfies
$$T=T_1+U=Id+S+U$$
Having null kernel and being surjective it is an automorphism of $\ell_\infty/c_0$.
Now let us show that $T$ cannot be lifted to an operator on $\ell_\infty/c_0$.
$S+U$ is weakly compact with nonseparable range as a sum of 
an operator with this property and a finite rank operator, so it cannot be lifted by
the previous proposition. Since the sum of two liftable operators is liftable and \emph{Id} is liftable,
it follows that $T$ is an automorphism which cannot be lifted.

\end{proof}
\begin{proposition}\label{noliftingnontrivial} If $T: \ell_\infty/c_0\rightarrow \ell_\infty/c_0$
is canonizable along a homeomorphism $\psi: \N^*\rightarrow \N^*$
and $\psi$ is a nontrivial homeomorphism (i.e., it is not induced by a bijection
of two coinfinite subsets of $\N$), then $T$ is not liftable.
\end{proposition}
\begin{proof} We may assume that $\str{T}=T_\psi$, that is,  the constant
$r$ of Definition \ref{trivializations} is $1$. Suppose $R:\ell_\infty\rightarrow \ell_\infty$ is a lifting of $T$.
Let $R=R_1+R_2$ (see \ref{decompositionmatrix}),  where $R_1$ is an
operator given by a $c_0$-matrix and $R_2$ is an antimatrix operator. 

\textsc{Claim 1:} There cannot exist disjoint functions $f,\, g\in \ell_\infty$ (i.e.,
such that $f^.g=0$)  and $\varepsilon>0$ such that
$\int \beta f \ud\dual{R}_1(\delta_i),\, \int \beta g \ud\dual{R}_1(\delta_i)
>\varepsilon$ for infinitely many $i\in\N$.

Indeed, in such a case, we would find 
an infinite $\tilde{B}\subseteq \N$ such that for every $y\in \tilde{B}^*$, we would have 
$\int \beta f \ud(\dual{R}_1(\delta_y)),\, \int \beta g \ud\dual{R}_1(\delta_y)\geq\varepsilon$.
Since  $T=T_\psi$, for all $y\in \N^*$ we have that $\dual{T}(\delta_y)=\delta_{\psi(y)}$, which implies
that either $\int f^* \ud\dual{T}(\delta_y)=0$ or $\int g^* \ud\dual{T}_1(\delta_y)=0$.
Therefore, we will obtain a contradiction if we can free $\dual{T}$ from
the influence of $\dual{R_2}$ somewhere.
By \ref{matrixweaklycompact} and using an argument as in the proof of \ref{pseudo_diagonal2}, 
 we find an infinite $B\subseteq\tilde{B}$
and pairwise disjoint finite $F_i\subseteq \N$ such that 
$|\dual{R_1}(\delta_i)|(\beta\N\setminus F_i)<\varepsilon/3\max\{\|f\|,\|g\|\}$ for $i\in B$.
This implies that $\int_{F_i} \beta f \ud\dual{R}_1(\delta_i)$,
$ \int_{F_i} \beta g \ud\dual{R}_1(\delta_i)\geq2\varepsilon/3$.
Consider an uncountable almost
disjoint family $\{B_\xi: \xi<\omega_1\}$ of subsets of $B$ and sets 
$A_\xi=\bigcup\{F_i: i\in B_\xi\}$, for $\xi<\omega_1$. We have that 
$\int_{\beta A_\xi} \beta f \ud\dual{R}_1(\delta_i)$,
$ \int_{\beta A_\xi} \beta g \ud\dual{R}_1(\delta_i)
\geq\varepsilon/3$ for each $i\in B_\xi$, since the measures are concentrated on $\N$ (see \ref{antimatrix_charac}).
Now, the sets $A_\xi^*$ are pairwise disjoint and the measures 
$\dual{R_2}(\delta_i)$ are concentrated on $\N^*$, so there is a $\xi_0<\omega_1$
such that $|\dual{R_2}(\delta_i)|(\beta A_{\xi_0})=0$ for all $i\in \N$.
So by \ref{*lifting} and by the weak$^*$ continuity of $\dual{R_1}$, for every $y\in B_{\xi_0}^*$ we have 
$$\int (f^*|A_{\xi_0}^*)\ud\dual{T}(\delta_y)
=\int_{\N^*} (\beta f|\beta A_{\xi_0})\ud(\dual{R_1}(\delta_y)+\dual{R_2}(\delta_y))\geq
\varepsilon/3+0=\varepsilon/3.$$
A similar calculation works for $\beta g|\beta A_{\xi_0}$ which gives the desired
contradiction since the restrictions of disjoint functions are disjoint. So the claim is proved.

Let $(b_{ij})_{i,j\in\N}$ be the matrix of $R_1$, i.e., $\dual{R}_1(\delta_i)=\sum_{j\in \N}b_{ij}\delta_j$
for all $i\in \N$. Let $j_i\in \N$ be such that $b_{ij_i}$ has the biggest absolute value among
the numbers $\{b_{ij}: j\in \N\}$ for all $i\in \N$.

\textsc{Claim 2:} The $b_{ij_i}$'s are separated from $0$. 

Assume otherwise. Then, we can find
an infinite $B\subseteq \N$  such that for $i\in B$ the
numbers  $b_{ij_i}$'s converge to $0$. If the sequence 
 $(\|b_i\|_{\ell_1})_{i\in B}$ is not separated from zero, then by \ref{matrixweaklycompact}
there would be an infinite $B'\subseteq B$ such that the map $f\mapsto R_1(f)|B'$ is
weakly compact 
and so $P_{B'}\circ [R_1]=0$ by \ref{doubleadjoint} and \ref{matrixweaklycompact}. 
This would then imply by \ref{antimatrixlocallynull} that $P_{B'}\circ T$ is locally null, 
which is impossible since $P_{B'}\circ T_\psi$ is an automorphism on $\psi[(B')^*]$-supported functions.
Therefore, the sequence $(\|b_i\|_{\ell_1})_{i\in B}$ is separated from zero.

By \ref{pseudo_diagonal2}, there exist $\delta>0$, an infinite $B_0\subseteq B$ 
and finite $F_i\subseteq \N$ for $i\in B$ which 
are pairwise disjoint and
such that $\sum_{j\in F_i}|b_{ij}|>\delta$ % and $\sum_{j \not \in F_i}|b_{ij}|<\delta/4$
for all $i\in B_0$. Since $b_{ij_i}$'s converge to $0$,  one
can partition  each $F_i$ into $H_i$ and $G_i$ such that
$\sum_{j\in G_i}|b_{ij}|>\delta/4$ and $\sum_{j\in H_i}|b_{ij}|>\delta/4$
for sufficiently large $i\in B$ (construct $G_i$ considering initial fragments $G_i(k)$ of $F_i$,
for $k\leq|F_i|$ starting with
$G_i(0)=\emptyset$ and increasing the previous fragment by one element. Since
the jumps between  $\sum_{j\in G_i(k)}|b_{ij}|$ and $\sum_{j\in G_i(k+1)}|b_{ij}|$
can be at most $|b_{i j_i}|$, which is eventually less than $\delta/4$, we can 
obtain the required $G_i$ and $H_i=F_i\setminus G_i$ at some stage $k\leq|F_i|$).
But then we can define two disjoint
functions, $f$ with support $\bigcup_{i\in B} G_i$ and $g$ with support $\bigcup_{i\in B} H_i$, 
which contradict claim 1. Therefore, the claim is proved.

Now consider the matrix $(c_{ij})_{i,j\in \N}$ such that $c_{ij_i}=b_{i j_i}$,
for $i\in \N$, and all other entries are zero. Write $R_1=R_3+R_4$ where 
$R_3$ is induced by $(c_{ij})_{i,j\in \N}$ 
and $R_4=R_1-R_3$.  If $R_4$ were not weakly compact, we would have
that the norms of its rows do not converge to zero (\ref{matrixweaklycompact}). 
Then, using \ref{pseudo_diagonal2} and an argument analogous to that of claim 2
 we can construct disjoint
 functions which contradict claim 1.
Thus $[R_4]$ must be zero by \ref{doubleadjoint} and \ref{matrixweaklycompact}
 and so $[R_1]=[R_3]$. Therefore, we may assume that $R_1$ is given by a matrix of
a function from $\N$ into $\N$, that is, all entries of the matrix are equal to zero
except for the $b_{ij_i}$'s, which are separated from zero
by some $\delta>0$. 

\textsc{Claim 3:}
There are cofinite sets $A, B\subseteq \N$ such that
$J:B\rightarrow A$ given by $J(i)=j_i$ is a bijection. 

If  $\{j_i:  i\in \N\}$ is coinfinite, say disjoint from an infinite $A\subseteq \N$, then $[R_1]\circ I_A=0$
which together with the fact that $[R_2]$ is locally null (see \ref{antimatrixlocallynull}) 
leads to a contradiction with the fact that $T$ is an automorphism.
Of course $J$ cannot send infinite sets into one value, because it would
give rise to a column of the matrix of $R_1$ which would not be in $c_0$, as the entries 
of the matrix are separated from $0$,
contradicting \ref{charac_matrix}. If there are infinitely many values of $J$ which are assumed
on distinct integers, then there are two disjoint infinite sets $B_1=\{i^1_n: n\in \N\}\subseteq \N$ and
$B_2=\{i^2_n: n\in \N\}\subseteq \N$  such that $j_{i^1_n}=j_{i^2_n}$.
Define $J'(n)=j_{i^1_n}=j_{i^2_n}$ and put $f(i^2_n)=b_{i^1_n J'(n)}/b_{i^2_n J'(n)}$ and otherwise put the value
 of $f$ to be $0$. Note that whenever $A'\subseteq \{J'(n): n\in \N\}=A$, 
% then for every $n\in (J')^{-1}[A']$ 
we have that 
% $R_1(\chi_{ A'})(i^1_n)=b_{i^1_n j_n}$ and $R_1(\chi_{A'})(i^2_n)=b_{i^2_n j_n}$,
% so 
$$\delta_{i^1_n}(R_1(\chi_{ A'}))-f(i^2_n)\delta_{i^2_n}(R_1(\chi_{ A'}))=0,$$
for all $n\in\N$.
% However $A$ must be coinfinite because otherwise (as each column of $R_1$
% has just one non-zero entry, $j_n$)  there would be an infinite $A'\subseteq \N$
% disjoint with $A$ and so $[R_1]\circ I_{A'}=0$ leading to a contradiction since $[R_2]$
% is locally null. So the above holds for any $A'\subseteq A$
%[[\textbf{Maybe it is better to label the equation you refer to. Moreover, I don't see why this is true.}]].
Let $\eta:B_1^*\rightarrow B_2^*$ be the extension of the bijection
from $B_1$ to $B_2$ sending $i^1_n$ to $i^2_n$ for all $n\in \N$. 
It follows that for every $A$-supported $g$ and every
$x\in B_1^*$ we have 
$$(\delta_x-(\beta f)(\eta(x))\delta_{\eta(x)})([R_1]([g]))=0.$$
Find $B_1'\subseteq B_1$ such that $\psi[(B_1')^*\cup \eta[(B_1')^*]]=(A')^*$
for some infinite $A'\subseteq A$  such that $[R_2]\circ I_{A'}=0$ (by \ref{antimatrixlocallynull}).
This can be achieved since $\psi$ is a homeomorphism.
Considering the  values $T_\psi(g^*)$ for $A'$-supported $g$'s
we obtain all functions supported by $B_1'\cup \eta[B_1']$. However, $[R_2]$ on such $g$'s is zero,
so we obtain a contradiction since the values of $[R_1]$ on such functions
have the above restrictions. The claim is proved.

Thus $R_1^*$ is $T_\phi$ where $\phi:\N^*\rightarrow \N^*$ is a trivial
homeomorphism of $\N^*$. Therefore, there is $x\in \N^*$ such that
$\psi^{-1}(x)\not=\phi^{-1}(x)$. It follows that there are infinite  $B_1, B_2, A\subseteq \N$
such that $A^*=\phi[B^*_1]$, $A^*=\psi[B_2^*]$ and $B_1\cap B_2=\emptyset$.
Using \ref{antimatrixlocallynull} take an infinite $A'\subseteq A$ such that $[R_2]\circ I_{A'}=0$.
Then, $T_\phi([\chi_{A'}])$ and $T_\psi([\chi_{A'}])$ %$[R_1]([\chi_{A'}])$ and $T([\chi_{A'}])$
have disjoint supports, so we cannot have $[R_1+R_2]=T$, which completes the proof.

\end{proof}

\section{Canonizing operators acting along a quasi-open mapping}

In  \cite{drewnowski} it was proved that for a linear bounded operator $T$
on $\ell_\infty/c_0$ and  an infinite $A\subseteq \N$ there 
is a real $r\in\R$ and an infinite $B\subseteq A$ such that 
$$T(f)|B^*=rf$$
for every $B$-supported $f$.
This  gives, for example, that if $P_1$ and $P_2$ are complementary projections
on $\ell_\infty/c_0$, then  at least one of them  canonizes as above
for a nonzero $r$, in other words we obtain a local canonization along the identity on $B^*$.
However, a big  disadvantage of this result  is that in general
we cannot guarantee that the constant $r$ is nonzero.
If one works with an automorphism, this kind of result is of no use.
For example, consider an infinite and coinfinite set $D\subseteq \N$ and
the bijection $\sigma: \N\rightarrow \N$ such that  $\sigma[D]=\N\setminus D$, 
$\sigma [\N\setminus D]=D$ and $\sigma^2$ is the identity. Define
an automorphism $T$ of $\ell_\infty/c_0$ by $T([f]_{c_0})=[f\circ \sigma]_{c_0}$.
The above result gives an infinite $B\subseteq D$ such that $T([f]_{c_0})|B=0[f]_{c_0}$
for every $B$-supported $f\in \ell_\infty$, which looses much of the information.
So in this section we embark on finding a surjective $\psi: B^*\rightarrow A^*$ along which
$T$ may canonize with $r$ nonzero as required in Definition \ref{trivializations}.
Note that a potential obstacle for finding such a canonization would be 
if  $\bigcup\varphi^T[B^*]$ were nowhere dense. Actually, we have examples 
such that $\bigcup\varphi^T[\N^*]$ is nowhere dense and $T$ is surjective 
(\ref{locallynullsurjective}, \ref{locallydeterminednwd}).
So it is natural to assume that the surjections we consider are fountainless
and that embeddings are  funnelless. Under these assumptions we obtain a
quasi-open $\psi$ such that $\dual{T}(\delta_x)(\{\psi(x)\})\not=0$ 
holds locally, which is sufficient
for the canonization by the following:

\begin{theorem}\label{quasi-opendrewnowski}
 Let $T:C(\N^*)\rightarrow C(\N^*)$ be a bounded linear operator and let
$\tilde{A}\subseteq \N$ be infinite. 
If $\psi:\N^*\rightarrow\N^*$ is a quasi-open continuous function such that 
$\tilde{A}^*\subseteq \psi[\N^*]$,
then there exist $r\in\mathbb{R}$ and clopen sets $A^*\subseteq\tilde{A}^*$ and 
$B^*=\psi^{-1}[A^*]$ such that 
$T_{B,A}(f^*)=r(f^*\circ \psi)| B^*$,
for every $f\in \ell_\infty(A)$.
\end{theorem}

\begin{proof}
Fix  $\tilde{A}$ and $\psi$ as above.

 \textsc{Claim:} There exists an infinite $A\subseteq \tilde{A}$ and a clopen 
$E^*\subseteq \psi^{-1}[A^*]$ such that for every 
$y\in E^*$ there exists $r_y\in\mathbb{R}$ satisfying 
$$\dual{T}(\delta_y)| A^*=r_y\delta_{\psi(y)}.$$

Suppose this does not hold. We will construct recursively sequences $(A_{\xi})_{\xi<\omega_1}$,
$(D_{\xi})_{\xi<\omega_1}$  and $(E_{\xi})_{\xi<\omega_1}$ of infinite subsets of $\N$, and a sequence 
$(a_{\xi})_{\xi<\omega_1}$ of nonzero reals such that 
\begin{enumerate}
 \item $A_{\eta}\subseteq_* A_{\xi}\subseteq_*\tilde{A}$ and $D_{\xi}\subseteq_* A_{\xi}\setminus A_{\xi+1}$,
for every $\xi<\eta<\omega_1$;
\item $E_{\eta}\subseteq_* E_{\xi}$, for every $\xi<\eta<\omega_1$;
\item $E_{\xi}^*\subseteq \psi^{-1}[A_{\xi}^*]$, for all $\xi<\omega_1$;
\item either $\dual{T}(\delta_y)(D_{\xi}^*)>a_{\xi}>0$ for all $y\in E_{\xi+1}^*$, or
$\dual{T}(\delta_y)(D_{\xi}^*)<a_{\xi}<0$ for all $y\in E_{\xi+1}^*$.
\end{enumerate}

Let $A_0=\tilde{A}$ and $E_0^*=\psi^{-1}[A_0^*]$. Let $\eta<\omega_1$ and 
suppose we have constructed $A_{\xi}$, $D_{\xi}$,
$E_{\xi}$ and $a_{\xi}$
satisfying (1)--(4) for every $\xi<\eta$. 
If $\eta$ is a limit ordinal, take an infinite $E$ such that 
$E\subseteq_* E_{\xi}$ for every $\xi<\eta$. By hypothesis 
there exists a clopen $A_{\eta}^*\subseteq\psi[E^*]$. Put $E_{\eta}^*=\psi^{-1}[A_{\eta}^*]\cap E^*$.
 Now we may suppose
we have $A_{\xi}$ and $E_{\xi}$ for every $\xi\leq\eta$, and $D_{\xi}$ and $a_{\xi}$
for every $\xi<\eta$. 

Take an infinite $A'_{\eta}$ such that 
$(A'_{\eta})^*\subseteq \psi[E_{\eta}^*]$.
By our assumption, there exist $y\in \psi^{-1}[(A'_{\eta})^*]\cap E_{\eta}^*$ and 
$X\subseteq (A'_{\eta})^*\setminus\{\psi(y)\}$ such that $\dual{T}(\delta_y)(X)\neq 0$. 
By the regularity of $\dual{T}(\delta_y)$, there exists
an infinite $D_{\eta}\subseteq_* A'_{\eta}$ such that $\psi(y)\notin D_{\eta}^*$
and $\dual{T}(\delta_y)(D_{\eta}^*)\neq 0$. Let $a_{\eta}$ be such that
either $0<a_{\eta}<\dual{T}(\delta_y)(D_{\eta}^*)$ or $0>a_{\eta}>\dual{T}(\delta_y)(D_{\eta}^*)$.

By the weak$^*$ continuity of $\dual{T}$, %Since $\dual{T}(\delta_y)(D_{\eta}^*)=T(\chi_{D_{\eta}^*})(y)$ and $T(\chi_{D_{\eta}^*})$
%is a continuous function, 
there exists $V$ a clopen neighbourhood of $y$
such that either $\dual{T}(\delta_z)(D_{\eta}^*)>a_{\eta}$ for all $z\in V$, or
$\dual{T}(\delta_z)(D_{\eta}^*)<a_{\eta}$ for all $z\in V$.
Finally, choose $A_{\eta+1}=A'_{\eta}\setminus D_{\eta}$ and  
$E_{\eta+1}^*=\psi^{-1}[A_{\eta+1}^*]\cap V\cap E_{\eta}^*$ (notice that $y\in E_{\eta+1}^*$).
% $\psi(x)\in A^*_{\eta+1}\subseteq (A'_{\eta})^*\cap \psi[V]\setminus D_{\eta}$. 
% \textcolor{red}{[Here we use that $\psi$ is a  homeomorphism.]}
This ends the construction.

Since $|a_{\xi}|>0$ for every $\xi<\omega_1$, there must exist $n\in\N$
and an infinite $I\subseteq\omega_1$ such that $|a_{\xi}|>1/n$, for every $\xi\in I$.
Hence, we may choose $\xi_0, \dots, \xi_{k-1}\in I$, for some  $k\in\N$,
 such that  $a_{\xi_0},\dots,a_{\xi_{k-1}}$ are all of the same sign and
$|\sum_{i<k}a_{\xi_i}|>\|\dual{T}\|$. Assume $\xi_0\geq\xi_i$ for $i<k$.
Take $y\in E_{\xi_0+1}^*$. Then, since the $D_{\xi_i}^*$ are pairwise
disjoint and since $y\in E_{\xi_i+1}^*$, for every $i<k$, we have
$$|\dual{T}(\delta_y)(\bigcup_{i<k}D_{\xi_i}^*)|=
|\sum_{i<k}\dual{T}(\delta_y)(D_{\xi_i}^*)|>|\sum_{i<k}a_{\xi_i}|>\|\dual{T}\|.$$
This contradiction proves the claim.

Therefore, for every $A$-supported $f\in \ell_\infty$ and every $y\in E^*$ we have 
$T(f)(y)=\dual{T}(\delta_y)(f)=r_yf(\psi(y))$. In particular, 
$T(\chi_{A^*})(y)=r_y$, for every $y\in E^*$. This means that 
the function $y\mapsto r_y$ with domain $E^*$ is continuous.
Then, by \ref{constantfunction} it must be constant on some clopen $B^*\subseteq E^*$.
%  \textcolor{red}{[Here we use that $\psi$ is a  homeomorphism.]}
This means that for some $r\in\mathbb{R}$ we have
 $T_{B,A}(f)=r(f\circ\psi)| B$, for every $A$-supported $f\in \ell_\infty$.
By going to a subset of $A$ we may choose $A$ and $B$ so that $\psi^{-1}[A^*]=B^*$

\end{proof}

\begin{theorem}\label{irreducibledrewnowski}
 Let $T:C(\N^*)\rightarrow C(\N^*)$ be a bounded linear operator and let
$\tilde{A}\subseteq \N$ be infinite and $F\subseteq \N^*$ be closed. 
If $\psi:F\rightarrow\N^*$ is an irreducible continuous function, %such that $\tilde{A}^*\subseteq \psi[F]$,
then there exist $r\in\mathbb{R}$ and an infinite $A\subseteq\tilde{A}$  such that 
$T(f^*)|\psi^{-1}[A^*]=r(f^*\circ \psi)| \psi^{-1}[A^*]$,
for every $A$-supported $f\in \ell_\infty$.
\end{theorem}

\begin{proof} The proof is similar to that of \ref{quasi-opendrewnowski}, so
we will skip identical parts. The main difference is that nonempty $G_\delta$'s of
$F$ do not need to have nonempty interior. However the irreducibility of the map
onto $\N^*$ gives through Lemma \ref{denseirreducible} that appropriate $G_\delta$'s
have nonempty interior.
Fix  $\tilde{A}$, $F$ and $\psi$ as above.

 \textsc{Claim:} There exists an infinite $A\subseteq \tilde{A}$  such that for every 
$y\in \psi^{-1}[A^*]$ there exists $r_y\in\mathbb{R}$ satisfying 
$$\dual{T}(\delta_y)| A^*=r_y\delta_{\psi(y)}.$$

Suppose this does not hold. We will construct recursively sequences $(A_{\xi})_{\xi<\omega_1}$ and
$(D_{\xi})_{\xi<\omega_1}$ of infinite subsets of $\N$, and a sequence 
$(a_{\xi})_{\xi<\omega_1}$ of nonzero reals such that 
\begin{enumerate}
 \item $A_{\eta}\subseteq_* A_{\xi}\subseteq_*\tilde{A}$ and $D_{\xi}\subseteq_* A_{\xi}\setminus A_{\xi+1}$,
for every $\xi<\eta<\omega_1$;
\item either $\dual{T}(\delta_y)(D_{\xi}^*)>a_{\xi}>0$ for all $y\in \psi^{-1}[A_{\xi+1}^*]$, or
$\dual{T}(\delta_y)(D_{\xi}^*)<a_{\xi}<0$ for all $y\in \psi^{-1}[A_{\xi+1}^*]$.
\end{enumerate}

Let $A_0=\tilde{A}$ . Let $\eta<\omega_1$ and 
suppose we have constructed $A_{\xi}$, $D_{\xi}$
and $a_{\xi}$
satisfying (1)--(2) for every $\xi<\eta$. 
If $\eta$ is a limit ordinal, take an infinite $A_\eta$ such that 
$A_\eta\subseteq_*  A_\xi$ for every $\xi<\eta$. 
 Now we may suppose
we have $A_{\xi}$  for every $\xi\leq\eta$, and $D_{\xi}$ and $a_{\xi}$
for every $\xi<\eta$.

By our assumption, there exist $y\in \psi^{-1}[A_{\eta}^*]$ and 
$X\subseteq A_{\eta}^*\setminus\{\psi(y)\}$ such that $\dual{T}(\delta_y)(X)\neq 0$. 
By the regularity of $\dual{T}(\delta_y)$, there exists
an infinite $D_{\eta}\subseteq_* A_{\eta}$ such that $\psi(y)\notin D_{\eta}^*$
and $\dual{T}(\delta_y)(D_{\eta}^*)\neq 0$. Let $a_{\eta}$ be such that
either $0<a_{\eta}<\dual{T}(\delta_y)(D_{\eta}^*)$ or $0>a_{\eta}>\dual{T}(\delta_y)(D_{\eta}^*)$.

By the weak$^*$ continuity of $\dual{T}$, %Since $\dual{T}(\delta_y)(D_{\eta}^*)=T(\chi_{D_{\eta}^*})(y)$ and $T(\chi_{D_{\eta}^*})$
%is a continuous function, 
there exists $V$ a clopen neighbourhood of $y$ in $F$
such that either $\dual{T}(\delta_z)(D_{\eta}^*)>a_{\eta}$ for all $z\in V$, or
$\dual{T}(\delta_z)(D_{\eta}^*)<a_{\eta}$ for all $z\in V$. $V$ may be assumed to be included
in $\psi^{-1}[A_\eta^*\setminus D_\eta^*]$ as $y\in V\cap \psi^{-1}[A_\eta^*\setminus D_\eta^*]$. Using the irreducibility
of $\psi$ and Lemma \ref{denseirreducible} there is an infinite $A_{\eta+1}\subseteq \N$
such that $\psi^{-1}[A_{\eta+1}^*]\subseteq V\subseteq \psi^{-1}[A_\eta^*\setminus D_\eta^*]$. In particular $A_{\eta+1}\subseteq_* A_\eta\setminus D_\eta$ (note that $y$ may not belong to $\psi^{-1}[A_{\eta+1}^*]$).
This ends the construction. We finish the proof of the claim as in Theorem \ref{quasi-opendrewnowski}.

Therefore, for every $A$-supported $f\in \ell_\infty$ and every $y\in \psi^{-1}[A^*]$ we have 
$T(f^*)(y)=\dual{T}(\delta_y)(f)=r_yf(\psi(y))$. In particular, 
$T(\chi_{A^*})(y)=r_y$, for every $y\in \psi^{-1}[A^*]$. This means that 
the function $y\mapsto r_y$ with domain $\psi^{-1}[A^*]$ is continuous.
Then, by \ref{constantirreducible} it must be constant on some clopen set of $F$ of the
 the form $\psi^{-1}[B^*]$ for an infinite $B\subseteq A$.
%  \textcolor{red}{[Here we use that $\psi$ is a  homeomorphism.]}
This means that for some $r\in\mathbb{R}$ we have
 $T(f)=r(f\circ\psi)| \psi^{-1}[B]$, for every $B$-supported $f\in \ell_\infty$.

\end{proof}

\subsection{Left-local canonization of fountainless operators}

\begin{lemma} \label{quasi-quasi-open}
Suppose that $B\subseteq \N$ is infinite, $T: C(\N^*)\rightarrow C(\N^*)$
is fountainless and everywhere present. Then,
$$F=\bigcup\{\varphi^T(y): y\in B^*\}$$
has nonempty interior.
\end{lemma}

\begin{proof}
Suppose $F$ is nowhere dense. Take $B_1\subseteq_* B$ and $C_n$, $D_n\subseteq\N$
as in \ref{norm_stability}.
In view of applying \ref{locallydet+somewherelocallynull->null}, we will find a dense 
family $\mathcal A$ such that $T(f^*)| B_1^*=0$ whenever the support of 
$f$ is included in an element of $\mathcal A$. This would contradict the fact that $T$
is everywhere present.

Fix a nonempty clopen $U\subseteq \N^*$ disjoint from $F$. Notice that for every 
$y\in B_1^*$ the measure $\dual{T}(\delta_y)$ has no atoms in $U$.
Therefore, by the regularity of $\dual{T}(\delta_y)$ and the compactness of $U$
we may find for each $n\in\N$ an open covering $(U_i(y,n))_{i<j(y,n)}$
of $U$ by clopen sets such that $|\dual{T}(\delta_y)|(U_i(y,n))<1/2(n+1)$ holds
for all $i<j(y,n)$. We may further assume that 
either $U_i(y,n)\subseteq C_n^*$  or $U_i(y,n)\subseteq D_n^*$, for each 
$n\in\N$ and each $i<j(y,n)$.

Now  by the weak$^*$ continuity of $\dual{T}$, we may choose for each $n\in\N$ an
open neighbourhood $V_n(y)$ of $y$ such that for all $z\in V_n(y)$ we have
\begin{equation}\label{equation_small_measure'}
|\dual{T}(\delta_z)(U_i(y,n))|<1/2(n +1)
\end{equation}
for all $i<j(y,n)$.

We have thus constructed for each $n\in\N$ an open covering $\{V_n(y):y\in B_1^*\}$ of $B_1^*$.
By the compactness of $B_1^*$, for each $n \in\N$ take $y_0(n), ..., y_{m(n)-1}(n)\in B_1^*$
such that 
$$B_1^*\subseteq \bigcup_{l< m(n)}V_n(y_l).$$

Now consider the family $\mathcal A_U$ of those sets $E\subseteq \N$
such that given $n\in \N$,  for each $l< m(n)$ there is
$i<j(y_l,n)$ such that $E^*$ is included in $U_i(y_l,n)$.
We claim that if $E\in\mathcal A_U$, then for every $E^*$- supported 
$f^*\in C(\N^*)$ we have that $T(f^*)| B_1^*=0$.

Fix $E\in\mathcal A_U$ and $y\in B_1^*$. We show that for every $n\in\N$
we have $|\dual{T}(\delta_y)|(E^*)<1/(n+1)$. Let 
$\dual{T}(\delta_y)=\mu^+-\mu^-$ be a Jordan decomposition of the measure.
By \ref{norm_stability}  we have that $\mu^-(C_n^*)<1/4(n+1)$ and $\mu^+(D^*_n)<1/4(n+1)$.
 By construction there exists $l<m(n)$ such that $y\in V_n(y_l)$, and by the definition of $\mathcal A_U$,
there exists $i<j(y_l,n)$ such that $E^*\subseteq U_i(y_l,n)$. 
We may assume without loss of generality that 
$U_i(y_l,n)\subseteq C_n^*$. From this and from (\ref{equation_small_measure'}) above we obtain:
$$\begin{array}{rcl}
  |\dual{T}(\delta_y)|(E^*)&\leq& |\dual{T}(\delta_y)|(U_i(y_l,n))\\
&=&\dual{T}(\delta_y)(U_i(y_l,n))+2\mu^-(U_i(y_l,n))\\
&\leq&|\dual{T}(\delta_y)(U_i(y_l,n))|+2\mu^-(C^*_n)\\
&<&1/2(n+1)+2/4(n+1)\\
% &=&1/(n+1) 
 \end{array}$$

Therefore, $|\dual{T}(\delta_y)|(E^*)=0$ for every $y\in B_1^*$. 
So if the support of $f^*\in C(\N^*)$ is included in $E^*$,
we have that $0=\int f^*\,\mathrm{d} \dual{T}(\delta_y)=\dual{T}(\delta_y)(f^*)=T(f^*)(y)$, 
for every $y\in B_1^*$, as claimed.

Now, notice that it is enough to show that $\mathcal A_U$ is dense under $U$,
as in that case $\mathcal A=\bigcup\{\mathcal A_U:U\subseteq\N^*\setminus F,\, U \text{ clopen}\}$
is the dense family we are after.% (this is because the set $\bigcup\{U\subseteq\N^*:U\text{ is clopen, }U\cap F=\emptyset\}$ is dense).

To see that $\mathcal A_U$ is dense under $U$, fix an infinite $E_{0,0}\subseteq \N$ 
such that $E_{0,0}^*\subseteq U$. We define by induction a $\subseteq_*$- decreasing sequence
of infinite sets $\{E_{n,l}:n\in\N,\,l<m(n)\}$ ordered lexicographically.
 Suppose $(n',l')$ is successor of $(n,l)$. Since 
$(U_i(y_l,n))_{i<j(y_l,n)}$ is an open covering
of $U$, we may choose $i<j(y_l,n)$ such that $E_{n,l}^*\cap U_i(y_l,n)\neq\emptyset$.
Take $\emptyset\neq E_{n',l'}^*\subseteq E_{n,l}^*\cap U_i(y_l,n)$.
Finally, if we take an infinite $E$ such that $E\subseteq_* E_{n,l}$ for every 
$n\in\N$, $l<m(n)$, then it is clear
that $E\subseteq_* E_{0,0}$ and $E\in\mathcal A_U$.

\end{proof}

\begin{lemma}\label{atombound} 
Let $B\subseteq\N$ be infinite. Suppose $T:C(\N^*)\rightarrow C(\N^*)$ is such that 
$\varphi^T(y)\not=\emptyset$ for each $y\in B^*$
and $\bigcup\{\varphi^T(y): y\in V\}$ has nonempty interior for every open $V\subseteq \N^*$.
Then, there is an infinite $B_1\subseteq_* B$  and $\varepsilon>0$
 such that 
\begin{enumerate}
 \item $\varphi^T_\varepsilon(y)\not=\emptyset$ for each $y\in B_1^*$, and 
\item $\bigcup\{\varphi^T_\varepsilon(y): y\in D^*\}$ has nonempty interior for
 every infinite $D\subseteq_* B_1$.
\end{enumerate}
\end{lemma}

\begin{proof}
Suppose that (1) fails for all infinite $B'\subseteq_* B$  and all $\varepsilon>0$.
Let $B_0\subseteq_* B$, $C_n$, $D_n\subseteq\N$ be given by Lemma \ref{norm_stability}.
We will construct by induction a $\subseteq_*$- decreasing 
sequence $(B_n)_{n\in\N}$ of infinite subsets of $\N$ such that
$\varphi^T_{1/(n+1)}(y)=\emptyset$ for every $y\in B_{n+1}^*$. If we then take
any $y\in\bigcap_{n\in\N} B^*_n$, we will have that $\varphi^T(y)=\emptyset$, which
contradicts our hypothesis.

Assume we have already constructed $B_n$. Since we are assuming that (1)
fails, there exists $y\in B_n^*$ such that $\varphi^T_{1/2(n+1)}(y)=\emptyset$.
This means that $|\dual{T}(\delta_y)|(\{x\})=|\dual{T}(\delta_y)(\{x\})|<1/2(n+1)$, 
for every $x\in\N^*$. 

By the regularity of the measure $\dual{T}(\delta_y)$ and by the 
compactness of $\N^*$, we may cover $\N^*$ by finitely many clopen $(U_i)_{i<k}$
such that $|\dual{T}(\delta_y)|(U_i)<1/2(n+1)$, for each $i<k$. We may further assume
that each $U_i$ is included in either $C_n^*$ or $D_n^*$. Since $\dual{T}$ is weak$^*$
continuous, we may find an open neighbourhood of $y$, say $V$, such that 
for every $z\in V$ we have $|\dual{T}(\delta_z)(U_i)|<1/2(n+1)$, for each $i<k$.
Take $B_{n+1}$ such that $y\in B_{n+1}^*\subseteq B_n^*\cap V$.
We claim that $\varphi^T_{1/(n+1)}(z)=\emptyset$, for every $z\in B_{n+1}^*$.

Fix any $z\in B_{n+1}^*$ and take $\dual{T}(\delta_z)=\mu^+-\mu^-$ a Jordan decomposition
of the measure. Recall that by Lemma \ref{norm_stability} we have that 
$\mu^-(C_n^*)<1/4(n+1)$ and $\mu^+(D_n^*)<1/4(n+1)$.
 Now take any $x\in\N^*$. 
Let $i<k$ be such that $x\in U_i$ and assume without loss of 
generality that $U_i\subseteq C_n^*$. Then,
$$ |\dual{T}(\delta_z)(\{x\})|\leq|\dual{T}(\delta_z)|(U_i)\leq |\dual{T}(\delta_z)(U_i)|+2\mu^-(C_n^*)
<1/(n+1),$$
and the claim is proved. This finishes the proof of the first part, so let us assume 
that $\varepsilon_0>0$ and $B_0\subseteq_* B$ are such that (1) holds.

To prove the second part, let us assume that for every
$B'\subseteq_* B_0$ and every  $\varepsilon>0$ there exists an infinite $D\subseteq_* B_0$ 
with $\bigcup\{\varphi^T_\varepsilon(y): y\in D^*\}$ nowhere dense.
We may then find a $\subseteq_*$- descending sequence of infinite sets 
$(D_n)_{n\in\N}$ such that $D_n\subseteq_* B$ and $\bigcup\{\varphi^T_{1/n}(y): y\in D_n^*\}$
is nowhere dense, for every $n\in\N$.
Let $V\subseteq\bigcap_{n\in\N}D_n^*$ be a nonempty open.
Then, since $\varphi^T(y)=\bigcup_{n\in\N}\varphi^T_{1/n}(y)$, we have
$$\bigcup\{\varphi^T(y): y\in V\}= \bigcup_{n\in \N}\bigcup\{\varphi^T_{1/n}(y): y\in V\}
\subseteq \bigcup_{n\in \N}\bigcup\{\varphi^T_{1/n}(y): y\in D^*_n\},$$
which is nowhere dense by \ref{countablenwd}.
This contradicts the hypothesis of the lemma.
Therefore, there exist an infinite $B_1\subseteq_* B_0$ and $\varepsilon_1>0$ which satisfy (2).

The lemma holds for $B_1$ and $\varepsilon=\min\{\varepsilon_0,\varepsilon_1\}$.
\end{proof}

\begin{lemma} \label{continuous_quasi-open}
Suppose that  $T: C(\N^*)\rightarrow C(\N^*)$
is fountainless and everywhere present. Then, for every infinite $B\subseteq\N$
there is an infinite  $B_1\subseteq_* B$ and a continuous quasi-open $\psi: B_1^*\rightarrow \N^*$
such that 
$$\dual{T}(\delta_y)(\{\psi(y)\})\not=0$$
for all $y\in B_1^*$.
\end{lemma}

\begin{proof} 
Since $T$ is fountainless and everywhere present, by  \ref{nonempty} and \ref{quasi-quasi-open}
we know that the hypothesis of \ref{atombound} are satisfied. So
find $\varepsilon>0$ and an infinite
$B_0\subseteq_* B$ such that
$\varphi^T_\varepsilon(y)\not=\emptyset$ for every $y\in B_0$,
and $\bigcup\{\varphi^T_\varepsilon(y): y\in D^*\}$ has nonempty interior for
 every infinite $D\subseteq_* B_0$. We may also assume that there exist 
$C_n$, $D_n\subseteq\N$ for every $n\in\N$ such that the statement in \ref{norm_stability}
holds for $B_0$ and $C_n$, $D_n$.

\textsc{Claim 1:} There exists an infinite $B'_0\subseteq_* B_0$ and a finite collection
 $(V_i)_{i<k}$ of almost disjoint infinite subsets of $\N$ such that for every $z\in (B'_0)^*$
we have that $\varphi^T_{\varepsilon}(z)\subseteq \bigcup_{i<k} V_i^*$ and 
$|\varphi^T_{\varepsilon}(z)\cap V_i^*|=1$, for each $i<k$.

We will construct recursively a $\subseteq_*$-descending sequence $(A_n)$ of subsets of $\N$, 
$y_n\in A_n^*$, finite collections $(V_{n,i})_{i<k_n}$ of almost disjoint infinite subsets of $\N$
and open intervals $I_i^n\subseteq\mathbb{R}$, $i<k_n$, such that for every $n$ we have
\begin{enumerate}
\item $\varphi^T_{\varepsilon}(z)\subseteq\bigcup_{i<k_n}V^*_{n,i}$, for all $z\in A_{n+1}^*$
 \item For every $i<k_{n+1}$ there exists $j<k_n$ such that $V_{n+1,i}\subseteq_*V_{n,j}$
\item $|\dual{T}(\delta_z)(V_{n,i})|\in I_i^n$, for all $z\in A_{n+1}^*$ and all $i<k_n$
\item \emph{length}$(I_i^n)=\varepsilon(2^{n+1}k_n)^{-1}$, for all $i<k_n$ 
\end{enumerate}

Begin by noticing that for every $y\in\N^*$ the number of elements of 
$\varphi^T_\varepsilon(y)$ is finite, as it must be bounded by $\|T\|/\varepsilon$.
% Otherwise, 
% $\|T\|\geq \sum_{x\in\varphi^T_{\varepsilon}(y)}|\dual{T}(\delta_y)(\{x\})|
% \geq\varepsilon|\varphi^T_{\varepsilon}(y)|>\|T\|$.
Let $A_0=B_0$ and fix any $y_0\in A_0^*$ and
let $\{x^0_i:i<k_0\}$ be an enumeration of $\varphi^T_{\varepsilon}(y_0)$ 
(note that $k_0\geq 1$).
Let $N_0\in\N$ be such that $1/N_0<\varepsilon/8k_0$.
By the regularity of the measure $\dual{T}(\delta_{y_0})$, we find for each $i<k_0$
a clopen neighbourhood $V^*_{0,i}$  of $x^0_i$ such that
$$%|\dual{T}(\delta_{y_0})(\{x^0_i\})|-\varepsilon/4(k_0+1)<
|\dual{T}(\delta_{y_0})|(V_{0,i}^*)<|\dual{T}(\delta_{y_0})(\{x^0_i\})|+\varepsilon/8k_0.$$
We may assume that the $V_{0,i}$'s are almost disjoint and that each of them is almost included in 
either $C_{N_0}$ or $D_{N_0}$. 

For each $i<k_0$ we define $I^0_i\subseteq \mathbb{R}$ to be the open
 interval with centre $|\dual{T}(\delta_{y_0})(\{x^0_i\})|$ and radius $\varepsilon/4k_0$.
By the above, $|\dual{T}(\delta_{y_0})|(V_{0,i}^*)$ lies in $I^0_i$,
and as we shall see,  $|\dual{T}(\delta_{y_0})(V_{0,i}^*)|$ does so as well. Indeed, take $i<k_0$ and
assume without loss of generality that $V^*_{0,i}\subseteq C_{N_0}^*$.
If $\dual{T}(\delta_{y_0})=\mu^*-\mu^-$ is a Jordan decomposition of the 
measure, then
$$|\dual{T}(\delta_{y_0})|(V_{0,i}^*)\leq \dual{T}(\delta_{y_0})(V_{0,i}^*)+2\mu^-(V_{0,i}^*)
\leq |\dual{T}(\delta_{y_0})(V_{0,i}^*)|+2\mu^-(C_{N_0}^*).$$
From this it follows that 
$|\dual{T}(\delta_{y_0})|(V_{0,i}^*)-|\dual{T}(\delta_{y_0})(V_{0,i}^*)|<1/2(N_0+1)<\varepsilon/8k_0$,
and so $|\dual{T}(\delta_{y_0})(V_{0,i}^*)|\in I^0_i$.

By the upper semicontinuity of $\varphi^T_{\varepsilon}$ (Lemma \ref{upper_semicontinuity})
and  the weak$^*$-continuity of $\dual{T}$, we now find a clopen neighbourhood of 
$y_0$, say $A_1^*$, which we may assume to be included in $A_0^*$, such that
for every $z\in A_1^*$ we have 
$$\varphi^T_{\varepsilon}(z)\subseteq\bigcup_{i<k_0}V^*_{0,i} \quad\text{ and }$$%\quad
$$|\dual{T}(\delta_z)(V_{0,i}^*)|\in I^0_i,\text{ for each }i<k_0.$$

If we have that $|\varphi^T_{\varepsilon}(z)\cap V_{0,i}^*|= 1$,
 for every $z\in A_1^*$ and for each $i<k_0$, then the recursion stops and 
the claim is proved. Otherwise, choose $y_1\in A_1^*$ a witness to this fact,
and repeat the procedure to obtain open intervals $I^1_i$ with centre 
$|\dual{T}(\delta_{y_1})(\{x^1_i\})|$ and radius $\varepsilon/2^3k_1$, 
and clopen sets $V_{1,i}^*$ such that both $|\dual{T}(\delta_{y_1})|(V_{1,i}^*)$ 
and $|\dual{T}(\delta_{y_1})(V_{1,i}^*)|$ lie inside $ I^1_i$, for each $i<k_1$. 
Notice that we may take each set $V_{1,i}$ as a subset of one of the $V_{0,j}$.
Then, by the same argument using the upper semicontinuity of $\varphi^T_\varepsilon$
 and the  weak$^*$-continuity of $\dual{T}$
we obtain an infinite $A_2\subseteq_* A_1$ such that for every $z\in A_2^*$ we have
$\varphi^T_{\varepsilon}(z)\subseteq \bigcup_{i<k_1} V_{1,i}^*$ and
$|\dual{T}(\delta_z)(V_{1,i}^*)|\in I^1_i$, for each $i<k_1$.

We claim that this process stops after
finitely many steps.
First notice that the failure to stop at step $n$ is due to one of two reasons:
\begin{itemize}
 \item [(a)] there exists $y_{n+1}\in A_{n+1}^*$ such that 
$|\varphi^T_{\varepsilon}(y_{n+1})\cap V_{n,i}^*|\geq 2$ for some $i<k_n$
\item [(b)] there exists $y_{n+1}\in A_{n+1}^*$ such that 
$\varphi^T_{\varepsilon}(y_{n+1})\cap V_{n,i}^*=\emptyset$ for some $i<k_n$
\end{itemize}
Notice also that once condition (a) fails, it continues to fail in subsequent steps.
So we may assume that we first only check for condition (a), and only after
it does not occur do we check for condition (b). 

By condition (2) in the construction, we have that every time (a) occurs there exists $i<k_0$ such that
$|\varphi^T_{\varepsilon}(y_{n+1})\cap V_{0,i}^*|\geq 2$.
So for each $i<k_0$ consider $m_i\in\N$  such that 
$m_i\cdot\varepsilon\geq \dual{T}(\delta_{y_0})(\{x_i^0\})$ and suppose (a) has 
occurred at $n=\sum_{i<k_0}m_i$ many steps. Suppose that still (a) happens once
more. Then, there exists $i_0<k_0$  and 
$n_0<\dots<n_{m_{i_0}}=n$ such that $|\varphi^T_{\varepsilon}(y_{n_j+1})\cap V_{0,i_0}^*|\geq 2$,
for every $j\leq m_{i_0}$. Hence, if
$x^{n+1}_0,\ x^{n+1}_1\in \varphi^T_{\varepsilon}(y_{n+1})\cap V_{n,i_n}^*$,
for certain $i_n<k_n$, then
$|\dual{T}(\delta_{y_{n+1}})|(\{x^{n+1}_0\})+\varepsilon\leq
|\dual{T}(\delta_{y_{n+1}})|(\{x^{n+1}_0,x^{n+1}_1\})
\leq|\dual{T}(\delta_{y_{n+1}})|(V_{n,i_n}^*)$
and we obtain
$$\begin{array}{rcl}
|\dual{T}(\delta_{y_{n+1}})|(\{x^{n+1}_0\})&<&\sup I^n_{i_n}-\varepsilon\\
&=&|\dual{T}(\delta_{y_n})|(\{x^n_{i_n}\})+\varepsilon(2^{n+2}k_n)^{-1}-\varepsilon\\
&<&\dots\\
&<&|\dual{T}(\delta_{y_0})|(\{x_{i_0}^0\})+\sum_{j<n+1}\varepsilon(2^{j+2}k_j)^{-1}-(m_{i_0}+1)\cdot\varepsilon\\
&\leq&(|\dual{T}(\delta_{y_0})|(\{x_{i_0}^0\})-m_{i_0}\cdot\varepsilon)+(\varepsilon\sum_{j<n+1} 1/2^{j+2}-\varepsilon)\\
&<&0-\varepsilon/2.
\end{array}$$
From this contradiction we conclude that (a) can occur at most at $(\sum_{i<k_0}m_i)$
many steps.

Now assume that $n_0$ is such that (a) does not hold at 
step $n$, for all $n\geq n_0$. 
Suppose the recursion does not stop at step $n\geq n_0$. 
 Assume without loss of generality that 
$\varphi^T_{\varepsilon}(y_{n+1})\cap V_{n,0}^*=\emptyset$.
Therefore $\varphi^T_{\varepsilon}(y_{n+1})\subseteq \bigcup_{0<i<k_n} V_{n,i}^*$.
Since we also have $|\dual{T}(\delta_{y_{n+1}})|(V_{n,i}^*)\in I^n_i$, for each $i<k_n$,
we obtain 
$$\begin{array}{rcl}
  |\dual{T}(\delta_{y_{n+1}})|(\{x^{n+1}_i:i<k_{n+1}\})
&\leq& \sum_{i<k_n}|\dual{T}(\delta_{y_{n+1}})|(V_{n,i}^*)-|\dual{T}(\delta_{y_{n+1}})|(V_{n,0}^*)\\
&<& \sum_{i<k_n}\sup I^n_i-\inf I^n_0\\
&\leq& \sum_{i<k_n}|\dual{T}(\delta_{y_n})(\{x^n_i\})|+k_n\varepsilon/(2^{n+2}k_n)\\
& &\qquad \qquad -(\varepsilon-\varepsilon/2k_n)\\
&\leq&|\dual{T}(\delta_{y_n})|(\{x^n_i:i<k_n\})-\varepsilon/2.  \end{array}$$
The last inequality holds because $k_n\geq 1$.
Since %$\varphi^T_{\varepsilon}(y)\neq\emptyset$, for every $y\in B_0^*$,
$|\dual{T}(\delta_{y_n})|(\{x^n_i:i<k_n\})\geq \varepsilon$, for every $n$,
we conclude that (b) cannot occur indefinitely. Hence the recursion
must stop after finitely many steps.

\textsc{Claim 2:} There exists $i_0<k$ and an infinite $B_1\subseteq_* B'_0$ such 
that $V_{i_0}^*\cap\bigcup\{\varphi^T_\varepsilon(y): y\in D^*\}$ has nonempty interior for
 every infinite $D\subseteq_* B_1$.

Suppose this is not the case. Then, we may find a sequence on infinite sets 
$B_0'=A_0\supseteq_*A_1\supseteq_*\dots\supseteq A_k$ such that 
$V_i^*\cap\bigcup\{\varphi^T_\varepsilon(y): y\in A_{i+1}^*\}$ is nowhere dense, for $i<k$.
By Claim 1, we know that
$\bigcup\{\varphi^T_\varepsilon(y): y\in A_k^*\}\subseteq\bigcup_{i<k}V_i^*$,
and so $\bigcup\{\varphi^T_\varepsilon(y): y\in A_k^*\}=
\bigcup_{i<k}V_i^*\cap\bigcup\{\varphi^T_\varepsilon(y): y\in A_k^*\}$ is 
also nowhere dense. But this contradicts the choice of $B_0$.

Now we may define $\psi:B_1^*\rightarrow \N^*$  by 
$\{\psi(y)\}=V_{i_0}^*\cap\varphi^T_{\varepsilon}(y)$. It is clear that $\psi$ is 
quasi-open by Claim 2, so we conclude the proof by 
showing that $\psi$ is  continuous. Let $U\subseteq V_{i_0}^*$ be any open set. Since $V_{i_0}^*$ is clopen
and $\varphi^T_{\varepsilon}$ is upper semicontinuous, % (by Lemma \ref{upper_semicontinuity}),
 we have that 
$\psi^{-1}[U]=B_1^*\cap\{y\in \N^*:\varphi^T_{\varepsilon}(y)\subseteq U\cup \N^*\setminus V_{i_0}^*\}$
 is open.

\end{proof}

\begin{theorem}\label{surcanonization}
Suppose that $T:C(\N^*)\rightarrow C(\N^*)$ is a fountainless, everywhere present operator.
Then, $T$ is left-locally canonizable along a quasi-open map.
\end{theorem}
\begin{proof}
Fix an infinite $B\subseteq \N$ and  find using \ref{continuous_quasi-open}
 an infinite $B_1\subseteq B$ and a quasi-open
$\psi: B_1^*\rightarrow \N^*$ such that
$$\dual{T}(\delta_y)(\{\psi(y)\})\not=0$$
for all $y\in B_1^*$. Now use \ref{quasi-opendrewnowski} to find clopen sets 
$A^*\subseteq \psi[B_1^*]$ and $B_2^*\subseteq \psi^{-1}[A^*]$
and a real $r\in \R$ such that 
$T_{B_2,A}(f^*)=r(f^*\circ\psi)|B_2$ for every $f\in\ell_\infty(A)$. It follows that
$\dual{T}(\delta_y)|A^*=r\delta_{\psi(y)}$ for each $y\in B_2^*$,
 and so $0\not=\dual{T}(\delta_y)(\{\psi(y)\})=r$.
Hence $T_{B_2,A}$ is canonizable along $\psi$.
% hence $r\not=0$ and $P_{B_2}\circ T\circ I_A$ satisfies Definition \ref{trivializations} (4) of canonizable operator
% along $\psi$.
\end{proof}

\subsection{Right-local canonization of funnelless automorphisms}

The main result of this section is a consequence of  our generalization 
\ref{quasi-opendrewnowski} of the Drewnowski-Roberts canonization lemma and the
following result of Plebanek which is implicitly proved in Theorem 6.1 of \cite{plebanekisrael}.

\begin{theorem}\label{plebanekiso} Suppose that $T:C(K)\rightarrow C(K)$ is an automorphism.
Then, there is a $\pi$-base $\mathcal U$  of $\N^*$ such that for every $U\in \mathcal U$
there is  a closed $F\subseteq \N^*$ and a continuous surjection $\psi: F\rightarrow {\overline{U}}$
such that
$$|\dual{T}(\delta_y)|(\{\psi(y)\})\not=0$$
for all $y\in F$.
\end{theorem}

\begin{theorem}\label{injcanonization}
Suppose that $T:C(\N^*)\rightarrow C(\N^*)$ is a funnelless automorphism.
Then, $T$ is right-locally canonizable along a quasi-open map.
\end{theorem}
\begin{proof} 
Fix an infinite $A\subseteq \N$. Let $\mathcal U$ be as in \ref{plebanekiso} for $T$
and find $U\in \mathcal U$ such that $U\subseteq A^*$. Let 
$F$ and $\psi$ be as in \ref{plebanekiso} for $U$.
Let $A_1\subseteq\N$ be infinite
such that $A_1^*\subseteq U$
and put $F_1=\psi^{-1}[A^*_1]\subseteq F$.
Let $F_2\subseteq F_1$ be closed such that
$\psi|F_2$ is irreducible and  onto $A_1^*$
 and hence quasi-open (relative to the subspace topology on $F_2$) by lemma \ref{irreduciblequasi-open}.  

As $T$ is funnelless (Definition \ref{funnelless}) and $A_1^*$ is open, $F_2$ cannot 
be nowhere dense, so let $B^*$ be a nonempty clopen
set included in $F_2$. Now $\psi: B^*\rightarrow A_1^*$ is a continuous, quasi-open map
(as a restriction of a quasi-open map to a clopen subset of $F_2$)
satisfying $\dual{T}(\delta_y)(\{\psi(y)\})\not=0$ for each $y$ in $B^*$.
Therefore, we can apply 
\ref{quasi-opendrewnowski} to obtain an infinite $A_2\subseteq A_1$, a clopen 
$B_1^*\subseteq \psi^{-1}[A_2^*]$ and a real $r\in\R$ such that
$T_{B_1,A_2}(f^*)=r(f^*\circ \psi)|B_1^*$,
for every  $f\in\ell_\infty(A_2)$. 
In particular we have that  
$$\dual{T}(\delta_y)(E^*)=T_{B_1,A_2}(\chi_{E^*})(y)=r(\chi_{E^*}\circ \psi)(y)=r\delta_{\psi(y)}(E^*)$$
for every infinite $E\subseteq A_2$ and every $y\in B_1^*$.
It follows that for each $y\in B_1^*$ we have
$\dual{T}(\delta_y)|A_2^*=r\delta_{\psi(y)}$, and so
 $0\not=\dual{T}(\delta_y)(\{\psi(y)\})=r$. Having
 $r\not=0$, we conclude that $T_{B_1,A_2}$ is canonizable along $\psi$.
% as required by Definition \ref{trivializations} (4).

\end{proof}

\section{The impact of combinatorics on the 
canonization and trivialization of operators on $\ell_\infty/c_0$}

As expected based on the study of $\N^*$ (e.g., \cite{shelahproper}, \cite{stevoquotients}, \cite{boban}, \cite{dow2to1}, \cite{dowhart}, \cite{ilijasauto}), the impact of additional set-theoretic
assumptions on the structure of operators on $\ell_\infty/c_0$ is also very dramatic.

\subsection{Canonization and trivialization of operators on $\ell_\infty/c_0$ under OCA+MA}

Recall from \cite{ilijasauto} that an ideal $\mathcal I$ of subsets of $\N$ 
is called c.c.c. over Fin if, and only if, there are no uncountable
 almost disjoint families of $\mathcal I$-positive sets. Dually
a closed subset $F\subseteq \N^*$ is called c.c.c. over Fin if
$A_\xi^*\cap F=\emptyset$ for some $\xi<\omega_1$ whenever $\{A_\xi: \xi<\omega_1\}$
is an almost disjoint family of infinite subsets of $\N$.
The following theorem by I. Farah (3.3.3. and  3.8.1. from \cite{ilijasauto}) 
will be crucial in this subsection:

\begin{theorem}[OCA+MA (\cite{ilijasauto})]\label{ocacontinuous}
Let $h:\wp(\N)/\text{\emph{Fin}}\longrightarrow \wp(\N)/\text{\emph{Fin}}$ be a homomorphism. Then,
 there is an infinite $B\subseteq \N$,
 a function $\sigma:B\rightarrow \N$ and a homomorphism
 $h_2:\wp(\N)/\text{\emph{Fin}}\rightarrow \wp(\N\setminus B)/\text{\emph{Fin}}$
  such that 
 $h(A)=[\sigma^{-1}[A]]\cup h_2([A])$ for every $A\subseteq \N$,
  and $Ker(h_2)$ 
 is c.c.c. over Fin.
\end{theorem}

 The following is a topological reformulation of the above theorem:
 
 \begin{theorem}[Proposition 7, \cite{dow2to1}(OCA+MA)]\label{ocacontinuoustopo}
 Suppose $\psi: \N^*\rightarrow \N^*$ is a continuous  mapping. 
Then, there exist an infinite $B\subseteq \N$ and a function $\sigma: B\rightarrow \N$ such that
 $$\psi(x)=\sigma^*(x)$$
 for all $x\in B^*$ and  $F=\psi[(\N\setminus B)^*]$ is a nowhere dense closed 
  c.c.c. over Fin set.
 \end{theorem}
 \begin{proof} By the Stone duality every continuous $\psi: \N^*\rightarrow \N^*$
corresponds to a homomorphism $h:\wp(\N)/\text{\emph{Fin}}\rightarrow \wp(\N)/\text{\emph{Fin}}$
given by $h([A])=[D]$, where $D^*=\psi^{-1}[A^*]$, for every $A\subseteq \N$.
Let $B\subseteq \N$, $\sigma: B\rightarrow \N$ and $h_2$
be as in  \ref{ocacontinuous} for the homomorphism $h$.
For every infinite $A\subseteq \N$ we  have $h([A])\cap [B]=[\sigma^{-1}[A]]$ by \ref{ocacontinuous}
and so $\psi^{-1}[A^*]\cap B^*=(\sigma^{-1}[A])^*$. Therefore, $\sigma^*=\psi|B^*$.
For every infinite $A\subseteq \N$ we  have $h([A])\cap[\N\setminus B]
=h_2([A])$ by \ref{ocacontinuous}, so for every $x\in(\N\setminus B)^*$ we have $\psi(x)=h^{-1}_2[\{[A]: A\in x\}]$
by the Stone duality. The set $F=\psi[(\N\setminus B)^*]$ is closed and 
for every $x\in(\N\setminus B)^*$, the set
 $h^{-1}_2[\{[A]: A\in x\}]$ is disjoint from $Ker(h_2)$, which is c.c.c. over Fin.
Therefore, $F$ is c.c.c. over Fin and c.c.c. over Fin sets are nowhere dense.
 \end{proof}

It turns out that quasi-open maps $\psi: B^*\rightarrow \N^*$ can be reduced to
bijections between subsets of $\N$ assuming OCA+MA.

\begin{lemma}[OCA+MA]\label{ocacontinuoussemi} 
Let $B\subseteq \N$ be infinite.
Suppose that $\psi:B^*\rightarrow \N^*$ is a continuous quasi-open mapping.
Then, there are infinite $B_1\subseteq B$, $A\subseteq \N$ and a bijection $\sigma:B_1\rightarrow A$
 such that $\psi| B_1^*=\sigma^*$.  In particular,  $\psi|B^*_1$
is a homeomorphism.
\end{lemma}
\begin{proof} 
By \ref{ocacontinuoustopo} there exist
 $B_0\subseteq B$ and a function $\sigma:B_0\rightarrow \N$ such that
 $\psi(x)=\sigma^*(x)$  for all $x\in B_0^*$.

Now, since $\psi$ is quasi-open, there exists an infinite $E\subseteq \N$ such that 
$E^*\subseteq \psi[B_0^*]=\sigma^*[B_0^*]$. Therefore,
$E\subseteq ^* \sigma[B_0]$ and in particular the image of $\sigma$ is infinite and so
there is an infinite $B_1\subseteq B_0$
such that $\sigma|B_1$ is a bijection onto its image $A$.
\end{proof}

\begin{theorem}[OCA+MA]\label{ocatrivialization} If $T:\ell_\infty/c_0\rightarrow\ell_\infty/c_0$
is a fountainless everywhere present operator, then it is left-locally canonizable and so left-locally trivial.
If $T:\ell_\infty/c_0\rightarrow\ell_\infty/c_0$ is 
 a funnelless automorphism, then it is right-locally canonizable  and so right-locally trivial.
\end{theorem}
\begin{proof}
Apply \ref{surcanonization} and \ref{injcanonization}
to obtain left-local or right-local canonization along a quasi-open mapping, respectively. 
Now use \ref{ocacontinuoussemi} to conclude that this mapping is somewhere induced by a bijection.
\end{proof}

\subsection{Operators on $\ell_\infty/c_0$ under CH}

The continuum hypothesis is a strong tool allowing transfinite induction constructions in $\wp(\N)/\text{\emph{Fin}}$
which induce objects in $\ell_\infty/c_0$. Actually, a considerable part of this strength is 
included in a powerful consequence of Parovi\v cenko's theorem: 
if $X$ is zero-dimensional, locally compact, $\sigma$-compact, noncompact Hausdorff space of weight
at most continuum, then $X^*=\beta X\setminus X$ is
 homeomorphic to $\N^*$ (1.2.6 of \cite{vanmillhandbook}). 
In this section we will often be using this result combined with the universal
property of $\beta X$ for locally compact $X$, that every continuous function on $X$
into a compact space extends to  $\beta X$

\begin{theorem}[CH]\label{nowherecanonizable} 
There is an automorphism $T: \ell_\infty/c_0\rightarrow \ell_\infty/c_0$
which is nowhere canonizable along a quasi-open map on
an open set, in particular along a homeomorphism.
\end{theorem}
\begin{proof}
Let $K$ be an uncountable compact  zero-dimensional metrizable space with countably many
isolated points $\{x_m: m\in \N\}$ which form a  dense open subspace of $K$.
By the classical classification of separable spaces of continuous functions there is
 an isomorphism $S: C(2^\N)\rightarrow C(K)$.

Let $X=\N\times K$ and $Y=\N\times 2^\N$. Note that $X, Y$ satisfy the hypothesis
of the topological consequence of Parovi\v cenko's theorem (1.2.6 of \cite{vanmillhandbook})
mentioned above, hence there are homeomorphisms $\pi: \N^*\rightarrow X^*$
and $\rho: Y^*\rightarrow \N^*$.  Define $\tilde\tau: K\rightarrow \|S\| B_{\dual{C(2^\N)}}$ by
$\tilde\tau(x)=\dual{S}(\delta_x)$ for each $x\in K$, where  the dual ball is considered with the
weak$^*$ topology and identified with the Radon measures on $2^\N$.
Define $\tau: X\rightarrow \| S\| B_{\dual{C(\beta Y)}}$ by 
putting $\tau(n, x)$  to be the measure on $\beta Y$ which is zero on the 
complement of $\{n\}\times 2^\N$ and is equal to the measure $\tilde\tau(x)$ on 
$\{n\}\times 2^\N$. 
% here we identify the measure on $\{n\}\times 2^\N$ with the measure 
% on $\beta Y$ which is zero on the complement of $\{n\}\times 2^\N$. 
% 
By the universal property of $\beta X$ there is an extension $\beta\tau: \beta X
\rightarrow  \| S\| B_{\dual{C(\beta Y)}}$

\textsc{Claim:} For each $t\in X^*$ the measure $\beta\tau(t)$ is concentrated on $Y^*$.

Fix $t\in X^*$. Note that for every $n\in\N$
the set $\beta X\setminus \{k\in \N: k\leq n\}\times  K$ is a neighbourhood of $t$. 
Also,  for $k>n$ if $x\in \{k\}\times K$, then $\tau(x)(U)=0$
for every Borel subset $U$ of $\{n\}\times 2^\N$. This completes the proof of the claim
by the weak$^*$ continuity of $\beta\tau$.

Now we can define $T: C(Y^*)\rightarrow C(X^*)$ by $T(f)(t)=\int f \ud(\beta\tau(t))$
for every $t\in X^*$.
It is a well defined bounded linear operator by Theorem 1 in VI.7 of \cite{dunford}. 
We will show that  $T_\pi\circ T\circ T_\rho$ is an automorphism of $\N^*$
which is nowhere canonizable along  a quasi-open map.
For the former we need to prove that $T$ is an isomorphism and  for the latter we need to
prove that for every nonempty clopen sets $U\subseteq X^*$, $O\subseteq Y^*$
 there is no quasi-open $\phi: U\rightarrow O$ 
such that $(\beta\tau(t))|O=r\delta_{\phi(t)}$ for every $t$ in $U$ and some nonzero $r\in \R$. 

To prove that $T$ is an isomorphism, note that one can define  
$R: C(\beta Y)\rightarrow C(\beta X)$ by $R(f)(x)=\int f \ud(\beta\tau(x))$ for every $x\in X$,
and that $C(\beta Y)$ can be identified with the $\ell_\infty$-sum of $C(2^\N)$ while
 $C(\beta X)$ can be identified with the $\ell_\infty$-sum of $C(K)$. 
$R$ sends the subspace corresponding to the $c_0$-sum of $C(2^\N)$ into
the subspace corresponding to the $c_0$-sum of $C(K)$
since the original
$S$ is an isomorphism and $R$ is the $\ell_\infty$-sum of the operator $S$.
It follows that $T$ is induced by $R$ modulo the subspaces corresponding to the $c_0$-sums.
Moreover, one can note using the fact that $S$ is bounded below
that elements  outside  the subspace corresponding to the $c_0$-sum of $C(K)$ are
send by $R$ onto elements  outside 
the subspace corresponding to the $c_0$-sum of $C(K)$.
It follows that $T$ is nonzero on every nonzero element, i.e., is injective. 
The surjectivity of $T$ follows from the surjectivity of $R$ which follows
from the surjectivity of $S$.

Now let us prove that $T$ is nowhere canonizable along a quasi-open mapping.
Fix $U$, $O$ clopen subsets of $X^*$ and $Y^*$ respectively, and suppose
 $\phi$ is as above and quasi-open. Fix a clopen $V\subseteq\phi[U]$.
Let $U'$ be a clopen subset of $X$ such that  $\beta U'\cap X^*=U$.
The set $E$ of integers $n$ such that $U_n=U'\cap (\{n\}\times K)\not=\emptyset$
must be infinite. Since the 
isolated points $\{x_m: m\in \N\}$ are dense in $K$, we may assume, by going to a subset of $U$,
 that $U_n=\{x_{k_n}\}$
for all $n\in E$ and some $k_n\in \N$. Therefore,
$$U'=(\bigcup_{n\in E} \{n\}\times \{x_{k_n}\}).$$
Let $V_n=V'\cap \{n\}\times2^\N$ for $n\in E$, where $\beta V'\cap Y^*=V$.
Let $W_n\subseteq V_n$ be a nonempty clopen such that 
$\tilde\tau(x_{k_n})|W_n$ has its total variation less than $|r|/2$ which can be found
since $2^\N$ has no isolated points. Consider
$$W=(\bigcup_{n\in E} \{n\}\times W_n).$$
Then, $|\beta\tau(n,x_{k_n})(W')|<|r|/2$ for any $W'\subseteq W$ and any $n\in E$.
By the weak$^*$ continuity of $\beta\tau$ 
we have that $|\beta\tau(t)(W')|\leq|r|/2$ for any
 $t\in U$, but this shows that
$\beta\tau(t)$ is not $r\delta_{\phi(t)}$ as required.

\end{proof}

One concrete construction using the methods as above  due to E. van Douwen and J. van Mill is 
 a nowhere dense retract $F\subseteq \N^*$
which is homeomorphic to $\N^*$  and which is a $P$-set 
(see  1.4.3. and 1.8.1. of \cite{vanmillhandbook}).
We will require the following:

\begin{lemma}(CH)\label{nwdpset-isomorphism}
Let $F\subseteq\N^*$ be a nowhere dense $P$-set. The space $\{f\in C(\N^*): f|F=0\}$ is isomorphic to $C(\N^*)$.
\end{lemma}
\begin{proof}
Fix a $P$-point $p\in \N^*$ which exists assuming CH by the results of \cite{rudin}.
Let $(A_\alpha^*: \alpha<\omega_1)$ and 
$(B_\alpha^*: \alpha<\omega_1)$ be sequences of strictly increasing clopen sets such that
$\bigcap_{\alpha<\omega_1}(\N^*\setminus A_\alpha^*)=F$ 
and $\bigcap_{\alpha<\omega_1}(\N^*\setminus B_\alpha^*)=\{p\}$  
(they exists because $F$ is a $P$-set and $p$ is a $P$-point).

Using the standard argument
construct recursively one-to-one, onto functions $\sigma_\alpha: B_\alpha\rightarrow A_\alpha$ such that
$\sigma_\alpha=^*\sigma_\beta|B_\alpha$ for all $\alpha<\beta<\omega_1$. 
Put $\psi_\beta=\sigma_\beta^*:B_\beta^*\rightarrow A_\beta^*$ 
which is the corresponding homeomorphism.

Note that if $f\in C(\N^*)$ is such that 
$f|F=0$, then for every $n\in\N$ there exists $\alpha<\omega_1$ such that 
$\N^*\setminus A_\alpha^*\subseteq f^{-1}[\{t\in\mathbb{R}:|t|<1/(n+1)\}]$. 
Therefore, for each such $f$ there exists $\alpha<\omega_1$ such that $f|(\N^*\setminus A_\alpha^*)=0$.
So define 
$$S: \{f\in C(\N^*): f|F=0\}\rightarrow\{f\in C(\N^*): f(p)=0\}$$
by putting $S(f)=(f\circ\psi_\alpha)\cup 0_{\N\setminus B_\alpha^*}$,
where $\alpha$ is any countable ordinal such that $f|(\N^*\setminus A_\alpha^*)=0$.
It is well defined because the homeomorphisms extend each other, and it is clearly 
a linear isometry. Now it is enough to
note that $\{f\in C(\N^*): f(p)=0\}$ is isomorphic to $C(\N^*)$. To see that this is 
the case, notice that this space
is a hyperplane, and recall that all hyperplanes are isomorphic to each other in any Banach space
(see exercises 2.6 and 2.7 of \cite{fabian}). In the case of $C(\N^*)$ we have 
$$C(\N^*)\sim C(\N^*)\oplus \ell_\infty\sim C(\N^*)\oplus \ell_\infty\oplus \R\sim 
C(\N^*)\oplus \R$$%because $\ell_\infty$ is primary, all of its hyperplanes are isomorphic to $\ell_\infty$.
and so all hyperplanes are isomorphic to the entire $C(\N^*)$.
This completes the proof.
\end{proof}

\begin{proposition}\label{notrightidealCH}(CH) 
The collection of locally null operators is not a right ideal. 
Moreover, the right ideal generated by locally null operators is improper.
\end{proposition}
\begin{proof} Let $F\subseteq \N^*$ be a nowhere dense retract of $\N^*$ homeomorphic to
$\N^*$,  $\psi_1: \N^*\rightarrow F$ the witnessing retraction and $\psi_2: \N^*\rightarrow F$
the homeomorphism. Note that $\psi_2^{-1}\circ\psi_1$ is a well defined continuous
map from $\N^*$ onto itself, and so $T_{\psi_2^{-1}\circ\psi_1}$ is a well defined 
operator from $C(\N^*)$ into itself.
$T_{\psi_2}$ is locally null because $F$ is nowhere dense and hence 
$T_{\psi_2}(f^*)=f^*\circ \psi_2$ is zero for every $f\in\ell_\infty$ such that
 $f^*|F$ is zero. But for every $f\in\ell_\infty$ we have
$$T_{\psi_2}\circ T_{\psi_2^{-1}\circ\psi_1}(f^*)=f^*\circ \psi_2^{-1}\circ\psi_1\circ \psi_2=f^*,$$
because $Im(\psi_2)= F$ and $\psi_1|F=Id_F$. This means that 
$T_{\psi_2}\circ T_{\psi_2^{-1}\circ\psi_1}=Id$, which is not locally null.
Moreover, for any operator $S:\ell_\infty/c_0\rightarrow\ell_\infty/c_0$ we have 
that $S=(T_{\psi_2}\circ T_{\psi_2^{-1}\circ\psi_1})\circ S$, which is in the 
right ideal generated by locally null operators.

\end{proof}

Nowhere dense $P$-sets homeomorphic to $\N^*$ which are retracts give also
more concrete (compared to  \ref{nowherecanonizable}) examples of automorphisms 
failing canonizability like in \ref{ocatrivialization}.

\begin{example}[CH]\label{notlocalstrongcanon}
There is an  automorphism $T$ of 
$\ell_\infty/c_0$  with the following properties:
\begin{enumerate}
\item $T$ is not fountainless 
\item$T$ is not left-locally  canonizable along any continuous map.
\item $T^{-1}$ is not funnelless
\item $T^{-1}$ is not right-locally canonizable  along any continuous map.
\end{enumerate}
\end{example}
\begin{proof} 
 Let $F$ be a nowhere dense retract of $\N^*$ which is a $P$-set and is 
homeomorphic to $\N^*$. 
Let $\psi_1: \N^*\rightarrow F$ be 
the witnessing retraction. We will need one more additional property of $F$, namely
that $\psi_1$ is not one-to-one while restricted to any nonempty clopen set.
This can be obtained by modifying the construction of 1.4.3. of \cite{vanmillhandbook}
by replacing $W(\omega_1+1)$ with the ``zero-dimensional long line", i.e., 
the space $K$ obtained by gluing Cantor sets inside every ordinal interval $[\alpha, \alpha+1)$
for $\alpha<\omega_1$, obtaining a nonmetrizable
subspace $K$ of the long line which contains $W(\omega_1+1)$ and has no isolated points. 
One takes $\tilde\pi:K\rightarrow K$ which collapses the  entire $K$ to the point $\omega_1$, 
$X=\N\times K$,  and $\pi: X\rightarrow X$ given by $\pi(n, x)=(n, \tilde\pi(x))$.
As in  1.4.3. of \cite{vanmillhandbook} one proves that $\beta\pi[X^*]\subseteq X^*$
and $\psi_1=\beta\pi|X^*$ is the required retraction.
The argument why $\psi_1$ is not a one-to-one while restricted to any
clopen set is similar to the one from the proof of Theorem 2.1. from \cite{vanmillch}:
if $U\subseteq X^*$ clopen, it is of the form $\beta U'\cap X^*$ where
$$U'=\bigcup_{n\in E} \{n\}\times U_n$$
for some infinite $E\subseteq \N$ and nonempty clopen sets   $U_n\subseteq K$ (consider $\chi_U$
and the relation of $X$ to $\beta X$). But these nonempty open sets have at least two
points $x_n, y_n$ as $K$ has no isolated points. Of course $\pi(n, x_n)=(n,\omega_1)=\pi(n, y_n)$.
Consider $x=\lim_{n\in u}x_n$ and $y=\lim_{n\in u}y_n$ ($u$ is a nonprincipal ultrafilter in $\wp(\N)$)
which can be easily separated, so $x\not=y$ and $x, y\in U$. However 
$\psi_1(x)=\lim_{n\in u}\pi(n, x_n)=\lim_{n\in u}\pi(n, y_n)=\psi_1(y)$.

We can decompose $C(\N^*)=X\oplus Y$ where
$$X=\{g\circ\psi_1: g\in C(F)\}, \ \ \  Y=\{f\in C(\N^*): f|F=0\}.$$%h=h|F\circ\psi+(h-h|F\circ\psi)
The first factor is  isometric to $C(F)$ (the isometry is defined by restricting to $F$),
 which in turn is isometric to $C(\N^*)$ because
of the homeomorphism between $F$ and $\N^*$. By Lemma \ref{nwdpset-isomorphism},
the second factor is also isomorphic $C(\N^*)$.

 Fix an infinite, coinfinite $A\subseteq \N$. Let 
$S: Y\rightarrow C(\N^*\setminus A^*)$
be an isomorphism.
Let $\psi_2: A^*\rightarrow F$ be a homeomorphism.
 Finally, let $T: C(\N^*)\rightarrow C(\N^*)$
be the isomorphism defined by 
$$T=I_A\circ T_{\psi_2}+I_{\N^*\setminus A^*}\circ S\circ(Id-T_{\psi_1}).$$
That is, roughly speaking,  $T$ sends $X$ to $C(A^*)$ and $Y$ to $C(\N^*\setminus A^*)$.
For  $y\in A^*$ we have  $\dual{T}(\delta_y)=\dual{(I_A\circ T_{\psi_2})}(\delta_y)=\delta_{\psi_2(y)}$,
 i.e., $\dual{T}(\delta_y)$ is concentrated on $F$, so $(F, A)$ is a fountain for $T$.
This also implies that $T$ cannot be
canonized along a homeomorphism onto a clopen set below $A^*$ because $F$ is nowhere dense.

On the other hand, 
$$T^{-1}=T_{\psi_1} \circ T_{\psi_2^{-1}}\circ P_A +   S^{-1}\circ P_{\N\setminus A}$$
So  $\dual{(T^{-1})}(\delta_x)|A^*=  \delta_{\psi_2^{-1}(\psi_1(x))}$
for every $x\in \N^*$. In particular $(A^*, F)$ is a funnel for $T^{-1}$ 
and $T^{-1}$ 
 cannot be  canonized on a pair $(A_0, B_0)$ for infinite $A_0\subseteq A$,
 $B_0\subseteq \N$ because $\psi_2^{-1}\circ\psi_1$ is not one-to-one on
any clopen set $B_0\subseteq \N$ by the choice of $\psi_1$.

\end{proof}

\begin{theorem}[CH]\label{nowheretrivialch} 
There is an  automorphism of $\ell_\infty/c_0$ with no fountains and no funnels which
is  nowhere trivial.
\end{theorem}
\begin{proof}
Let $\psi:\N^*\rightarrow \N^*$ be nowhere trivial homeomorphism of $\N^*$. The 
existence of such a  homeomorphism is a folklore result, its first construction
 is implicitly included in \cite{rudin}.
By Propositions \ref{locallydet=quasi-open} and \ref{funnelless=nwdp}, $T_\psi$  has no fountains nor funnels. 
It is not locally trivial because $\psi$ is not trivial on any clopen set.
\end{proof}

\begin{theorem}[CH]\label{quasi-opench} 
 There is a quasi-open surjective map $\psi:\N^*\rightarrow\N^*$ such that 
the images of nowhere dense sets under $\psi$ are nowhere dense  and it is 
not a bijection while restricted to any clopen set. Therefore, $T_\psi$ is
an everywhere present isomorphic embedding of $\ell_\infty/c_0$ into itself  with no fountains 
and with no funnels which  is nowhere canonizable along a homeomorphism.

\end{theorem}
\begin{proof}
It is enough to construct a quasi-open irreducible surjection $\psi:\N^*\rightarrow \N^*$
which is not a homeomorphism when restricted to any clopen set
and consider $T_\psi$ by \ref{irreduciblequasi-open}, \ref{locallydet=quasi-open} and \ref{funnelless=nwdp}.
Let $\tilde\phi: 2^\N\rightarrow 2^\N$ be an irreducible surjection which is not
a bijection while restricted to any clopen subset of $2^\N$ (e.g., obtained via the Stone duality
by taking a dense atomless subalgebra of the free countable algebra which is proper below any element,
see \ref{denseirreducible}).
Consider $X=\N\times 2^\N$ and $\phi: X\rightarrow X$, given by $\phi(n, x)=(n, \tilde\phi(x))$.
By a topological consequence of  Parovi\v cenko's theorem 
(see Theorem 1.2.6. of \cite{vanmillhandbook}) $X^*=\beta X\setminus X$ is homeomorphic to
$\N^*$. Moreover $\beta\phi:\beta X\rightarrow \beta X$ sends $X^*$ into $X^*$.

To check that $\psi=\beta\phi|X^*$ is irreducible take any clopen $U\subseteq X^*$,
which must be of the form $\beta U'\cap X^*$, where
$$U'=\bigcup_{n\in E} \{n\}\times U_n$$
for some infinite $E\subseteq \N$ and nonempty clopen sets   $U_n\subseteq K$ (consider $\chi_U$
and the relation of $X$ to $\beta X$). By the irreducibility of $\tilde\phi$, there are clopen $V_n\subseteq 2^\N$
such that $\tilde\phi[2^\N\setminus U_n]\cap V_n=\emptyset$. So 
$$\beta\phi[U]\cap \beta (\bigcup_{n\in E} \{n\}\times V_n)=\emptyset,$$
which completes the proof of the irreducibility of $\psi$.
The argument why $\psi$ is not a one-to-one while restricted to any
clopen set is similar to the one from the proof of Example \ref{notlocalstrongcanon}.  
\end{proof}

A similar example as above is constructed in the proof of Theorem 2.1 from \cite{vanmillch} however
it does not have the property of preserving nowhere dense sets.

\section{Open problems and final remarks}
\label{sec:problems}
In this section we mention some open problems and some  observations
related to them. This should not be considered as a full list of urgent open problems
concerning the Banach space $\ell_\infty/c_0$, for example we do not touch problems related to the 
primariness of $\ell_\infty/c_0$ (see \cite{drewnowski}, \cite{dowpfa}, \cite{opit})
subspaces of $\ell_\infty/c_0$ (\cite{univ}, \cite{radova}, \cite{stevoembeddings}, \cite{krupski}),
$\ell_\infty$-sums (\cite{drewnowski}, \cite{dowpfa}, \cite{christina})
or extensions of operators on $\ell_\infty/c_0$ (\cite{castilloplichko}, \cite{separablyinjective}).

\begin{problem} Is it consistent (does it follow from PFA or OCA+MA)  that every automorphism $T: \ell_\infty/c_0\rightarrow\ell_\infty/c_0$
can be lifted modulo a locally null operator? That is, is every such operator 
 of the form $T=[R]+S$, where $R: \ell_\infty\rightarrow \ell_\infty$ and 
$S$ is locally null?
\end{problem}

This is related to the fact that our ZFC nonliftable operator (see \ref{nonliftable})
is of the above form. A ZFC possibility of somewhere canonizing every isomorphic embedding is excluded by 
\ref{quasi-opench} or \ref{nowherecanonizable}. As under PFA or OCA+MA canonization along a homeomorphism gives
trivialization we may ask:

\begin{problem} Is it consistent (does it follows from PFA or OCA+MA)  that every isomorphic
embedding $T: \ell_\infty/c_0\rightarrow\ell_\infty/c_0$
is somewhere trivial?
\end{problem}

\begin{problem} Is it true in ZFC that for every isomorphic
embedding $T: \ell_\infty/c_0\rightarrow\ell_\infty/c_0$
there is an infinite $A$, a closed $F\subseteq \N^*$ and 
a homeomorphism $\psi: F\rightarrow A^*$ such that
$T(f^*)|F=r(f^*\circ \psi)$ for $A$-supported $f$'s and some nonzero $r\in \R$?
\end{problem}

The positive solution of this problem would give  the positive solution 
to Problem \ref{furthercopyproblem}.

\begin{problem} Is it consistent (does it follow from PFA, or OCA+MA) that
every  automorphism of $\ell_\infty/c_0$ is somewhere trivial?
\end{problem}

In other words we ask here if the hypothesis in \ref{ocatrivialization} of $T$ being funnelless
or fountainless is needed under PFA or OCA+MA. In principle
there may not be any fountains or funnels
of automorphisms of $\ell_\infty/c_0$ under these assumptions. The only examples we have of such 
phenomena are for automorphisms  under CH (\ref{notlocalstrongcanon}).  

Note that by Plebanek's result \ref{plebanekiso} %\cite{plebanekisrael} 
and by going to a subset of $F$
using \ref{irreducibledrewnowski} and \ref{irreducible}, 
we may assume that there is an infinite $A\subseteq \N$,
a closed $F$ and a continuous $\psi: F\rightarrow A^*$ such that
$T(f^*)|F=f^*\circ\psi$ for every $A$-supported $f$. If we knew that $(A^*, F)$
is not a funnel i.e., that $F$ is not nowhere dense, we could use Farah's result \ref{ocacontinuoustopo} as in
the proof of \ref{ocatrivialization} to obtain somewhere trivialization of $T$.
We were not able to prove, however,  a similar reduction for fountains, which could be
more useful in the context
of applying Farah's result as then $F$ would be a continuous image of $A^*$ which is a copy of $\N^*$
so the domain of $\psi$ is as required in \ref{ocacontinuoustopo}.

One strategy for proving that under OCA+MA automorphisms  do not have funnels 
or fountains is to use
the result of I. Farah \ref{ocacontinuoustopo} directly to prove that the sets $F$ appearing in
potential funnels or fountains cannot be nowhere dense. For this we would need
to know that (a) such $F$'s  are not c.c.c. over Fin, and that (b) they (or their
clopen sets) are continuous (or homeomorphic) images of $\N^*$. 
There are several related natural questions which we were unable to solve.

\begin{problem} Suppose that $T: C(\N^*)\rightarrow C(\N^*)$ is 
an automorphism (an isomorphic embedding), $A\subseteq\N$ is infinite, $F\subseteq \N^*$ is closed
nowhere dense and  $\psi: F\rightarrow A^*$ is continuous irreducible such that
$T(f^*)|F=f^*\circ \psi$ for each $A$-supported $f\in \ell_\infty$.
Is it  consistent (under OCA+MA, or PFA) that $F$ cannot be c.c.c. over Fin?
\end{problem}

\begin{problem} Suppose that $T: C(\N^*)\rightarrow C(\N^*)$ is 
an automorphism   and $F\subseteq \N^*$ is closed
nowhere dense and $A\subseteq \N$ infinite such that $(F, A^*)$ is a
fountain for $T$ or $(A^*, F)$ is a funnel for $T$.
Is it true or consistent (under OCA+MA, or PFA) that $F$ cannot be c.c.c. over Fin?
\end{problem}

A more general problem is

\begin{problem} Suppose that $F\subseteq \N^*$ is nowhere dense  closed not c.c.c
(or even homeomorphic to $\N^*$)
which has the linear extension property. Is it 
 consistent (under OCA+MA, or PFA) that $F$ cannot be c.c.c. over Fin?
\end{problem}

Recall that $F\subseteq \N^*$ has the linear extension property if there
is a linear bounded operator $T: C(F)\rightarrow C(\N^*)$ such that
$T(f)|F=f$ for each $f\in C(F)$. The existence of such an operator is a weak version of
the existence of a retraction from $\N^*$ onto $F$. Alan Dow in \cite{dowpfa}
developed new methods (which may be quite useful in the above context) proving that PFA implies
that the cozero sets do not have the linear extension property.
The last couple of problems is related to possible applications of canonizations
of embeddings.

\begin{problem}Is it consistent that every copy of 
$\ell_\infty/c_0$ inside $\ell_\infty/c_0$  is complemented?
\end{problem}

\begin{problem}\label{furthercopyproblem}Is it true or consistent that every copy of 
$\ell_\infty/c_0$ inside $\ell_\infty/c_0$ contains a further copy of $\ell_\infty/c_0$
which is complemented in the entire space?
\end{problem}

One should note that under CH examples of uncomplemented copies of 
$\ell_\infty/c_0$ inside $\ell_\infty/c_0$ were constructed in \cite{castilloplichko}.
They can also be obtained under CH from a superspace of $\ell_\infty/c_0$ 
 obtained in \cite{antoniogaps} in which $\ell_\infty/c_0$
is not complemented.

\section{Appendix}

\subsection{$C(\N^*)$}

\begin{lemma}\label{gdelta}
Every nonempty $G_\delta$ set in $\N^*$ has a nonempty interior.
\end{lemma}
\begin{proof}
See section 1.2 of \cite{vanmillhandbook}. 
\end{proof}

\begin{lemma}\label{constantfunction} Suppose $f:\N^*\rightarrow \R$ is continuous and
$r\in \R$ is a value of $f$ at some point. Then there is
a clopen $A^*\subseteq \N^*$ such that $f|A^*\equiv r$.
\end{lemma}
\begin{proof}
$f^{-1}[\{r\}]=\bigcap_{n\in \N} f^{-1}[\{t\in\mathbb{R}: r-1/n<t< r+1/n\}]$ is a nonempty $G_\delta$ set.
By \ref{gdelta}, we obtain an infinite $A\subseteq \N$ such that $A^*\subseteq f^{-1}[\{r\}]$.
\end{proof}

\begin{definition} A surjective map is called irreducible if, and only if, it is not surjective
when restricted to any proper closed subset.
\end{definition}

\begin{lemma}\label{irreducible} If $\psi: K\rightarrow L$ is surjective and $K $ and $L$ are compact, then there
is a  closed $F\subseteq K$ such that $\psi|F: F\rightarrow L$ is irreducible.
\end{lemma}

\begin{lemma}\label{denseirreducible} 
Suppose that $\psi: F\rightarrow \N^*$ is a continuous surjection, where
$F\subseteq \N^*$ is a closed subset of $\N^*$.
$\psi$ is irreducible if, and only if, $\{\psi^{-1}[A^*]: A\subseteq \N^*\}$ is a 
dense subalgebra of clopen subsets of $F$.
\end{lemma}
\begin{proof}
If $U\subseteq F$ were a clopen subset of $F$ such that $\psi^{-1}[A^*]\subseteq U$
does not hold for any infinte $A\subseteq \N$, then $\psi|(F\setminus U)$ is onto $\N^*$
contradicting the irreducibility.
If $U\subseteq F$ were a clopen subset of $F$ such that  $\psi|(F\setminus U)$ is onto $\N^*$,
then $\psi^{-1}[A^*]\subseteq U$ cannot hold for any infinite $A\subseteq \N$.
\end{proof}

\begin{lemma}\label{irreduciblequasi-open} Irreducible maps are quasi-open
and map nowhere dense sets onto nowhere dense sets.
\end{lemma}
\begin{proof} Suppose that $\psi: F\rightarrow G$ is irreducible.
If the interior of $\psi[U]$ is empty for some open $U\subseteq F$, then
it means that $\psi[F\setminus U]$ is dense in $G$, but $\psi[F\setminus U]$
is compact, and  so is equal to $G$ contradicting the irreducibility of $\psi$.

Now suppose that $K\subseteq F$ is nowhere dense whose image contains an
open $U\subseteq G$. As $\psi^{-1}[U]$ is open,  there is $V\subseteq \psi^{-1}[U]$
such that $V\cap K=\emptyset$ and so $\psi[V]\subseteq \psi[K]$. Note that 
$\psi[F\setminus V]=G$ contradicting the irreducibility of $\psi$.

\end{proof}

\begin{lemma}\label{constantirreducible} Suppose $f:F\rightarrow \R$ is continuous,
$F\subseteq \N^*$ is compact and
there is an irreducible map $\psi: F\rightarrow \N^*$. Then, there is
an infinite $A\subseteq \N$ such that $f|\psi^{-1}[A^*]$ is constant.
\end{lemma}
\begin{proof} Construct infinite $A_n\subseteq \N$ such that $A_{n+1}\subseteq_* A_n$ and 
intervals $I_n\subseteq \R$ such that the diameter of $I_n$ is less than $1/n$ and such that
$f|\psi^{-1}[A^*_n]\subseteq I_n$. The irreducibility guarantees the recursive step
through Lemma \ref{denseirreducible}.
If $A\subseteq_*A_n$ for all $n$, then $f|\psi^{-1}[A^*]$ is constant.
\end{proof}

\begin{lemma}\label{countablenwd} 
Countable unions of nowhere dense sets in $\N^*$ are nowhere dense.
\end{lemma}
\begin{proof}
 Let $F_n\subseteq \N^*$ be nowhere dense for every $n\in\N$. We may assume each 
$F_n$ to be closed.  Fix 
an open set $U\subseteq \N^*$. Let $B_0\subseteq U\setminus F_0$ be a nonempty
clopen and choose by induction 
$B_{n+1}\subseteq B_n\setminus F_{n+1}$ nonempty clopen. Since there exists a 
nonempty clopen 
$V\subseteq\bigcap_{n\in\N}B_n \subseteq U\setminus\bigcup_{i\in\N}F_i$, we know that 
$U\not\subseteq \overline{\bigcup_{n\in\omega}F_n}$.

\end{proof}

\subsection{Operators on $\ell_\infty$ preserving $c_0$}

\begin{lemma}\label{pseudo_diagonal2}
 Let $(b_{ij})_{i,j\in\mathbb{N}}$ be a $c_0$-matrix. If $J\subseteq \mathbb{N}$
is such that $(\sum_{j\in J}|b_{ij}|)_i\notin c_0$, then 
there exist $\varepsilon>0$, an infinite set $B$ and 
 finite $F_n\subseteq J$ for each $n\in B$, such that

\begin{enumerate}
 \item  $F_n\cap F_k=\emptyset$, for distinct $n,k\in B$,
\item $\sum_{j\in F_i}|b_{ij}|=|\sum_{j\in F_i}b_{ij}|>\varepsilon/4$, for all $i\in B$, and
\item $\lim_{\stackrel{i\rightarrow\infty}{i\in B}}\sum_{j\in \cup_{k\neq i}F_k}|b_{ij}|=0$
% for all $\delta>0$ there is an $m\in\mathbb{N}$ such that
%  $$\sum_{j\in \cup_{k\neq i}F_k}|b_{ij}|<\delta,\quad\text{ for all }i>m.$$
\end{enumerate}

\end{lemma}

\begin{proof}
  Let $(b_{ij})_{i,j\in\mathbb{N}}$ be a $c_0$-matrix and fix 
a $J\subseteq\mathbb{N}$ as in the hypothesis. Since $(b_{ij})_{i,j\in\mathbb{N}}$ is a $c_0$-matrix,
we know that for every $k\in\mathbb{N}$ the sequence 
$(\sum_{j\leq k}|b_{ij}|)_i$ converges to zero. Hence,
$J$ must be infinite.
Let $J=\{j_n:n\in\mathbb{N}\}$ be the increasing enumeration of $J$. 
By hypothesis, there exist $\varepsilon>0$ and an infinite $\tilde{B}\subseteq \mathbb{N}$
such that $\sum_{n\in\mathbb{N}}|b_{ij_n}|>\varepsilon$, for all $i\in\tilde{B}$.
% Let $\tilde{B}=\{k_n:n\in\mathbb{N}\}$ be the increasing enumeration of $\tilde{B}$. 

We will carry out an inductive construction from where we will obtain the 
sequence $(F_n)$ and the set $B$. Let $i_0$ be the first element of $\tilde{B}$ and $m_0=0$.
Since $\sum_{n\in\mathbb{N}}|b_{i_0j_n}|$ converges, we may choose $m_1>m_0$ such that
$\sum_{n\geq m_1}|b_{i_0j_n}|<\varepsilon/2$. Suppose for every $l\leq k$ we have chosen 
$i_l$ and $m_{l+1}$ satisfying
\begin{enumerate}
 \item[(a)] $m_l<m_{l+1}$, $i_l\in \tilde{B}$, and $i_{l-1}<i_l$,
\item[(b)] $\sum_{n< m_l}|b_{i_lj_n}|<\frac{\varepsilon}{4(l+1)}$, and
\item[(c)]  $\sum_{n\geq m_{l+1}}|b_{i_lj_n}|<\frac{\varepsilon}{4(l+1)}$.
\end{enumerate}

Since $(\sum_{j<m_{k+1}}|b_{ij}|)_i$ converges to zero, we may choose
$i_{k+1}\in\tilde{B}$, such that $i_{k+1}>i_k$ and for all $i\geq i_{k+1}$ we
have $\sum_{n< m_{k+1}}|b_{ij_n}|<\frac{\varepsilon}{4(k+2)}$.
Furthermore, since $\sum_{n\in\mathbb{N}}|b_{i_{k+1}j_n}|$ converges,
we may choose $m_{k+2}>m_{k+1}$ such that 
$\sum_{n\geq m_{k+2}}|b_{i_{k+1}j_n}|<\frac{\varepsilon}{4(k+2)}$.
This finishes the inductive construction.

Notice that for every $k\in\mathbb{N}$ we have
$$\begin{array}{rcl}
\varepsilon <\sum_{n\in\mathbb{N}}|b_{i_kj_n}|&=&
\sum_{n< m_k}|b_{i_kj_n}|+ \sum_{m_k\leq n<m_{k+1}}|b_{i_kj_n}|+\sum_{n\geq m_{k+1}}|b_{i_kj_n}|\\
&<&  \frac{\varepsilon}{4(k+1)}+\sum_{m_k\leq n<m_{k+1}}|b_{i_kj_n}|+\frac{\varepsilon}{4(k+1)}\\
&\leq&\sum_{m_k\leq n<m_{k+1}}|b_{i_kj_n}|+\varepsilon/2.
  \end{array}$$

Hence, $\sum_{m_k\leq n<m_{k+1}}|b_{i_kj_n}|>\varepsilon/2$. 
By splitting the sum $\sum_{m_k\leq n<m_{k+1}}b_{i_kj_n}$
into its positive and negative parts, we obtain 
$F_{i_k}\subseteq \{j_n\in\mathbb{N}:m_k\leq n<m_{k+1}\}$ such that 
$|\sum_{j\in F_{i_k}}b_{i_kj}|>\varepsilon/4$. 
So by letting $B=\{i_k:k\in\mathbb{N}\}$, we know that conditions (1) and (2)
of the lemma are satisfied.
To obtain (3), fix $\delta>0$ and take $m\in\mathbb{N}$ such that
 $\frac{\varepsilon}{2(m+1)}<\delta$. By construction we have that
$\bigcup_{l\neq k}F_{i_l}\subseteq\{j_n\in\mathbb{N}:n<m_k\text{ or }n\geq m_{k+1}\}$,
for every $k\in\mathbb{N}$. So, in particular for every $k>m$, we have
$$ \sum_{j\in \cup_{l\neq k}F_{i_l}}|b_{i_kj}|\leq \sum_{n< m_k}|b_{i_kj_n}|+\sum_{n\geq m_{k+1}}|b_{i_kj_n}|
<\frac{\varepsilon}{2(k+1)}<\delta.$$
\end{proof}

In the following propositions we list some facts leading to the proof of 
Theorem \ref{summary_matrices}.

We start by noting that
 the two conditions that characterize matrices which induce
operators from $\ell_{\infty}$ into $\ell_{\infty}$, also characterize matrices
whose transpose induces operators from $\ell_1$ into $\ell_1$.
\begin{proposition}\label{charac_l_1-matrix}
$S:\ell_1\rightarrow\ell_1$ is a  linear bounded operator if, and only if,
there exists a matrix $(b_{ij})_{i,j\in\mathbb{N}}$ which induces $S$ and is
 such that every column is in $\ell_1$ and the set 
$\{\|(b_{ij})_{i\in\N}\|_{\ell_1}:j\in\mathbb{N}\}$ is bounded.
\end{proposition}

\begin{proof}
Let $S:\ell_1\rightarrow\ell_1$ be a linear bounded operator.
Since $\ell_1^*=\ell_{\infty}$, for each $i\in\mathbb{N}$
 there exists $(b_{ij})_{j\in\N}\in\ell_{\infty}$ such that 
$S(a)(i)=\dual{S}(\delta_i)(a)=\sum_{j\in\N} b_{ij}a_j$, for every $a=(a_k)_{k\in\N}\in\ell_1$.
In other words, $S$ is given by the matrix $(b_{ij})_{i,j\in\mathbb{N}}$.

Note that the $j$-th column of the matrix is equal to $S(\delta_j)\in\ell_1$.
 Moreover, since $S$ is bounded,
we have that $S(\{\delta_j:j\in\mathbb{N}\})=\{(b_{ij})_{i\in\N}:j\in\mathbb{N}\}$
is bounded in $\ell_1$.

Conversely, suppose $(b_{ij})_{i,j\in\mathbb{N}}$ is a matrix
 such that every column $(b_{ij})_{i\in\N}$ is in $\ell_1$ and the set 
$\{\|(b_{ij})_{i\in\N}\|_{\ell_1}:j\in\mathbb{N}\}$ is bounded. We claim that this matrix
induces a linear bounded operator $S:\ell_1\rightarrow\ell_1$. 

First, we shall prove that such an operator is well defined.
Fix $a=(a_k)_{k\in\N}\in\ell_1$. Start by noting that for every $i\in\mathbb{N}$ we have that 
$(b_{ij})_{j\in\N}$ is bounded by hypothesis, and so is in $\ell_{\infty}$. 
Hence, $\sum_{j\in\mathbb{N}}b_{ij}a_j$ is
convergent for every $i\in\mathbb{N}$. Now, we need to show that the sequence 
$(\sum_{j\in\mathbb{N}}b_{ij}a_j)_{i\in\N}$ is in $\ell_1$.

In view of applying Theorem 8.43 of \cite{apostol},
note that by hypothesis $\sum_{i\in\mathbb{N}}b_{ij}a_j$
is absolutely convergent for every $j\in\mathbb{N}$, and
there exists $M\in\mathbb{N}$ such that 
$\sum_{i\in\mathbb{N}}|b_{ij}|<M$ for every $j\in\mathbb{N}$.
Therefore, 
$$\sum_{j\in\mathbb{N}}\sum_{i\in\mathbb{N}}|b_{ij}a_j|=
\sum_{j\in\mathbb{N}}|a_j|\sum_{i\in\mathbb{N}}|b_{ij}|\leq \sum_{j\in\mathbb{N}}|a_j|M<\infty.$$
Hence, by the cited Theorem, we have that both iterated series $\sum_{i\in\mathbb{N}}\sum_{j\in\mathbb{N}}b_{ij}a_j$
and $\sum_{j\in\mathbb{N}}\sum_{i\in\mathbb{N}}b_{ij}a_j$ converge absolutely. 
In particular, we have that  
$$\sum_{i\in\mathbb{N}}|\sum_{j\in\mathbb{N}}b_{ij}a_j|\leq 
\sum_{i\in\mathbb{N}}\sum_{j\in\mathbb{N}}|b_{ij}a_j|<\infty,$$ so that 
$S(a)\in\ell_1$.

Clearly, $S$ is linear. Moreover, if we take $a=(a_k)_{k\in\N}\in\ell_1$ such that
$\|a\|_{\ell_1}\leq 1$, then by a similar argument as above we have
$$\|S(a)\|_{\ell_1}\leq%=\sum_{j\in\mathbb{N}}|\sum_{i\in\mathbb{N}}b_{ij}a_i|\leq 
\sum_{i\in\mathbb{N}}\sum_{j\in\mathbb{N}}|b_{ij}a_j|=
\sum_{j\in\mathbb{N}}\sum_{i\in\mathbb{N}}|b_{ij}a_j|\leq 
\sum_{j\in\mathbb{N}}|a_j|M\leq M.$$
Therefore, $S$ is bounded.

\end{proof}

In the following proposition we identify $\dual{\ell_1}$ with $\ell_\infty$.

\begin{proposition}\label{T-dual-transpose}
 Let $R:\ell_{\infty}\rightarrow\ell_{\infty}$ be a linear bounded operator.
Then $R=\dual{S}$ for some linear bounded operator $S:\ell_1\rightarrow\ell_1$
if, and only if, $R$ is given by a matrix.
Moreover, the matrix corresponding to the operator $S$ is the transpose of
the matrix corresponding to $R$.
\end{proposition}

\begin{proof}
Assume $R=\dual{S}$, for a linear bounded operator $S:\ell_1\rightarrow\ell_1$.
Let $M_S=(b_{ij})_{i,j\in\mathbb{N}}$ be the matrix corresponding to $S$.
Since $S(\delta_j)$ is the $j$-th column of
$M_S$, for every $f\in\ell_{\infty}$  we have that
$$R(f)(j)=\dual{S}(f)(j)=(f\circ S)(\delta_j)=\sum_{i\in\mathbb{N}}b_{ij}f(i).$$

In other words, $R$ is given by the transpose of $M_S$.\\

Conversely, suppose $R$ is given by a matrix $M_R=(b_{ij})_{i,j\in\mathbb{N}}$.
By \ref{l_infty-matrix} and \ref{charac_l_1-matrix}  we have that the transpose of 
$M_R$ defines a linear bounded operator
$S:\ell_1\rightarrow\ell_1$. It only remains to show that $\dual{S}=R$.
So fix for this purpose $f\in\ell_{\infty}$. If we regard it as an element of
$\dual{\ell_1}$, then for every $i\in\N$ we have
$$ \dual{S}(f)(\delta_i) = f\circ S(\delta_i)=\sum_{j\in\mathbb{N}}f(j)b_{ij},$$
but on the other hand,
$$ R(f)(i)= \sum_{j\in\mathbb{N}}f(j)b_{ij}.$$
\end{proof}

In the following proposition we identify $\dual{c_0}$ with $\ell_1$.

\begin{proposition}\label{S-dual-transpose}
Let $R:c_0\rightarrow c_0$ be a linear bounded operator. Then,
 $\dual{R}$ is induced by the transpose of the matrix corresponding
to $R$.
\end{proposition}
\begin{proof}
 Let $M_R=(b_{ij})_{i,j\in\mathbb{N}}$ and $M_{\dual{R}}=(b'_{ij})_{i,j\in\mathbb{N}}$ 
be the matrices corresponding to $R$ and $\dual{R}$, respectively.
Observe that for every $i,j\in\mathbb{N}$ we have that 
$b'_{ij}=\dual{R}(\delta_j)(\chi_{\{i\}})=\delta_j(R(\chi_{\{i\}}))=b_{ji}$. Hence, $M_{\dual{R}}$ is the 
transpose of $M_R$.
\end{proof}

\begin{proposition}\label{doubleadjoint_appendix}
A linear  bounded operator $R:\ell_\infty\rightarrow\ell_\infty$ is given by
a  $c_0$-matrix if, and only if, $R=\ddual{(R|c_0)}$.
\end{proposition}
\begin{proof}
 Assume $R$ is given by a $c_0$-matrix $M$. Then, $R[c_0]\subseteq c_0$
and so $R|c_0:c_0\rightarrow c_0$ is well defined and is also
 induced by $M$. By propositions 
 \ref{T-dual-transpose} and \ref{S-dual-transpose} we have that $\ddual{(R|c_0)}$
is also induced by $M$. In other words, $R=\ddual{(R|c_0)}$.

Conversely, assume $R=\ddual{(R|c_0)}$.
Let $M$ be the matrix corresponding to $R|c_0:c_0\rightarrow c_0$.
By propositions \ref{T-dual-transpose} and \ref{S-dual-transpose}  we 
have that $R=\ddual{(R|c_0)}$ is also induced by $M$, which is a $c_0$-matrix.
\end{proof}

\begin{theorem}\label{c_0-matrix--weakly_compact_appendix}
Let $R:c_0\rightarrow c_0$ be a linear bounded operator and let $(b_{ij})_{i,j\in\N}$
be the corresponding matrix. The following are equivalent:
\begin{enumerate}
 \item $R$ is weakly compact.
\item $\ddual{R}[\ell_{\infty}]\subseteq c_0$.
\item $\|b_i\|_{\ell_1}\rightarrow 0$, where $b_i=(b_{ij})_{j\in\N}$. 

\end{enumerate}
\end{theorem}

\begin{proof}
The equivalence of of (1) and (2) is well-known (see exercise 3 of Chapter 3 in \cite{diestel}).

By \ref{doubleadjoint_appendix}, we have that $\ddual{R}$ is given $(b_{ij})_{i,j\in\N}$.
Therefore, for any $f\in\ell_{\infty}$ and any $i\in\N$ we have
 $|\ddual{R}(f)(i)|=|\sum_{j\in\N}b_{ij}\cdot f(j)|\leq \sum_{j\in\N}|b_{ij}|\|f\|_{\ell_{\infty}}$,
and it is clear that (3) implies (2).

For the converse, assume (3) does not hold. Then, by Lemma \ref{pseudo_diagonal2}
there exist $\varepsilon>0$, an infinite set $A\subseteq \N$ and 
 finite $F_n\subseteq \N$ for each $n\in A$, such that

\begin{enumerate}
 \item[(i)]  $F_n\cap F_k=\emptyset$, for distinct $n,k\in A$,
\item[(ii)] $\sum_{j\in F_i}|b_{ij}|=|\sum_{j\in F_i}b_{ij}|>\varepsilon/4$, for all $i\in A$, and
\item[(iii)] there is an $m\in\mathbb{N}$ such that
 $$\sum_{j\in \cup_{k\neq i}F_k}|b_{ij}|<\varepsilon/8,\quad\text{ for all }i>m.$$
\end{enumerate}

Let $f\in\ell_{\infty}$ be such that $supp(f)\subseteq \bigcup_{n\in\N} F_n$
and $b_{ij}\cdot f(j)=|b_{ij}|$, for every $i\in A$ and every $j\in F_i$.
Then, for every $i\in A\setminus m$ we have

$$\begin{array}{rcl}
   |\ddual{R}(f)(i)|&=& |\sum_{j\in\bigcup_{n\in\N}F_n}b_{ij}\cdot f(j)|\\
&\geq&|\sum_{j\in F_i}b_{ij}\cdot f(j)| -|\sum_{j\in\bigcup_{k\neq i}F_k}b_{ij}\cdot f(j)|\\
&\geq& \sum_{j\in F_i}|b_{ij}|-\sum_{j\in\bigcup_{k\neq i}F_k}|b_{ij}|\\
&>&\varepsilon/4-\varepsilon/8=\varepsilon/8. 
  \end{array}$$

Therefore, $\ddual{R}(f)\notin c_0$.

\end{proof}

\begin{proposition}\label{double-dual}
 Let $T:\ddual{X}\rightarrow \ddual{X}$. Then, $T=\ddual{R}$  for some 
$R:X\rightarrow X$, and only if,
$T$ is $w^*$-$w^*$-continuous and $T[X]\subseteq X$.
\end{proposition}
\begin{proof}
 Assume $T=\ddual{R}$ for some 
$R:X\rightarrow X$. It is well known %(see exercise 3.20, page 90, of \cite{fabian})
that the fact that $T$ is a dual operator implies that it is $w^*$-$w^*$-continuous.
Since $R[X]\subseteq X$ and $R\subseteq \ddual{R}$, we have that $T[X]\subseteq X$.

Conversely, suppose $T$ is $w^*$-$w^*$-continuous and $T[X]\subseteq X$.
%Once more by the result cited above, we have that
Then,  $T$ is a dual operator,
say of $S:\dual{X}\rightarrow \dual{X}$. It is sufficient to show that $S$ is
$w^*$-$w^*$-continuous. So fix an open set $U\subseteq\mathbb{R}$ and an $x\in X$. 
We denote by $\tilde{x}\in \ddual{X}$
the element corresponding to $x\in X$. Consider the preimage
by $S$ of the  $w^*$-open subbasic $\{g\in \dual{X}:g(x)\in U\}$,
for given $x\in X$ and $U\subseteq \mathbb{R}$ open:
$$\begin{array}{rcl}
   S^{-1}[\{g\in \dual{X}:g(x)\in U\}]&=&\{h\in \dual{X}:S(h)(x)\in U\}\\
&=&\{h\in \dual{X}:\tilde{x}(S(h))\in U\}\\
&=&\{h\in \dual{X}:\dual{S}(\tilde{x})(h)\in U\}.
  \end{array}$$
Since $\dual{S}=T$ and $T[X]\subseteq X$, if we put $v=\dual{S}(\tilde{x})\in X$, then we have
$$ S^{-1}[\{g\in \dual{X}:g(x)\in U\}]=\{h\in \dual{X}:h(v)\in U\},$$
which is clearly $w^*$-open.
\end{proof}

\begin{lemma}\label{w*operators-functionals}
Let $X$ be a Banach space.
\begin{enumerate}
 \item[(a)] For every $x\in X$ the functional $F_x:\dual{X}\rightarrow \mathbb{R}$
given by $F_x(f)=f(x)$ is $w^*$-continuous.
\item[(b)]  $T:\dual{X}\rightarrow \dual{X}$ is $w^*$-$w^*$-continuous if, and only if,
for every $x\in X$ the operator $T_x:\dual{X}\rightarrow\mathbb{R}$
given by $T_x(f)=T(f)(x)$ is $w^*$-continuous.
\end{enumerate}
\end{lemma}
%\begin{proof}
%\noindent \textbf{(a)} \hspace{3pt} We shall show that $F_x$ is a dual operator, and so is $w^*$-continuous.
%Consider $G_x:\mathbb{R}\rightarrow X$ given by $G_x(r)=rx$. Then 
%$\dual{G_x}:\dual{X}\rightarrow \dual{\mathbb{R}}=\mathbb{R}$ is such that
%$\dual{G_x}(f)(r)=f(G_x(r))=f(rx)=f(x)r$, for every $f\in \dual{X}$ and every $r\in\mathbb{R}$.
 %Since a functional on $\mathbb{R}$ 
%is a linear transformation, we have that $\dual{G_x}(f)=f(x)=F_x(f)$, for every
%$f\in \dual{X}$.\\

%\noindent \textbf{(b)} \hspace{3pt}Assume $T:\dual{X}\rightarrow \dual{X}$ is $w^*$-$w^*$-continuous.
%By part (a) and since $T_x=F_x\circ T$, we have that $T_x$ is 
%$w^*$-continuous for every $x\in X$.

%Conversely, suppose $T_x$ is $w^*$-continuous for every $x\in X$.
%Consider the preimage of the $w^*$-open subbasic $\{f\in \dual{X}:f(x)\in U\}$,
%for given $x\in X$ and $U\subseteq \mathbb{R}$ open:

% $$T^{-1}[\{f\in \dual{X}:f(x)\in U\}]=\{g\in \dual{X}:T(g)(x)\in U\}$$
%$$=\{g\in \dual{X}:T_x(g)\in U\}
%=T_x^{-1}[U]$$
%This is $w^*$-open by hypothesis.
%\end{proof}

\begin{proposition}\label{w*-product-top}
In both $\ell_1$ and $\ell_\infty$ the product topology is coarser than the 
$w^*$-topology.
 Moreover, in both cases the converse is only true when restricted to a bounded subspace.
\end{proposition}
 
\begin{proof}
 We prove the statement for $\ell_1$, the argument being the same for $\ell_{\infty}$.

Let $k$ be a positive integer and $I_i$ be an open interval for every $i<k$.
We will show that the following basic open of the product topology
$U=\{(b_j)_{j\in\N}\in\ell_1:b_i\in I_i, \forall i<k\}$ is also $w^*$-open.
So let $(a_j)_{j\in\N}\in U$ and take $\varepsilon>0$ such that $(a_i-\varepsilon,a_i+\varepsilon)\subseteq I_i$,
for each $i<k$. %Let $e_j\in c_0$ be the unit vector which has only $0$s
% except at the $j$-th coordinate, where it has a $1$.
Consider the following $w^*$-open set
$$\begin{array}{rcl}
   O &=&  \{(b_j)_{j\in\N}\in\ell_1:|\sum_{j\in\mathbb{N}}(b_j-a_j)\chi_{\{i\}}|<\varepsilon/2, \forall i<k\}\\
& =&  \{(b_j)_{j\in\N}\in\ell_1:|b_i-a_i|<\varepsilon/2, \forall i<k\}. 
  \end{array}$$
Clearly, $(a_j)_{j\in\N}\in O$ and $O\subseteq U$. So 
the product topology is coarser than the $w^*$-topology.
%$(\ell_1,\tau_p)\subseteq(\ell_1,\tau_{w^*})$.

Now fix $M\in\mathbb{R}$ and let $Id:(B_{\ell_1}(M),\tau_{w^*})\rightarrow (B_{\ell_1}(M),\tau_p)$
be the identity map, where $B_{\ell_1}(M)=\{a\in\ell_1:\|a\|\leq M\}$
and $\tau_{w^*}$, $\tau_p$ are the weak$^*$ topology and the product topology,
respectively. 
By the above, this is a continuous function, and since
$B_{\ell_1}$ is $w^*$-compact, we know that it is actually a homeomorphism.
So $\tau_{w^*}$ and $\tau_p$ coincide on every bounded set.

However, this is not the case everywhere. Indeed, fix any $(a_j)_{j\in\N}\in\ell_1$ and $\varepsilon>0$,
and let $f\in c_0$ be such that it is not eventually zero. We show that the $w^*$-open set
$O=\{(b_j)_{j\in\N}\in\ell_1:|\sum_{j\in\mathbb{N}}(b_j-a_j)f(j)|<\varepsilon\}$
is not open in the product topology.
For every positive integer $k$ and every $\delta>0$, let 
$A_k^{\delta}=\{(b_j)_{j\in\N}\in\ell_1:|b_j-a_j|<\delta, \forall i<k\}$. 
Note that the family $\{A_k^{\delta}:k\in\mathbb{N}\setminus\{0\}, \delta>0\}$
is a local basis for $(a_i)_{j\in\N}$ in the product topology. Hence, it is enough to 
show that $A_k^{\delta}\nsubseteq O$, for every $k$ and every $\delta$.
So fix $k\in\mathbb{N}\setminus\{0\}$ and $\delta>0$. Take $m>k$ such that
$f(m)\neq 0$ and define $(b_j)_{j\in\N}\in\ell_1$ by putting
 $b_j=a_j$ for all $j\neq m$ and choosing $b_m$  such that
$b_m\geq \frac{\varepsilon}{|f(m)|}+a_m$.
Then, $\varepsilon\leq (b_m-a_m)|f(m)|=|\sum_{j\in\mathbb{N}}(b_j-a_j)f(j)|$.
Therefore, $(b_j)_{j\in\N}\in A_k^{\delta}\setminus O$.
\end{proof}

\begin{theorem}\label{summary_matrices_appendix}
 Let $R:\ell_{\infty}\rightarrow\ell_{\infty}$ be a linear bounded operator.
 The following are equivalent:
\begin{enumerate}
 \item $R=\ddual{(R|c_0)}$.
\item $R$ is given by a $c_0$-matrix.
\item $R$ is $w^*$-$w^*$-continuous and $R[c_0]\subseteq c_0$.
\item $R| B_{\ell_{\infty}}:(B_{\ell_{\infty}},\tau_p)\rightarrow (\ell_{\infty},\tau_p)$
is continuous and $R[c_0]\subseteq c_0$.
\end{enumerate}
\end{theorem}

\begin{proof}

\noindent  \textbf{(1) $\Leftrightarrow$ (2)} \hspace{3pt}
% Assume $T=\ddual{R}$, for some bounded linear operator $R:c_0\rightarrow c_0$.
% Let $M$ be the matrix corresponding to $R$.
% By propositions \ref{T-dual-transpose} and \ref{S-dual-transpose}  we 
% have that $T$ is iduced by $M$, which is a $c_0$-matrix.
% 
% Conversely, assume $T$ is given by a $c_0$-matrix $M$. Let $R:c_0\rightarrow c_0$
% be the linear bounded operator induced by $M$. Then, by propositions 
%  \ref{T-dual-transpose} and \ref{S-dual-transpose} we have that $\ddual{R}$
% is also induced by $M$. In other words, $T=\ddual{R}$.\\
See Proposition \ref{doubleadjoint_appendix}\\

\noindent \textbf{(1) $\Leftrightarrow$ (3)} \hspace{3pt}
See Proposition \ref{double-dual}.\\

\noindent \textbf{(3) $\Leftrightarrow$ (4)} \hspace{3pt}
Suppose $R$ is $w^*$-$w^*$-continuous and $R[c_0]\subseteq c_0$.
Fix a set $U\subseteq\ell_{\infty}$ open in the product topology.
By Proposition \ref{w*-product-top} and by the $w^*$-$w^*$-continuity of $R$, 
we know that $R^{-1}[U]\cap B_{\ell_{\infty}}$ is open in the product topology. 

Conversely, assume $R$ restricted to the unit ball is continuous in the product topology 
and $R[c_0]\subseteq c_0$. Then, by Proposition \ref{w*-product-top} and
since $R$ is bounded, we have that $R$ restricted to the unit ball is 
$w^*$-$w^*$-continuous. Now consider for each $a\in\ell_1$ the functional 
$R_a:\ell_{\infty}\rightarrow \mathbb{R}$ defined by $R_a(x)=R(x)(a)$. 
By Lemma \ref{w*operators-functionals}(a), we have that 
$R_a| B_{\ell_{\infty}}=F_a\circ R| B_{\ell_{\infty}}$
is $w^*$-$w^*$-continuous. Then, by Corollary 4.46 in \cite{fabian}
we have that $R_a$ is $w^*$-continuous, and since this is true for
every $a\in\ell_1$, we know by Lemma \ref{w*operators-functionals}(b)
that $R$ is $w^*$-$w^*$-continuous.
\end{proof}

\bibliographystyle{amsplain}

\end{document}